\documentclass[a4paper,10pt]{amsart}

\usepackage[matrix,arrow,tips,curve]{xy}
\usepackage[english]{babel}
\usepackage[leqno]{amsmath}
\usepackage{amssymb,amsthm}
\usepackage{amscd}
\usepackage{enumerate}

\usepackage{layout}
\usepackage{fullpage}

\input{xy}
\xyoption{all}

\newtheorem{thm}{Theorem}[section]
\newtheorem{lemma}[thm]{Lemma}

\newtheorem{prop}[thm]{Proposition}

\newtheorem{cor}[thm]{Corollary}

\theoremstyle{remark}
\newtheorem{remark}[thm]{Remark}
\theoremstyle{definition}
\newtheorem{defi}[thm]{Definition}
\newtheorem{convention}[thm]{}
\newtheorem{lemmadefi}[thm]{Lemma - Definition}

\numberwithin{equation}{section}

\newenvironment{sis}{\left\{\begin{aligned}}{\end{aligned}\right.}

\newtheorem{example}[thm]{Example}

\newcommand{\wt}{\widetilde}

\newcommand{\ov}{\overline}
\newcommand{\un}{\underline}

\newcommand{\Aut}{\operatorname{Aut}}

\newcommand{\End}{\operatorname{End}}

\newcommand{\id}{\operatorname{id}}

\newcommand{\Stab}{\operatorname{Stab}}
\newcommand{\Fix}{\operatorname{Fix}}
\newcommand{\Lie}{\operatorname{Lie}}
\newcommand{\ad}{\operatorname{ad}}
\newcommand{\Ad}{\operatorname{Ad}}
\newcommand{\tr}{\operatorname{tr}}
\newcommand{\Hol}{\operatorname{Hol}}
\newcommand{\Tr}{\operatorname{Tr}}

\newcommand{\Gras}{\operatorname{Gr}}

\renewcommand{\Re}{\operatorname{Re}}
\renewcommand{\Im}{\operatorname{Im}}

\newcommand{\calD}{{ \mathcal D}}

\newcommand{\calF}{{ \mathcal F}}

\newcommand{\calH}{{ \mathcal H}}

\newcommand{\calO}{{ \mathcal O}}

\newcommand{\calQ}{{ \mathcal Q}}

\newcommand{\bbA}{{\mathbb A}}

\newcommand{\bbC}{{\mathbb C}}

\newcommand{\bbG}{{\mathbb G}}
\newcommand{\bbH}{{\mathbb H}}

\newcommand{\bbN}{{\mathbb N}}
\newcommand{\bbO}{{\mathbb O}}
\newcommand{\bbP}{{\mathbb P}}

\newcommand{\bbR}{{\mathbb R}}
\newcommand{\bbS}{{\mathbb S}}

\newcommand{\bbZ}{{\mathbb Z}}

\newcommand{\p}{\mathfrak p}
\renewcommand{\k}{\mathfrak k}

\newcommand{\g}{\mathfrak{g}}
\newcommand{\gS}{\mathfrak{S}}
\renewcommand{\u}{\mathfrak{u}}

\newcommand{\Gm}{\operatorname{\bbG_m}}
\newcommand{\GL}{\operatorname{GL}}
\newcommand{\GS}{\operatorname{S}}

\newcommand{\SL}{\operatorname{SL}}
\newcommand{\PSL}{\operatorname{PSL}}

\renewcommand{\O}{\operatorname{O}}
\newcommand{\SO}{\operatorname{SO}}
\newcommand{\PSO}{\operatorname{PSO}}
\newcommand{\SOnc}{\operatorname{SO}^{\rm nc}}
\newcommand{\PSOnc}{\operatorname{PSO}^{\rm nc}}
\newcommand{\SOc}{\operatorname{SO}^{\rm c}}
\newcommand{\PSOc}{\operatorname{PSO}^{\rm c}}

\newcommand{\Sp}{\operatorname{Sp}}
\newcommand{\PSp}{\operatorname{PSp}}
\newcommand{\Spnc}{\operatorname{Sp}^{\rm nc}}
\newcommand{\PSpnc}{\operatorname{PSp}^{\rm nc}}

\newcommand{\U}{\operatorname{U}}
\newcommand{\SU}{\operatorname{SU}}
\newcommand{\PSU}{\operatorname{PSU}}

\newcommand{\G}{\mathbb{G}}

\newcommand{\gl}{\operatorname{\mathfrak{gl}}}
\renewcommand{\sl}{\operatorname{\mathfrak{sl}}}
\newcommand{\so}{\operatorname{\mathfrak{so}}}
\renewcommand{\sp}{\operatorname{\mathfrak{sp}}}
\newcommand{\su}{\operatorname{\mathfrak{su}}}

\newcommand{\sonc}{\operatorname{\mathfrak{so}^{\rm nc}}}
\newcommand{\soc}{\operatorname{\mathfrak{so}^{\rm c}}}

\newcommand{\spnc}{\operatorname{\mathfrak{sp}^{\rm nc}}}

\newcommand{\Herm}{\operatorname{Herm}}

\renewcommand{\SS}{\mathcal{S}}

\renewcommand{\P}{\mathcal{P}}

\begin{document}

\title[Hermitian symmetric manifolds]{A tour on Hermitian symmetric manifolds}

\author{Filippo Viviani}
\address{ Dipartimento di Matematica,
Universit\`a Roma Tre,
Largo S. Leonardo Murialdo 1,
00146 Roma (Italy)}
\email{filippo.viviani@gmail.com}

\begin{abstract}
 Hermitian symmetric manifolds  are Hermitian manifolds which are homogeneous and such that every point has a  symmetry preserving the Hermitian structure. The aim of these notes is to present an introduction to this important class of manifolds, trying to survey the several different perspectives from which Hermitian symmetric manifolds can be studied.
\end{abstract}

\maketitle

\tableofcontents


\section{Introduction}

 A Hermitian symmetric manifold (or \textbf{HSM} for short) is a Hermitian manifold which is homogeneous and such that every point has a symmetry preserving the Hermitian structure. First studied by \'Elie Cartan \cite{Car2}, they are the specialization  of the notion of Riemannian symmetric manifolds (introduced by \'Elie Cartan himself in \cite{Car1}) to complex manifolds.

HSMs  (or more generally Riemannian symmetric manifolds) arise in a wide variety of mathematical contexts: representation theory, harmonic analysis, automorphic forms, complex analysis, differential geometry, algebra (Lie theory and Jordan theory), number theory and algebraic geometry. For example, in algebraic geometry, HSMs arise often as (orbifold) fundamental covers  of moduli spaces, such as the moduli space of polarized abelian varieties (possibly with level structures or with fixed endomorphism algebras), the moduli space of polarized K3 surfaces, the moduli space of polarized irreducible symplectic manifolds, etc..

Due to their frequent occurrence in different areas of mathematics, there is a vast literature on HSMs (e.g. \cite{AMRT}, \cite{Bor}, \cite{BJ}, \cite{FKKLR}, \cite{Hel}, \cite{Koe},  \cite{Loo1}, \cite{Loo2}, \cite{Loo4}, \cite{Mok},  \cite{PS}, \cite{Sat}, \cite{Wol2}) dealing with the various aspects of the theory. This vast literature, however, makes it difficult for a non-expert to have a global overview on the subject. The aim of these notes is to survey the different points of view on HSMs, so that a beginner can orient himself inside the vast literature. For this reason, we have chosen to give very few proofs of the results presented, referring the reader to the relevant literature for complete proofs.

\vspace{0.2cm}

Let us now examine more in detail the contents of the paper. In studying HSMs, the reader should keep in mind  the (well known) classification of HSMs  of (complex) dimension one.
Namely, a HSM of complex dimension one is isomorphic to one of the following HSMs:
\begin{enumerate}
\item The complex manifold $\bbC/\Lambda$, where $\Lambda\subset \bbC$ is a discrete additive subgroup, endowed
 with the Hermitian structure induced by the standard Euclidean metric $g=dxdx+dydy$ on $\bbC=\bbR^2_{(x,y)}$
(which has constant zero curvature).
The translations $z\mapsto z+a$ (with $a\in \bbC$) act transitively  via holomorphic isometries
and the inversion symmetry at $[0]\in \bbC/\Lambda$ is given by $s_{[0]}: z\mapsto -z$.

\item The upper half space $\calH:=\{z=x+iy\in \bbC\: : \: \Im z=y>0\}$ endowed  with a Hermitian structure induced by the hyperbolic metric
$g=\frac{dxdy}{y^2}$ (which has constant negative curvature). The group $\SL_2(\bbR)$ acts transitively via M\"obius transformations (which are holomorphic isometries)
$$\begin{pmatrix}
a &b  \\
c & d
\end{pmatrix}\cdot z :=\frac{az+b}{cz+d} $$
and the inversion symmetry at $i\in \calH$ is given by $\displaystyle s_i: z\mapsto -\frac{1}{z}$.

\item The complex projective line $\bbP^1_{\bbC}$ with the Fubini-Studi Hermitian metric (with constant positive curvature), which is induced by pulling back the Euclidean metric on the two dimensional sphere $\bbS^2\subset \bbR^3$ via the diffeomorphism $\bbP^1_{\bbC}\cong \bbS^2$ induced  via stereographic projection from the north pole $N=(1,0,0)\in \bbS^2$.
The group $\SO_3(\bbR)$ acts transitively on $\bbS^2$ via rotations (which are holomorphic isometries of $\bbP^1_{\bbC}\cong \bbS^2$)
and the inversion symmetry at the north pole $N$ is given by the rotation $s_N:(x,y,z)\mapsto (x,-y,-z)$.

\end{enumerate}
Note that, according to the Riemann's uniformization theorem, the unique simply connected complex manifolds of dimension one are $\bbC$, $\calH$ and $\bbP^1_{\bbC}$:

The above trichotomy in dimension one extends to arbitrary dimensions (see the \emph{Decomposition Theorem} \ref{T:decom}): any HSM can be written uniquely as the product of a HSM of the form $\bbC^n/\Lambda$ for some discrete additive subgroup $\Lambda\subset \bbC^n$ (which is called the Euclidean factor of the HSM), of  a HSM of non-compact type (i.e. a product of irreducible non-compact HSMs) and a HSM of compact type (i.e. a product of irreducible compact HSMs).

The rest of Section \ref{S:Herm} is devoted to the study of non-Euclidean HSM, i.e. those for which the Euclidean factor  in the above decomposition is trivial.

Hermitian symmetric manifolds of compact or of non-compact type admit another natural incarnation.
Namely, HSMs of compact type are  exactly the \emph{cominuscle rational homogeneous projective varieties}, i.e. those varieties
isomorphic to a quotient of the form $G/P$, where $G$ is a semisimple complex Lie group and $P\subset G$ is a parabolic subgroup whose unipotent radical is abelian.
We review this description in \S\ref{SS:cominu}.

HSMs of non-compact type admit a canonical embedding (the so called Harish-Chandra embedding) inside a complex vector space in such a way that they become
\emph{bounded symmetric domains}. And, conversely, any bounded symmetric domain becomes a HSM of non-compact type when it is endowed with the Bergman metric.
We review this correspondence between HSMs of non-compact type and bounded symmetric domains in \S\ref{SS:domain}.

There is a natural correspondence between HSMs of non-compact type and HSMs of compact type, which we review in \S\ref{SS:dual}. Moreover, this correspondence satisfies the property that each HSM of non-compact type is canonically realized (via the so called \emph{Borel embedding}) as an open subset inside the associated HSM of compact type (which is called its compact dual). We review the Borel embedding in \S\ref{SS:HC-B}.

Irreducible HSMs of compact or non-compact type can be classified using \emph{Lie theory}. Indeed, they are diffeomorphic to a quotient of the form $G/K$ where $G$ is a simple Lie group
(compact in the compact type case and non-compact in the non-compact type case) and $K\subset G$ is a maximal compact proper subgroup whose center is equal to
$\bbS^1$. We review this description in \S\ref{SS:HSMLiegrp}.

By passing to the associated Lie algebras, we get a correspondence between non-Euclidean HSMs and irreducible \emph{Hermitian symmetric Lie algebras}, which is the datum of a simple real Lie algebra $\g$ together with an involution $\theta:\g \to \g$ such that its $+1$ eigenvalue has a one-dimensional center. See \S\ref{SS:HSMLiealg} for more details.

There is an alternative approach to the study of HSMs of non-compact type based on Jordan theory rather than Lie theory. Indeed, there is a natural bijection between HSMs of non-compact type and \emph{Hermitian Jordan triple systems}; see \S\ref{SS:HSM-JTS} for more details.

Using either Lie theory or Jordan theory, it is possible to give a classification of irreducible HSMs of non-compact type (and of their compact duals).
They are divided into four infinite families, called (following Siegel's notation) $I_{p,q}$, $II_n$, $III_n$ and $IV_n$, and two exceptional cases, called $V$ and $VI$. Section \S\ref{S:irrHSM} is devoted to a detailed analysis of each of the above mentioned irreducible HSMs. In particular, we make explicit, in each of the cases, the general properties of HSMs presented in Section \S\ref{S:Herm}.

Section \S\ref{S:bound} is devoted to the study of the boundary components of HSMs of non-compact type. More precisely, fix a HSM of non-compact type and realize it as a bounded symmetric domain $D\subset \bbC^N$ via its Harish-Chandra embedding. The closure $\ov D$ of $D$ inside $\bbC^N$ can be partitioned into several equivalence classes for the equivalence relation of being connected through a chain of holomorphic disks. Each of these equivalence classes, called  \emph{boundary components} of $D$, is indeed again a HSM of non-compact type which is realized as a bounded symmetric domain inside its linear span in $\bbC^N$.

Boundary components can be classified via their normalizer subgroups, which turn out to be all the maximal parabolic subgroups of the group $G$ of automorphisms of $D$ (see Theorem \ref{T:clas-norm}). The structure of the normalizer subgroups of the boundary components is analyzed in detail in \S\ref{SS:norma}.

In \S\ref{SS:decom}, we show that, for every boundary component $F$ of $D$, the domain $D$ can be decomposed into the product of $F$, a real vector space $W(F)$ and a \emph{symmetric cone} $C(F)$ associated to $F$, i.e. an open homogeneous cone inside a real vector space which is self-dual with respect to a suitable scalar product.

In \S\ref{SS:symcon}, we show how symmetric cones correspond bijectively to \emph{Euclidean Jordan algebras} and we present the classification of irreducible symmetric cones via the classification of simple Euclidean Jordan algebras.

In \S\ref{SS:Siegel}, we show how bounded symmetric domains can be realized in a unique way as \emph{Siegel domains} (of the second type) associated to  a suitable symmetric cone and to a suitable representation of the associated Euclidean Jordan algebra.

In \S\ref{SS:boun-irr}, we describe explicitly the boundary components of each of the irreducible bounded symmetric domains by computing their normalizer subgroups and their associated symmetric cones.

\vspace{0.2cm}

These notes were written for a PhD course (held at the University of Roma Tre in Spring 2013) entitled ``Toroidal compactifications of locally symmetric varieties"  and a course in the Summer School  ``Combinatorial Algebraic Geometry" (held in Levico Terme in June 2013) with the same title. Thus, our original motivation was to write a survey on the construction of toroidal compactifications of locally symmetric varieties (i.e. quotients of Hermitian symmetric manifolds of non-compact type by arithmetic subgroups), by revisiting the original work of Ash-Mumford-Rapoport-Tai \cite{AMRT}. Due to limitations in space and time, we were unable to complete this project  and we ended up with an attempt to write a survey on the beautiful and rich theory of Hermitian symmetric manifolds. We plan to write a sequel to these notes on the construction of toroidal compactifications of locally symmetric varieties.

\subsection*{Notations}\label{S:notations}

\begin{convention}
 Given a Lie group $G$, we denote by $G^o$ the connected component of $G$ containing the identity and by $Z_G$ its center.
A semisimple Lie group $G$ is said to be adjoint if it has trivial center, or in symbols if $Z_G=\{e\}$.
\end{convention}

\begin{convention}
Given a Lie algebra $\g$, we denote its center by $Z(\g)$.
For a real Lie algebra $\g$, we denote by $\g_{\bbC}:=\g\otimes_{\bbR}\bbC=\g\oplus i\g$ its complexification.
\end{convention}

\begin{convention}\label{N:definite}
Given a real (resp. complex) finite-dimensional vector space $V$ and two symmetric (resp. Hermitian) linear operators $F,G\in \End(V)$, we write:
\begin{enumerate}[(i)]
\item $F>G$ (or $G<F$) if and only if $F-G$ is positive definite;
\item $F\geq G$ (or $G\leq F$) if and only if $F-G$ is positive semidefinite.
\end{enumerate}
\end{convention}

\begin{convention}
Given a matrix $M\in M_{n,n}(F)$ with entries in $F=\bbR, \bbC, \bbH$, we will denote by $M^t$ its transpose and by $\ov{M}$ its conjugate with respect to:
\begin{itemize}
\item the trivial conjugation if $F=\bbR$;
\item the conjugation $x_0+ix_1\mapsto x_0-ix_1$ if $F=\bbC$;
\item  the conjugation $x_0+ix_1+jx_2+kx_3\mapsto x_0-i x_1-jx_2-k x_3$ if $F=\bbH$.
\end{itemize}
Moreover, we set $M^*=\ov{M}^t$.
\end{convention}

\begin{convention}
We will denote by $0$ the zero matrix of any size, by $I_n$ the $n\times n$ identity matrix, by $J_n$ the $2n\times 2n$ standard symplectic matrix, i.e.
$J_n:=\begin{pmatrix} 0 & I_n \\ -I_n & 0\end{pmatrix}$, and we set $S_n:=\begin{pmatrix} 0 & I_n \\ I_n & 0\end{pmatrix}$.
\end{convention}

\begin{convention}
For the notation on simple real Lie groups, we will follow \cite[Chap. X, \S2]{Hel} (see also \cite[Chap. I, \S17]{Kna}).  
\end{convention}

\section{Hermitian symmetric manifolds}\label{S:Herm}

The aim of this section is to introduce Hermitian symmetric manifolds and to establish their basic properties.

Let us begin by recalling the definition of a complex structure and of an almost complex structure on a differentiable manifold $M$.

\begin{defi}\label{D:complex}
\noindent
\begin{enumerate}[(i)]
\item \label{D:complex1}
A \emph{complex manifold} is a pair $(M, \calO_M)$ consisting of a (connected) differentiable manifold $M$ and a sheaf $\calO_M$ of $\bbC$-valued smooth functions on $M$ such that $(M, \calO_M)$ is locally isomorphic to $(\bbC^N, \calO_{\bbC^N})$, where $\calO_{\bbC^N}$ is the sheaf of holomorphic functions on $\bbC^N$. The sheaf $\calO_M$ is said to be a complex structure on the manifold $M$.
\item \label{D:complex2}
A \emph{quasi-complex manifold} is a pair $(M, J)$ consisting of a (connected) differentiable manifold $M$ and a smooth tensor field $J$ of type $(1,1)$ such that for every $p\in M$ the induced linear map
$J_p:T_pM\to T_pM$ satisfies $J_p^2=-{\rm id}$, i.e. $J_p$ is a complex structure on the vector space $T_pM$. The tensor field $J$ is said to be a quasi-complex structure on the manifold $M$.
\end{enumerate}
\end{defi}

Given a complex manifold $(M,\calO_M)$, the local isomorphism of $(M,\calO_M)$ with $(\bbC^n,\calO_{\bbC^n})$ together with the natural complex structure on each tangent space $T_q\bbC^n\cong \bbC^n$
given by multiplication by $i$, induces a quasi-complex structure $J$ on $M$. A quasi-complex structure $J$ on $M$ induced by a complex structure is said to be \emph{integrable}.
Integrable quasi-complex structures are characterized by the following well-known theorem of Newlander-Nirenberg (see \cite[Chap. VIII, Thm. 1.2]{Hel} and the references therein).

\begin{thm}[Newlander-Nirenberg] \label{T:NN}
A quasi-complex structure $J$ on $M$ is induced by a complex structure on $M$ (i.e. it is integrable) if and only if
$$[JX, JY]=J[JX, Y]+J[X,JY]+[X,Y],$$
for any two vector fields $X$ and $Y$ on $M$. In this case, the complex structure on $M$ is uniquely determined by the almost complex structure $J$ on $M$.
\end{thm}

Therefore giving a complex manifold is equivalent to giving an almost complex manifold $(M,J)$ such that   $J$ is integrable.


\vspace{0.1cm}

There are three equivalent ways of giving an Hermitian structure on a complex manifold $(M,J)$, which we now recall.

\begin{lemmadefi}\label{D:Herm}
Let $(M,J)$ be a complex manifold. A Hermitian structure on $(M,J)$ is the assignment of one of the following equivalent structures:
\begin{enumerate}[(i)]
\item \label{D:Herm1} A smooth tensor field $h$ of type $(0,2)$ such that $h_p:T_p M\times T_pM\to \bbC$ is a positive definite Hermitian form with respect to $J_p$ for any $p\in M$ (called
a \emph{Hermitian metric}), i.e.
\begin{itemize}
\item $h_p(x,y)=\ov{h_p(y,x)}$ for any $x,y\in T_pM$;
\item $h_p(J_p x, y)=ih_p(x,y)$ for any $x,y\in T_pM$;
\item $h_p(x,x)>0$ for any $0\neq x\in T_pM$.
\end{itemize}
\item \label{D:Herm2} A Riemannian metric $g$ such that $g_p:T_pM\times T_pM\to \bbR$ is compatible with  $J_p$ for any $p\in M$, i.e.
 \begin{itemize}
\item $g_p(x,y)=g_p(y,x)$ for any $x,y\in T_pM$;
\item $g_p(J_p x, J_p y)=g_p(x,y)$ for any $x,y\in T_pM$;
\item $g_p(x,x)>0$ for any $0\neq x\in T_pM$.
\end{itemize}
\item \label{D:Herm3} A $2$-form $\omega$ such that $\omega_p:T_pM\times T_pM\to \bbR$ is compatible with $J_p$ and positive definite with respect to $J_p$ for any $p\in M$, i.e.
 \begin{itemize}
\item $\omega_p(x,y)=-\omega_p(y,x)$ for any $x,y\in T_pM$;
\item $\omega_p(J_p x, J_p y)=\omega_p(x,y)$ for any $x,y\in T_pM$;
\item $\omega_p(x,J_px)>0$ for any $0\neq x\in T_pM$.
\end{itemize}
\end{enumerate}
One can pass from one assignment to the other two by means of the following formulas:
$$\begin{aligned}
& g(X,Y)=\Re h(X,Y)= \omega(X,JY),\\
& \omega(X,Y)=- \Im h(X,Y)= g(JX,Y),\\
& h(X,Y)= g(X,Y)-ig(JX, Y)=\omega(X,J Y)-i\omega(X,Y).
\end{aligned}$$
for $X, Y$ any smooth vector fields on $M$.
\end{lemmadefi}
We say that $(M,J,h)$ (resp. $(M,J,g)$, resp. $(M,J,\omega)$) is a Hermitian manifold if $(M,J)$ is a complex structure and $h$ (resp. $g$, resp. $\omega$) defines a Hermitian structure on $(M,J)$.
Sometimes, we will say that $M$ is a Hermitian manifold if there is no need to specify the complex structure and the Hermitian structure.

Using partitions of unity, it is easy to show that every complex manifold  can be endowed with a Hermitian structure.
In what follows, we will be interested in complex manifolds that admits a special Hermitian structure in the following sense.

\begin{defi}\label{D:symHerm}
Let $(M,J,h)$ be a Hermitian manifold. Denote by $\Aut(M,J,h)$ the group of holomorphic isometries, i.e. the group of self-diffeomorphisms $\phi:M\to M$ such that $\phi^*J=J$ and $\phi^*h=h$.
We say that
\begin{enumerate}[(i)]
\item \label{D:symHerm1} $(M,J,h)$ is homogeneous if $\Aut(M,J,h)$ acts transitively on $M$;
\item \label{D:symHerm2} $(M,J,h)$ is symmetric (or \textbf{HSM} for short)
if it is homogeneous and for some $p\in M$ (or, equivalently, for any $p\in M$) there exists $s_p\in \Aut(M,J,h)$ (called a \emph{symmetry}  at $p$)
such that $s_p^2=\id$ and $p$ is an isolated fixed point of $s_p$.
\end{enumerate}
\end{defi}

\begin{remark}\label{R:symHerm}
\noindent
\begin{enumerate}[(i)]
\item If every point $p\in M$ admits a symmetry $s_p$ as above then $(M,J,h)$ is automatically homogeneous (see \cite[Prop. 1.6]{Mil}).
\item The symmetry $s_p$ at $p$ can be characterized as the unique $s_p\in \Aut(M,J,h)$ such that $s_p(p)=p$ and $d s_p=-\id_{T_pM}$.
It follows that $s_p$ is a geodesic symmetry at $p$, i.e. if $\gamma: (-a,a)\to M$ is any geodesic such that $\gamma(0)=p$ then $s_p(\gamma(t))=\gamma(-t)$ for any $-a<t<a$
(see \cite[Prop. 1.11]{Mil}).
\item If $(M,J,h=g-i\omega)$ is a Hermitian symmetric manifold then:
\begin{itemize}
\item $(M,g)$ is a (geodesically) complete, i.e. $M$ is a complete metric space or, equivalently, every geodesic of the Riemannian manifold $(M,g)$ can be defined on the entire real line
(see \cite[Prop. 1.11]{Mil});
\item $(M,J, \omega)$ is K\"ahler, i.e. $\omega$ is a closed $2$-form (see \cite[Chap. VIII, Thm. 4.1]{Hel}).
\end{itemize}
\item \label{R:symHerm4} A Riemannian manifold $(M,g)$ such that the group $\Aut(M,g)$ of isometries acts transitively on $M$ and for some $p\in M$ (or, equivalently, for any $p\in M$) there exists
$s_p\in \Aut(M,g)$
which is a geodesic symmetry at $p$ is called a \emph{Riemannian symmetric manifold}; see \cite{Hel} for an extensive study of Riemannian symmetric manifold.

Note that Hermitian symmetric manifolds are in particular Riemannian symmetric manifolds.
\end{enumerate}
\end{remark}

In dimension one, every Hermitian symmetric manifold is isomorphic to one of the following examples.

\begin{example}\label{E:dim1}
\noindent
\begin{enumerate}

\item \label{E:dim1a}
Let $\Lambda\subset \bbC$ be a discrete additive subgroup (note that $\Lambda$ is isomorphic to $(0)$, $\bbZ$ or  $\bbZ^2$).
The quotient $\bbC/\Lambda$ is a complex manifold which we endow with the Hermitian structure induced by the standard Euclidean metric $g=dxdx+dydy$ on $\bbC=\bbR^2_{(x,y)}$
(which has constant zero curvature).
The translations $z\mapsto z+a$ (with $a\in \bbC$) act transitively  via holomorphic isometries, so that $\bbC/\Lambda$ is a homogeneous Hermitian manifold.
If $o$ denotes the class of $0$ in the quotient $\bbC/\Lambda$, then the map $s_o:z\mapsto -z$ is an isometry at $o$, which shows that $\bbC/\Lambda$ is a Hermitian symmetric manifold.

\item \label{E:dim1b}
Let $\calH:=\{z=x+iy\in \bbC\: : \: \Im z=y>0\}$ be the upper half space. Then $\calH$ inherits from $\bbC$ a complex structure  and we endow it with a Hermitian structure induced by the hyperbolic metric
$g=\frac{dxdy}{y^2}$ (which has constant negative curvature). The group $\SL_2(\bbR)$ acts transitively via M\"obius transformations (which are holomorphic isometries)
$$\begin{pmatrix}
a &b  \\
c & d
\end{pmatrix}\cdot z :=\frac{az+b}{cz+d} ,$$
which shows that $\calH$ is a homogeneous Hermitian manifold. The M\"obius transformation $\displaystyle s_i: z\mapsto -\frac{1}{z}$ is a symmetry at $i\in \calH$, so that $\calH$ is a Hermitian symmetric
manifold.

\item  \label{E:dim1c}
Let $\bbP^1_{\bbC}$ be the complex projective line. Via stereographic projection from the north pole $N=(1,0,0)$, the two dimensional sphere $\bbS^2\subset \bbR^3_{(x,y,z)}$ is diffeomorphic to
$\bbP^1_{\bbC}$.
Via this diffeomorphism, the restriction of the Euclidean metric to $\bbS^2$ induces a metric $g$ on $\bbP^1_{\bbC}$ (with constant positive curvature) which is compatible with its complex structure, i.e.
it induces a Hermitian structure on $\bbP^1_{\bbC}$, which is called the Fubini-Studi metric.
The group $\SO_3(\bbR)$ acts transitively on $\bbS^2$ via rotations, which are holomorphic isometries of $\bbP^1_{\bbC}\cong \bbS^2$;
hence $\bbP^1_{\bbC}$ is a  homogeneous Hermitian manifold. The rotation $s_N:(x,y,z)\mapsto (x,-y,-z)$ is a symmetry at the north pole $N\in \bbS^2\cong \bbP^1_{\bbC}$, which show that
$\bbP^1_{\bbC}$ is a Hermitian symmetric manifold.

\end{enumerate}
\end{example}

The trichotomy of the previous Example \ref{E:dim1} extends to arbitrary dimension.

\begin{defi}\label{D:typeHSM}
Let $M$ be a Hermitian symmetric manifold (HSM).
\begin{enumerate}[(i)]
\item $M$ is said to be of \emph{Euclidean type} if $M$ is isomorphic to $\bbC^n/\Lambda$ for some discrete additive subgroup $\Lambda\subset \bbC^n$, where $\bbC^n/\Lambda$ is endowed with the
complex structure and the Hermitian metric induced by $\bbC^n$.
\item $M$ is said to be \emph{irreducible}  if it is not Euclidean and it cannot be written as the product of two non-trivial HSMs.
\item $M$ is said to be \emph{non-Euclidean}  if it is the product of irreducible HSMs.
\item $M$ is said to be of \emph{compact type} (resp. \emph{non-compact type})  if it is the product of compact (resp. non-compact) irreducible HSMs.
\end{enumerate}
\end{defi}

Hermitian symmetric manifolds of non-compact type are also called \emph{Hermitian symmetric domains}, due to the fact that they are biholomorphic to bounded symmetric domains 
(see \S\ref{SS:domain}). 

\begin{remark}\label{R:curva}
Clearly, Euclidean HSMs have identically zero Riemannian sectional curvature. On the other hand,
HSMs of compact type (resp. of non-compact type) have semipositive (resp. seminegative) Riemannian sectional curvature
(see \cite[Chap. V, Thm. 3.1]{Hel}) and therefore also semipositive (resp. seminegative) holomorphic bisectional curvature (see \cite[Chap. 2, (3.3), Prop. 1]{Mok}).
Moreover, irreducible HSMs of compact type (resp. of non-compact type) have positive (resp. negative) Ricci curvature (see \cite[Chap. 3, (1.3), Prop. 2]{Mok}).
\end{remark}

Every Hermitian symmetric manifold can be decomposed uniquely in the following way (see \cite[Chap. VIII, Prop. 4.4, Thm. 4.6, Prop. 5.5]{Hel}).

\begin{thm}[Decomposition theorem]\label{T:decom}
Every Hermitian symmetric manifold $M$ decomposes uniquely as
$$M=M_0\times M_-\times M_+,$$
where $M_0$ is a Euclidean HSM, $M_-$ is a HSM  of compact type and $M_+$ is a HSM of non-compact type.
Moreover, $M_-$ (resp. $M_+$) is simply connected and it decomposes uniquely as a product of compact (resp. non-compact) irreducible HSMs.
\end{thm}

In particular, any HSM is the product of a Euclidean HSM  and of  a non-Euclidean HSM.
Since Euclidean HSMs are easy to understand (being isomorphic to $\bbC^n/\Lambda$, for some discrete additive subgroup $\Lambda\subset \bbC^n$), from now on we will  focus on non-Euclidean HSMs.

An important invariant of a non-Euclidean HSM is its rank, which we are now going to define following \cite[Chap. V, \S 6]{Hel}. Recall that a submanifold $S$ of a Riemannian manifold $(M,g)$ is called
\emph{totally geodesic} if for every $p\in S$ it holds that all the geodesics of $M$ through $p$ that are tangent to $S$ are contained in  $S$. Moreover, in the case where $(M,g)$ is a Riemannian symmetric manifold (in the sense of Remark \ref{R:symHerm}\eqref{R:symHerm4}), $N$ is totally geodesic if and only if for every $p\in N$ we have that $s_p(N)=N$ (see \cite[Chap. 5, (1.1), Lemma 1.1]{Mok}).
Furthermore, $N$ is said to be \emph{flat} if the restriction of the Riemannian metric $g$ to $N$ has  identically zero curvature tensor.

\begin{defi}\label{D:rank}
Let $M$ be a non-Euclidean HSM. The \textbf{rank}Ê of $M$ is the maximal dimension of a flat totally geodesic submanifold of $M$.
\end{defi}

\subsection{Classifying non-Euclidean HSM  via Lie groups}\label{SS:HSMLiegrp}

The aim of this subsection is to classify non-Euclidean Hermitian symmetric manifolds in terms of Lie groups.

Let $M=(M,J,h)$ be a non-Euclidean Hermitian symmetric manifold and fix a point $o\in M$.
The group $\Aut(M)=\Aut(M,J,h)$ of holomorphic isometries of $M$, endowed with the compact-open topology, becomes a (real) Lie group (see \cite[Chap. VIII, \S4]{Hel}).
We denote by $\Aut(M)^o$ the connected component of $\Aut(M)$ containing the identity and by $\Stab(o)$ the Lie subgroup of $\Aut(M)^o$ consisting of all the elements that fix
$o\in M$. The symmetry $s_o$ at $o$ induces the following involutive automorphism of $\Aut(M)^o$:
\begin{equation*}
\begin{aligned}
\sigma: \Aut(M)^o& \longrightarrow \Aut(M)^o\\
g&\mapsto s_o g s_o. \\
\end{aligned}
\end{equation*}
Denote by $\Fix(\sigma)$ the closed Lie subgroup of $\Aut(M)^o$ consisting of all the elements that are fixed by $\sigma$ and let $\Fix(\sigma)^o$ be the connected component of $\Fix(\sigma)$
containing the origin.

\begin{thm}\label{T:LieHerm}
Notations as above.
\begin{enumerate}[(i)]
\item \label{T:LieHerm1} $\Aut(M)^o$ is a semisimple adjoint (i.e. with trivial center) Lie group  and $\Stab(o)$ is a compact Lie subgroup of $\Aut(M)^o$ such that
$$\Fix(\sigma)^o\subseteq \Stab(o)\subseteq \Fix(\sigma).$$
\item \label{T:LieHerm2}
The map
$$
\begin{aligned}
\Aut(M)^o/\Stab(o) & \longrightarrow M\\
[g]& \mapsto g\cdot o
\end{aligned}
$$
is a $\Aut(M)^o$-equivariant diffeomorphism.
\item \label{T:LieHerm3}
The symmetry $s_o$ is contained in the identity component of the center of $\Stab(o)$.
\end{enumerate}
\end{thm}
\begin{proof}
See \cite[Chap. IV, Thm. 3.3; Chap. VIII, Thm. 4.5]{Hel}
\end{proof}

A pair $(G,K)$ consisting of a connected semisimple Lie group $G$ and a compact subgroup $K$ for which there exists an involutive automorphism $\sigma$ of $G$ such that
$\Fix(\sigma)^o\subseteq K \subseteq \Fix(\sigma)$ is a particular case of a \emph{Riemann symmetric pair}  (see \cite[Chap. IV, \S 3]{Hel}). In particular, Theorem \ref{T:LieHerm}\eqref{T:LieHerm1}
is saying that for any non-Euclidean Hermitian symmetric manifold $M$, the pair $(\Aut(M)^o, \Stab(o))$ is a Riemannian symmetric pair.

Conversely, starting with a Riemannian symmetric pair $(G,K)$, the quotient manifold $G/K$ can always be endowed with a Riemannian metric $g$ such that $(G/K, g)$ is
a Riemannian symmetric space (see \cite[Chap. IV, Prop. 3.4]{Hel}). However, in order to endow $G/K$ with a complex structure $J$ and a Hermitian metric $h$ such that $(G/K,J,h)$ becomes
a Hermitian symmetric manifold, the pair $(G,K)$ must satisfy some extra conditions. Theorem \ref{T:LieHerm}\eqref{T:LieHerm3} says that a necessary condition for a Riemannian symmetric pair $(G,K)$
to come from a Hermitian symmetric manifold is that the center of $K$ is not finite. Indeed, in the irreducible case, this last condition is also sufficient.

\begin{thm}\label{T:Lie-irr}
\noindent
\begin{enumerate}[(i)]
\item \label{T:Lie-irr1} Every irreducible HSM of non-compact type is diffeomorphic to $G/K$ for a unique pair $(G,K)$ such that $G$ is a connected non-compact simple adjoint Lie group
and $K$ is a maximal connected and compact Lie subgroup of $G$ with non discrete center $Z_K$ (or, equivalently, with $Z_K=\bbS^1$).
\item \label{T:Lie-irr2} Every irreducible HSM of compact type is diffeomorphic to $G/K$ for a unique pair $(G,K)$ such that $G$ is a connected compact simple adjoint Lie group and
$K$ is a maximal connected and compact proper Lie subgroup of $G$  with non discrete center $Z_K$ (or, equivalently, with $Z_K=\bbS^1$).
\end{enumerate}
\end{thm}
\begin{proof}
See \cite[Chap. VIII, \S 6]{Hel}.

\end{proof}

Using the Lie-theoretic representation given by Theorem \ref{T:Lie-irr}, \`E. Cartan in \cite{Car2} (based upon his previous work \cite{Car1} in which he classifies Riemannian symmetric manifolds)  was able to classify the irreducible HSMs of non-compact type (rep. of compact type) into $6$ types (see also \cite{Bor} for a nice exposition of the work of Cartan). We list the six types (together with their real dimensions and their rank) in Table \ref{F:LieHerm}, referring to \S\ref{S:irrHSM} for more details on each type and to \S\ref{SS:dual} for an explanation of the duality between HSMs of non-compact type and HSMs of compact type.

\begin{table}[!ht]
\begin{tabular}{|c|c|c|c|c|c|}
\hline
Type  & Group $G$ & Group $G^c$ & Compact subgroup $K$ & $\dim_{\bbR} M$ & Rank $M$\\
\hline \hline
$I_{p,q} \:(p\geq q\geq 1)$ & $\PSU(p,q)$ & $\PSU(p+q)$ & $\ov{\GS(U_p\times U_q)}$ & $2pq$ & $q$ \\
\hline
$II_n \: (n\geq 2)$ & $\cong \PSO^*(2n)$ & $\cong \PSO(2n)$ & $\ov{\U(n)}$ & $n(n-1)$ & $\lfloor \frac{n}{2} \rfloor $\\
\hline
$III_n \: (n\geq 1)$ & $\cong \PSp(n,\bbR)$ & $\PSp(n)$ & $\ov{\U(n)}$ & $n(n+1)$& $n$ \\
\hline
$IV_n \: (2\neq n\geq 1)$ & $\cong \PSO(2,n) $ & $\PSO(2+n)$ & $\SO(2)\times \SO(n)$ & $2n$ & $\min\{2,n\} $\\
\hline
$V$ & $E_{6(-14)}$ &  $E_6^c$  & $ \SO(10)\times \SO(2)$ & $32$ & $2$ \\
\hline
$VI$ & $E_{7(-25)}$ &  $E_7^c$ & $ E_6^c\times \SO(2)$ & $54$ & $3$  \\
\hline
\end{tabular}
\caption{Irreducible HSMs: $G/K$ is of non-compact type and $G^c/K$ is its compact dual.}\label{F:LieHerm}
\end{table}

Using the above presentation of HSMs, it is possible to give a Lie theoretic description of the rank of a HSM of non-compact type (as in Definition \ref{D:rank}).

\begin{prop}\label{P:rank-Lie}
Let $M$ be a HSM of non-compact type and write $M\cong \Aut(M)^o/\Stab(o)=G/K$ as in Theorem \ref{T:LieHerm}.
Then the rank of $M$ is equal to the dimension of any maximal $\bbR$-split torus $T$ contained in $G$.
\end{prop}
\begin{proof}
See \cite[Sec. 8C]{Mor}.
\end{proof}

\subsection{Classifying non-Euclidean HSM  via Lie algebras}\label{SS:HSMLiealg}

The aim of this subsection is to classify non-Euclidean HSM in terms of Lie algebras data.

Let $M=(M,J,h)$ be a non-Euclidean HSM with a fixed base point $o\in M$. We shorten the notation used in \S\ref{SS:HSMLiegrp} by setting $G:=\Aut(M)^o$ and $K:=\Stab(o)$.
Let $\g=\Lie(G)$ be the real Lie algebra of $G$ (which is semisimple by Theorem \ref{T:LieHerm}\eqref{T:LieHerm1}) and let $\theta=d\sigma: \g\to \g$ be the involution of $\g$ given by the differential of the involution $\sigma$ of $G$.
Denote by $\k$ (resp. $\p$) the eigenspace for $\theta$ relative to the eigenvalue $+1$ (resp. $-1$). Since $\theta^2=\id$, we have a direct sum decomposition
\begin{equation}\label{E:decompo}
\g=\k\oplus \p
\end{equation}
in such a way that the Lie bracket $[,]$ of $\g$ satisfies
\begin{equation}\label{E:bracket}
[\k,\k]\subseteq \k, \hspace{0.4cm} [\k,\p]\subseteq \p, \hspace{0.4cm} [\p,\p]\subseteq \k.
\end{equation}
From Theorem \ref{T:LieHerm}\eqref{T:LieHerm1}, it follows that the Lie subalgebra $\k\subset \g$ is equal to  the Lie subalgebra $\Lie(K)\subset \Lie(G)$ of the
Lie subgroup $K\subset G$. In particular, since $K$ is a compact Lie group, it follows that $\k$ is a compactly embedded subalgebra of $\g$ (in the sense of \cite[p. 130]{Hel}).
Therefore, the pair $(\g,\theta)$ above defined is a special case of an \emph{effective orthogonal symmetric Lie algebra} in the sense of \cite[Chap. V, \S 1]{Hel}.

Moreover, the $G$-equivariant diffeomorphism $G/K\stackrel{\cong}{\to} M$ of Theorem \ref{T:LieHerm}\eqref{T:LieHerm1} induces a canonical identification $\p\cong T_o M$, where $T_o M$ is the tangent space of $M$ at $o\in M$. The above identification, together with the complex structure $J$ on $M$, induces a complex structure $J_o$ on $\p$. We extend $J_o$ to a linear operator on $\g$ by setting:
\begin{equation}\label{E:operJ}
J=\begin{cases}
0 & \text{ on } \k, \\
J_o & \text{ on } \p.
\end{cases}
\end{equation}
Using the fact that the complex structure $J$ on $M$ is compatible with the Hermitian metric $h$ on $M$, one can prove the following (see \cite[Chap. II, \S 3]{Sat})

\begin{lemma}\label{L:J-oper}
$J$ is a derivation of $\g$.
\end{lemma}

Since $\g$ is semisimple, each derivation on $\g$ is inner (see \cite[Prop. 1.98]{Kna}) and the center of $\g$ is trivial; therefore, $J=\ad H$ for a unique element $H\in \g$.
Using \eqref{E:operJ}, it follows that $H$ belongs to the center $Z(\k)$ of $\k$.

The properties of the triple $(\g, \theta,H)$ are summarized in the following definition, which is a slight adaptation of \cite[p. 54]{Sat}.

\begin{defi}\label{D:symLie}
A \emph{Hermitian symmetric Lie algebra} (or, for short, a \textbf{Hermitian SLA}) is a triple $(\g,\theta, H)$ consisting a semisimple real Lie algebra $\g$, an involution $\theta:\g \to \g$ whose associated decomposition
$\g=\k\oplus \p$ into eigenvalues for $+1$ and $-1$ is such that $\k$ is a compactly embedded subalgebra of $\g$ and an element $H\in Z(\k)$ such that $\ad(H)^2_{|\p}=-\id_{\p}$.

A Hermitian SLA $(\g,\theta, H)$ is said to be \emph{irreducible}  if the complexification $\g_{\bbC}$ of $\g$ is a simple complex Lie algebra.

A Hermitian SLA $(\g,\theta, H)$ is said to be
\begin{enumerate}[(i)]
\item \emph{of compact type}  if $\g$ is a compact Lie algebra, i.e. its Killing form $B$ is negative definite;
\item \emph{of non-compact type}  if $\g$ does not contain compact simple factors and $\theta$ is a Cartan involution of $\g$, i.e. $B$ is negative definite on $\k$ and positive definite on $\p$.
\end{enumerate}

\end{defi}

We have seen before how to associate to a non-Euclidean HSM $M$ a Hermitian SLA $(\g, \theta,H)$.
Indeed, this construction is bijective and it is compatible with the decomposition of HSMs as in Theorem \ref{T:decom}.

\begin{thm}\label{T:HSMLiealg}
There is a bijection
\begin{equation}\label{E:bijeHSM}
\left\{\text{non-Euclidean HSMs}\right\} \stackrel{\cong}{\longrightarrow} \left\{\text{Hermitian SLAs}\right\}
\end{equation}
obtained by sending a non-Euclidean Hermitian symmetric manifold $M$ to its associated Hermitian SLA $(\g, \theta,H)$.

\begin{enumerate}[(i)]
\item \label{T:HSMLiealg1} If $M=M_1\times \cdots \times M_r$ is the decomposition of $M$ into irreducible HSMs and $(\g_i,\theta, H_i)$ is the Hermitian SLA associated to $M_i$, then
\begin{equation}\label{E:dec-HerLA}
(\g, \theta, H)=(\g_1,\theta_1,H_1)\oplus \cdots \oplus (\g_r,\theta_r, H_r),
\end{equation}
meaning that $\g$ decomposes a direct sum of ideals $\g=\g_1\oplus \cdots \oplus \g_r$,  $\theta$ is the unique involution on $\g$ that preserves the above decomposition and such that
$\theta_{|\g_i}=\theta_i$ and $H=H_1+ \cdots + H_r$.

\item \label{T:HSMLiealg2} $M$ is an irreducible HSM if and only if $(\g, \theta, H)$ is irreducible.

\item \label{T:HSMLiealg3} $M$ is of non-compact type if and only if $(\g,\theta, H)$ is of non-compact type.

\item \label{T:HSMLiealg4} $M$ is of compact type if and only if $(\g,\theta, H)$ is of compact type.

\end{enumerate}
\end{thm}
\begin{proof}
Part \eqref{T:HSMLiealg1} follows from \cite[Prop. 4.4, Prop. 5.5]{Hel}.

According to \cite[Prop. 5.5, Thm. 5.3, Thm. 5.4]{Hel}, $M$ is irreducible if and only if $(\g, \theta)$ belongs to one of the following four types:
\begin{enumerate}
\item \un{Type I}: $\g$ is a simple compact Lie algebra;
\item \un{Type II}: $\g$ is a compact Lie algebra which is the sum of two simple ideals $\g=\g_1\oplus \g_2$ which are interchanged by $\theta$;
\item \un{Type III}: $\g$ is a simple non-compact Lie algebras such that $\g_{\bbC}$ is simple;
\item \un{Type IV}: $\g$ is a simple complex Lie algebra (regarded as a real Lie algebra).
\end{enumerate}
Moreover, the existence of a non-trivial element $H\in Z(\k)$ rules out Type II and Type IV (see \cite[p. 518]{Hel}).
The remaining cases (Type I and Type III) are exactly those cases for which $\g_{\bbC}$ is a simple complex Lie algebra. Part \eqref{T:HSMLiealg2} follows.

Parts \eqref{T:HSMLiealg3} and \eqref{T:HSMLiealg4} follow easily from \eqref{T:HSMLiealg2} and the Definitions \ref{D:typeHSM} and  \ref{D:symLie}.

It remains to prove that the map from non-Euclidean HSM to Hermitian SLAs is bijective. In order to prove that, we will construct the inverse map of \eqref{E:bijeHSM}.

First of all, we have the following

\un{Claim:}  Any Hermitian SLA admits a decomposition (as in \eqref{E:dec-HerLA}) into the direct sum of irreducible Hermitian SLAs.

Indeed, given a Hermitian SLA $(\g,\theta,H)$, using \cite[Prop. 5.2, Thm. 5.3, Thm. 5.4]{Hel}, we can write $\g$ as the direct sum of ideals
$$ \g=\g_1\oplus \cdots \oplus \g_r,$$
in such a way that $\theta$ preserves this decomposition and each factor $(\g_i,\theta_i:=\theta_{|\g_i})$ belongs to one of the four Types mentioned before.
The element $H$ can be written as a sum $H=H_1+\cdots + H_r$  in such a way that $(\g_i, \theta_i, H_i)$ is a Hermitian SLA.
As observed before, the existence of such an element $H_i\in Z(\k_i)$ forces $(\g_i,\theta_i)$ to be of Type I or Type III, so that each $(\g_i,\theta_i,H_i)$ is irreducible, q.e.d.

\vspace{0.1cm}

Using the Claim and part \eqref{T:HSMLiealg2}, it is now enough to construct an inverse of \eqref{E:bijeHSM} for irreducible Hermitian SLAs.

Let $(\g,\theta,H)$ be an irreducible Hermitian SLA and assume first that it is of non-compact type. Let $G$ the unique connected adjoint Lie group
with $\Lie(G)=\g$ and let $K$ be the unique Lie subgroup of $G$ corresponding to the Lie subalgebra $\k\subset \g$. Since $\theta$ is a Cartan involution of $\g$ and $\g$ is simple, we deduce that
$G$ is a simple non-compact Lie group and $K$ is a maximal compact Lie subgroup of $G$. On the quotient manifold $G/K$ (with base point $o=[e]$) , consider the unique $G$-invariant almost complex
structure $J$  such that $J_o$ is equal to $\ad(H)_{|\p}$ via the identification $T_oM\cong \p$ and the unique $G$-invariant Riemannian metric $g$ such that $g_o$ is equal to the Killing form of $\g$ restricted to
$\p$. It follows from \cite[Chap. VIII, Prop. 4.2]{Hel} that $(G/K, J, g)$ is a HSM, which is irreducible of non-compact type by Theorem \ref{T:Lie-irr}. Moreover, it is easy to check that the Hermitian SLA associated to  $(G/K, J, g)$ is the Hermitian SLA $(\g, \theta, H)$ we started with, and we are done.

The case where $(\g, \theta, H)$ is irreducible of compact type can be dealt with similarly and therefore it is left to the reader.

\end{proof}

\begin{remark}\label{R:funHSM-SLA}
The bijection \eqref{E:bijeHSM} becomes an equivalence of categories if the two sets are endowed with the following morphisms:
\begin{enumerate}[(i)]
\item A \emph{symmetric} (or equivariant) morphism between two pointed non-Euclidean HSMs $(M,o)$ and $(M',o')$ is a pointed holomorphic map $\phi:(M,o)\to (M',o')$ such that
$$\phi \circ s_p=s'_{\phi(p)}\circ \phi  \: \text{ for any } p\in M,$$
where $s_p$ is the symmetry of $M$ at $p\in M$ and $s'_{\phi(p)}$ is the symmetry of $M'$ at $\phi(p)\in M'$.
\item A morphism of Hermitian SLAs (also called a $H_1$-morphism) between two Hermitian SLAs $(\g,\theta,H)$ and $(\g',\theta',H')$ is a morphism of Lie algebras $\rho:\g\to \g'$ such that
$$
\begin{aligned}
& \rho\circ \theta = \theta'\circ \rho, \\
& \rho\circ \ad(H)=\ad(H')\circ \rho.\\
\end{aligned}
$$
\end{enumerate}
See \cite[Chap. II, \S8]{Sat} and \cite[Chap. III, \S2.2]{AMRT}.

\end{remark}

The correspondence in Theorem \ref{T:HSMLiealg} together with the classification of irreducible HSMs recalled in \S\ref{SS:HSMLiegrp} gives a classification of irreducible Hermitian SLAs.
We record the list of irreducible classical Hermitian SLAs of non-compact type (resp. of compact type) into Table \ref{F:SLA1} (resp. Table \ref{F:SLA2}). Moreover, to each irreducible Hermitian SLA
$(\g,\theta,H)$ of non-compact type in Table \ref{F:SLA1}, its dual Hermitian SLA of compact type (as defined in \S\ref{SS:dual}) is denoted by $(\g^*,\theta^*,H^*)$ in Table \ref{F:SLA2}.
We refer to \S\ref{S:irrHSM} for more details.

\begin{table}[!ht]
\begin{tabular}{|c|c|c|c|c|c|}
\hline
Type  & Lie algebra $\g$ & Involution $\theta$ & Central element $H\in Z(\k)$ \\

\hline \hline
$I_{p,q} $ & $\mathfrak{su}(p,q)$
& $\theta \begin{pmatrix}  Z_1 & Z_2 \\  \ov{Z_2}^t & Z_3  \end{pmatrix}=\begin{pmatrix}Z_1 & -Z_2 \\ -\ov{Z_2}^t & Z_3\end{pmatrix} $
& $i \begin{pmatrix} \frac{q}{p+q} I_p & 0\\ 0 & \frac{-p}{p+q} I_q\end{pmatrix}$  \\
\hline

$II_n$ & $\cong \mathfrak{so}^*(2n)$ &
 $ \theta \begin{pmatrix}  Z_1 &Z_2 \\ \ov{Z}_2^t & - Z_1^t  \end{pmatrix}= \begin{pmatrix}  Z_1 &-Z_2 \\ -\ov{Z}_2^t & - Z_1^t  \end{pmatrix} $
 & $\frac{i}{2} \begin{pmatrix} I_n & 0\\ 0 & - I_n\end{pmatrix} $  \\
\hline

$III_n $ & $\cong \mathfrak{sp}(n,\bbR)$
& $ \theta \begin{pmatrix}  Z_1 &Z_2 \\ \ov{Z}_2^t & - Z_1^t  \end{pmatrix}= \begin{pmatrix}  Z_1 &-Z_2 \\ -\ov{Z}_2^t & - Z_1^t  \end{pmatrix}$
 & $\frac{i}{2} \begin{pmatrix} I_n & 0\\ 0 & - I_n \end{pmatrix}$  \\
\hline

$IV_n $ & $\cong \mathfrak{so}(2,n) $
& $\theta  \begin{pmatrix} X_1 & iX_2 \\ -iX_2^t & X_3 \end{pmatrix}=\begin{pmatrix} X_1 & - iX_2 \\  i X_2^t & X_3\end{pmatrix} $
 & $ \begin{pmatrix} 0 & 0 \\ 0 & J_1 \end{pmatrix}$  \\
\hline

\end{tabular}
\caption{Irreducible classical Hermitian SLAs of non-compact type.}\label{F:SLA1}
\end{table}

\begin{table}[!ht]
\begin{tabular}{|c|c|c|c|c|c|}
\hline
Type  & Lie algebra $\g^*$ & Involution $\theta^*$ & Central element $H^*\in Z(\k)$ \\
\hline \hline
$I_{p,q} $ & $\mathfrak{su}(p+q)$
& $\theta^* \begin{pmatrix}  Z_1 & Z_2 \\  -\ov{Z_2}^t & Z_3  \end{pmatrix}=\begin{pmatrix}Z_1 & -Z_2 \\ \ov{Z_2}^t & Z_3\end{pmatrix} $
& $i  \begin{pmatrix}\frac{q}{p+q} I_p & 0\\  0 & \frac{-p}{p+q} I_q \end{pmatrix}$  \\
\hline

$II_n$ & $\cong \mathfrak{so}(2n)$
& $ \theta^* \begin{pmatrix}  Z_1 & Z_2 \\ - \ov{Z_2}^t & -Z_1^t  \end{pmatrix}=\begin{pmatrix}Z_1 & -Z_2 \\ \ov{Z_2}^t & - Z_1^t\end{pmatrix} $
 &$ \frac{i}{2} \begin{pmatrix} I_n & 0\\ 0 & - I_n\end{pmatrix} $  \\
\hline

$III_n $ & $\mathfrak{sp}(n)$
& $\theta^* \begin{pmatrix}  Z_1 & Z_2 \\  -\ov{Z_2}^t & -Z_1^t  \end{pmatrix} =\begin{pmatrix}  Z_1 & -Z_2 \\  \ov{Z_2}^t & -Z_1^t  \end{pmatrix} $
 & $\frac{i}{2} \begin{pmatrix} I_n & 0\\ 0 & -I_n \end{pmatrix}$  \\
\hline

$IV_n $ & $\mathfrak{so}(2+n) $
& $\theta^*  \begin{pmatrix} X_1 & X_2 \\- X_2^t & X_3 \end{pmatrix}=\begin{pmatrix} X_1 & -  X_2 \\  X_2^t & X_3\end{pmatrix} $
 & $ \begin{pmatrix} 0 & 0 \\ 0 & J_1 \end{pmatrix}$  \\
\hline

\end{tabular}
\caption{Irreducible classical Hermitian SLAs of compact type.}\label{F:SLA2}
\end{table}

The rank of a non-Euclidean HSM as defined in Definition \ref{D:rank} can be read off from the associated Hermitian SLA as it follows.

\begin{prop}\label{P:rank-SLA}
Let $M$ be a non-Euclidean HSM with base point $o\in M$ and let $(\g, \theta, H)$ its associated  Hermitian SLA. Consider the decomposition $\g=\k\oplus \p$ as in \eqref{E:decompo}.
Then any two maximal abelian Lie subalgebras of $\g$ contained in $\p$ are  conjugate by an element of $\Stab(o)\subset \Aut(M)^o$, acting on $\g$ via the adjoint representation, and their
common dimension is equal to the  the rank of $M$.
\end{prop}
\begin{proof}
See \cite[Chap. V, Prop. 6.1 and Lemma 6.3]{Hel}.
\end{proof}

\subsection{Duality between compact and non-compact HSMs.}\label{SS:dual}

We are now going to define an involution on non-Euclidean HSMs that exchanges compact HSMs with non-compact HSMs.
The involution is defined most easily in terms of Hermitian SLAs, using the bijection of Theorem \ref{T:HSMLiealg}.

Given a Hermitian SLA $(\g,\theta,H)$, define a new Hermitian SLA $(\g,\theta,H)^*=(\g^*,\theta^*,H^*)$ as it follows.
The Lie algebra $\g^*$ is the subalgebra of the complexification $\g_{\bbC}$ given by
$$\g^*:=\k\oplus i\p\subseteq \g_{\bbC}=(\k\oplus \p)\oplus i(\k\oplus \p),$$
where as usual $\k$ and $\p$ are the eigenspaces for $\theta$ relative to the eigenvalues $+1$ and $-1$. Since $\g_{\bbC}=(\g^*)_{\bbC}$ and the property of being semisimple is
preserved by the complexification functor, it follows that $\g^*$ is a semisimple real Lie algebra.  The involution $\theta^*$ on $\g^*$ is defined by
$$\theta^*=
\begin{cases}
+1 & \text{ on } \k, \\
-1 & \text{ on } i\p.
\end{cases}
$$
In other words,  $\g^*=\k\oplus i\p$ is the decomposition of $\g^*$ into eigenspaces for $\theta^*$ relative to the eigenvalues $+1$ and $-1$.
Finally, we set
$$H^*:=H.$$
Note that $H^*=H\in Z(\k)$ and that $\ad(H^*)_{|i\p}=-\id_{i\p}$. Therefore, $(\g,\theta, H)^*$ is a Hermitian SLA, which is called the \emph{dual Hermitian SLA}
of $(\g,\theta,H)$.

\begin{thm}\label{T:dualLA}
The map (called the \emph{duality map})
\begin{equation*}\label{E:dualLA}
\begin{aligned}
\left\{\text{Hermitian SLAs}\right\} & \longrightarrow \left\{\text{Hermitian SLAs}\right\} \\
(\g,\theta,H) & \mapsto (\g, \theta, H)^*:=(\g^*,\theta^*,H^*)
\end{aligned}
\end{equation*}
is an involution which satisfies the following properties:
\begin{enumerate}[(i)]
\item \label{T:dualLA1} If
$$(\g,\theta,H)=(\g_1,\theta_1,H_1)\oplus \cdots \oplus (\g_r,\theta_r, H_r)$$
is the decomposition of $(\g,\theta, H)$ into irreducible Hermitian SLAs, then the dual Hermitian SLA $(\g,\theta,H)^*$
admits the following decomposition into irreducible Hermitian SLAs
$$(\g,\theta,H)^*=(\g_1,\theta_1,H_1)^*\oplus \cdots \oplus (\g_r,\theta_r, H_r)^*.$$
\item \label{T:dualLA2} $(\g,\theta,H)$ is of compact type (resp. of non-compact type) if and only if $(\g,\theta,H)^*$ is of non-compact type (resp. of compact type).
\end{enumerate}
\end{thm}
\begin{proof}
The fact that the duality map is an involution follows immediately from the definition.

Part \eqref{T:dualLA1}  follows from easily from the definitions of the dual and of the direct sum of Hermitian SLAs, together with the observation that $(\g,\theta,H)$ is irreducible if and only if
$(\g,\theta,H)^*$ is irreducible since  $\g_{\bbC}=(\g^*)_{\bbC}$.

Part \eqref{T:dualLA2} follows from the well-know fact  that an involution $\theta$ on a semisimple real Lie algebra $\g$, with associated decomposition $\g=\k\oplus \p$  into $+1$ and $-1$ eigenvalues,
is a Cartan involution of $\g$ if and only if $\g^*=\k\oplus i\p$ is a compact (semisimple) Lie algebra (see e.g. \cite[Chap. III, Prop. 7.4]{Hel}).

\end{proof}

We can now define the dual of a non-Euclidean HSM, using the bijection of Theorem \ref{T:HSMLiealg}.

\begin{defi}\label{D:dualHSM}
Let $M$ be a non-Euclidean HSM whose associated Hermitian SLA is $(\g,\theta,s)$. The \emph{dual HSM}  of $M$, denoted by $M^*$, is the unique
non-Euclidean HSM whose associated Hermitian SLA is $(\g,\theta,H)^*$.
\end{defi}

From Theorem \ref{T:dualLA} and Theorem \ref{T:HSMLiealg}, we can deduce the following

\begin{cor}\label{C:dual-HSM}
The map (called the \emph{duality map})
\begin{equation*}
\begin{aligned}
\left\{\text{non-Euclidean HSMs}\right\} & \longrightarrow \left\{\text{non-Euclidean HSMs}\right\} \\
M & \mapsto M^*
\end{aligned}
\end{equation*}
is an involution which satisfies the following properties:
\begin{enumerate}[(i)]
\item \label{C:dual-HSM1} If
$$M =M_1\times \cdots \times M_r  $$
is the decomposition of $M$ into irreducible HSMs, then  $M^*$
admits the following decomposition into irreducible HSMs
$$M^*=M_1^*\times \cdots \times M_r^*.$$
\item \label{C:dual-HSM2} $M$ is of compact type (resp. of non-compact type) if and only if $M^*$ is of non-compact type (resp. of compact type).
\end{enumerate}
\end{cor}

\subsection{Harish-Chandra and Borel embeddings.}\label{SS:HC-B}

The aim of this subsection is to realize a given HSM of non-compact type as an open subset of a complex vector space (Harish-Chandra embedding) and as an open subset
of its dual HSM (Borel embedding), which  is also called its \emph{compact dual}.

Fix a HSM of non-compact type $M=(M,J,h)$ together with a base point $o\in M$. By Theorem \ref{T:LieHerm}\eqref{T:LieHerm2}, we have a diffeomorphism $M\cong G/K$ where $G=\Aut(M)^o$
and $K=\Stab(o)$.  Since $G$ is an adjoint semisimple Lie group (by Theorem \ref{T:LieHerm}\eqref{T:LieHerm1}), the adjoint representation $\Ad:G\to \GL(\g)$ is faithful.
Therefore $G$ admits a natural complexification $G_{\bbC}$ (in the sense of \cite[p. 437]{Kna}),
namely the complex connected (semisimple) Lie subgroup of $\GL(\g,\bbC)$ whose Lie algebra is the Lie subalgebra $\g_{\bbC}\stackrel{\ad}{\hookrightarrow}\gl(\g,\bbC)$.
Denote by $K_{\bbC}$ the unique connected Lie subgroup of $G_{\bbC}$ whose Lie algebra is
$\k_{\bbC}\subset \g_{\bbC}$. Note that $K_{\bbC}$ is a complex reductive Lie group which is a complexification of $K$.

Denote by $(\g,\theta,H)$ the Hermitian SLA associated to $M$, as in \S\ref{SS:HSMLiealg}.
Since $\ad(H)_{|\p}=-\id_{\p}$, the complexification $\p\otimes_{\bbR}\bbC$ admits a decomposition
$$\p_{\otimes \bbR}\bbC= \p_{+}\oplus \p_{-},$$
where $\p_{+}$ (resp. $\p_{-}$) is the eigenspace for $\ad(H)_{|\p}$ relative to the eigenvalue $+i$ (resp. $-i$). Therefore,  $\g_{\bbC}$ admits the following decomposition
\begin{equation}\label{E:dec-g}
\g_{\bbC}=\k_{\bbC}\oplus \p_+\oplus \p_{-}
\end{equation}
which satisfies the relations (see \cite[p. 53]{Sat}):
\begin{equation*}
[\k_{\bbC},\p_{\pm}]\subset \p_{\pm} \hspace{0.5cm} [\p_+,\p_-]\subset \k_{\bbC} \hspace{0.5cm} [\p_+,\p_+]=0 \hspace{0.5cm} [\p_-,\p_-]=0.
\end{equation*}
In particular,  $\p_+$ and $\p_-$ are abelian subalgebras of $\g_{\bbC}$ which are normalized by $\k_{\bbC}$. Denote by $P_+$ (resp. $P_-$) the connected Lie subgroups of $G_{\bbC}$ whose Lie algebra
is $\p_+\subset \g_{\bbC}$ (resp. $\p_-\subset \g_{\bbC}$). Then $P_+$ and $P_-$ are abelian unipotent Lie groups that are stabilized by $K_{\bbC}$. It turns out that $P_+$ and $P_-$ are simply connected so that  the exponential map
$$\exp: \p_{\pm}\to P_{\pm}$$
is a diffeomorphism (see \cite[Chap. VIII, Lemma 7.8]{Hel}). Moreover the multiplication map $(P_+\times P_-)\rtimes K_{\bbC}\to G_{\bbC}$ is injective and the image contains $G$.

Finally, let $G_c$ be the Lie subgroup of $G_{\bbC}$ corresponding to the Lie subalgebra $\g^*\subset (\g^*)_{\bbC}=\g_{\bbC}$. By Theorem \ref{T:dualLA} and Corollary \ref{C:dual-HSM}, $G_c$ is a compact Lie group containing $K$ such that
\begin{equation*}
M^*\cong G_c/K.
\end{equation*}
We can summarize the above discussion into the following commutative diagram
\begin{equation}\label{E:diag-emb}
\xymatrix{
G_{} \ar@{^{(}->}[r]& (P_+\times P_-)\rtimes K_{\bbC} \ar@{^{(}->}[r] & G_{\bbC} & G_c \ar@{_{(}->}[l]\\
K_{} \ar@{^{(}->}[r] \ar@{^{(}->}[u]&  P_-\rtimes K_{\bbC} \ar@{=}[r]  \ar@{^{(}->}[u]& P_-\rtimes K_{\bbC}   \ar@{^{(}->}[u]& K_{} \ar@{_{(}->}[l]  \ar@{_{(}->}[u]\\
}
\end{equation}

\begin{thm}\label{T:embed}
 By taking quotients  in \eqref{E:diag-emb}, we get a  diagram of complex manifolds
\begin{equation*}
\xymatrix{
M\cong G/K \ar@{^{(}->}[r] \ar@{^{(}->}[ddr]_{i_{HC}}& \left[(P_+\times P_-)\rtimes K_{\bbC}\right]/ \left(P_-\rtimes K_{\bbC}\right)    \ar@{^{(}->}[r]  & G_{\bbC}/\left(P_-\rtimes K_{\bbC}\right) & G_c/K\cong M^* \ar[l]^{\cong}_{\phi}\\
& P_+ \ar@{=}[u]& & \\
& \p_+ \ar[u]^{\cong}_{\exp} \ar@{^{(}->}_j[uur]& & \\
}
\end{equation*}
in which  $\phi$ is a biholomorphism, $i_{HC}$ is a holomorphic open embeddings and $j$ is a Zariski open embedding onto the homogeneous projective variety $G_{\bbC}/\left(P_-\rtimes K_{\bbC}\right)$
\end{thm}
\begin{proof}
See \cite[Chap. VIII, \S 7]{Hel} or \cite[Chap. II, \S4]{Sat}.
\end{proof}
The embedding $i_{HC}$ is called the \emph{Harish-Chandra embedding} while the composition $\phi^{-1}\circ j \circ i_{HC}$ is called the \emph{Borel embedding}.
Harish-Chandra embeddings and Borel embeddings for each irreducible HSM will be studied in detail in \S\ref{S:irrHSM}.

\subsection{HSMs of non-compact type as bounded symmetric domains}\label{SS:domain}

The Harish-Chandra embedding $i_{HC}:M\hookrightarrow \p_+$ defined in \S\ref{SS:HC-B} allows to realize canonically a given HSM of non-compact type $M$ as a bounded symmetric domain.

\begin{defi}\label{D:BSD}
A \textbf{domain}  $D\subseteq \bbC^N$ (i.e. an open connected subset of $\bbC^N$) is said to be
\begin{enumerate}[(i)]
\item  \emph{bounded}  it is bounded as a subset of $\bbC^N$;
\item  \emph{homogeneous} if the group $\Hol(D)$  of biholomorphisms of $D$ acts transitively on $D$;
\item  \emph{symmetric}  if it is homogeneous and for some $p\in D$ (or, equivalently, for any $p\in D$) there exists $s_p\in \Hol(D)$
(called a symmetry  at $p$) such that $s_p^2=\id$ and $p$ is an isolated fixed point of $s_p$.
\item \emph{irreducible} if there does not exist a non-trivial decomposition $\bbC^N=\bbC^{N_1} \times \bbC^{N_2}$ and two domains $D_1\subset \bbC^{N_1}$ and $D_2\subset \bbC^{N_2}$ such that
$D=D_1\times D_2$.
\end{enumerate}
\end{defi}
\begin{remark}
If $D\subset \bbC^N$ is a bounded domain, then we have that (see \cite[Chap. II, \S4, Rmk. 2]{Sat} and the references therein)
\begin{enumerate}[(i)]
\item the group $\Hol(D)$ admits a (unique) structure of Lie group compatible with the open-compact topology;
\item $D$ is symmetric if and only if it is homogeneous and $\Hol(D)$ is semisimple.
\end{enumerate}
Moreover, \'E. Cartan \cite{Car2} showed that, up to dimension three, every homogeneous bounded domain is also symmetric and he asked if this was true in every dimension. Counterexamples were found later, starting from dimension four, by Pyateskii-Shapiro (see \cite{PS}).
\end{remark}

The image of the Harish-Chandra embedding is a bounded symmetric domain inside the complex vector space $\p_+$.

\begin{thm}[Hermann, Harish-Chandra]\label{T:im-HC}
For any HSM of non-compact type $M$ together with a fixed base point $o\in M$,
the image of the Harish-Chandra embedding $i_{HC}:M\hookrightarrow \p_+$ is a bounded symmetric domain.
\end{thm}
\begin{proof}
Since $i_{HC}$ is a holomorphic open embedding and $M$ is connected, the image $i_{HC}(M)$ is a domain inside $\p_+$.

The group $\Aut(M)$ of biholomorphic isometries of $M$ is a subgroup of finite index of the group of biholomorphisms $\Hol(M)=\Hol(i_{HC}(M))$ (see \cite[Prop. 1.6]{Mil}).
Therefore, the group $\Aut(M)^o=\Hol(i_{HC}(M))^o$ acts transitively on $i_{HC}(M)$ by Theorem \ref{T:LieHerm}\eqref{T:LieHerm2}; hence $i_{HC}(M)$ is a homogeneous domain.
Moreover, the fact that every point of $M$ has a symmetry in the sense of Definition \ref{D:symHerm}\eqref{D:symHerm2}) implies that  every point of $i_{HC}(M)$ has a symmetry in the sense
of Definition \ref{D:BSD}; hence $i_{HC}(M)$ is a symmetric domain.

Finally, in order to prove that $i_{HC}(M)$ is a bounded domain, we need to recall an explicit description of the image of $i_{HC}$.
Consider the  Hermitian SLA of non-compact type $(\g,\theta,H)$ associated to $M$ as in \S\ref{SS:HSMLiealg}
and the induced decomposition $\g_{\bbC}=\k_{\bbC}\oplus \p_+\oplus \p_-$ as in \eqref{E:dec-g}. Denote by $\tau$ the complex conjugation on $\g_{\bbC}$ corresponding to the real form $\g^*$
of $\g_{\bbC}$ introduced in \S\ref{SS:dual}. Since $(\g, \theta,H)$ is of non-compact type, the algebra $\g^*$ is compact by Theorem \ref{T:dualLA}\eqref{T:dualLA2}. Therefore, $\tau$ is a Cartan involution of
$\g_{\bbC}$ (see \cite[Prop. 6.14]{Kna}), which implies that (see \cite[Chap. VI, \S2]{Kna})
\begin{equation*}
\begin{aligned}
B_{\tau}:\g_{\bbC}\times \g_{\bbC} & \longrightarrow \bbC\\
(X,Y) & \mapsto B_{\tau}(X,Y):=-B(X,\tau Y)
\end{aligned}
\end{equation*}
is a positive definite Hermitian form, where $B$ denotes as usual the Killing form of $\g_{\bbC}$.
For any $X\in \p_+$, define the linear operator
\begin{equation*}
\begin{aligned}
T(X): \p_- & \longrightarrow \k_{\bbC} \\
Y & \mapsto [Y, X]
\end{aligned}
\end{equation*}
and denote by $T(X)^*:\k_{\bbC}\to \p_-$ the adjoint of $T(X)$ with respect to $B_{\tau}$.
With these notations, the image of $M$ via the Harish-Chandra embedding can be described as (see \cite[Chap. III, Thm. 2.9]{AMRT}):
\begin{equation}\label{E:descr-BSD}
i_{HC}(M)=\{X\in \p_+ : T(X)^*\circ T(X) < 2\id_{\p_-} \}.
\end{equation}
From \eqref{E:descr-BSD}, it follows that $i_{HC}(M)$ is a bounded domain, as required.
\end{proof}

\begin{remark}
Let $M$ be a HSM of non-compact type of dimension $n$ and fix a base point $o\in M$.
\begin{enumerate}[(i)]
\item The Harish-Chandra embedding $i_{HC}:M\hookrightarrow \p_+$ (with respect to the base point $o\in M$) can be characterized as the unique open holomorphic embedding of $M$ inside $\bbC^n$,
up to linear complex isomorphisms, such that $i_{HC}(o)=0$ and $i_{HC}(M)$ is a circular domain, i.e. it is stable under multiplication by $\bbS^1\subset \bbC^*$.
See \cite[Chap. II, \S4, Rmk. 1]{Sat} and the references therein.
\item From the description \eqref{E:descr-BSD}, it follows easily that $i_{HC}(M)$ is a convex bounded domain (Hermann's convexity theorem).  Conversely, Mok-Tsai proved that, if the rank of
$M$ is greater than one, then the Harish-Chandra embedding is the unique embedding of $M$ inside  $\bbC^n$, up to complex affine transformations, as a bounded convex domain; see
\cite[Chap. 5, \S2]{Mok} and the references therein.
\end{enumerate}
\end{remark}

We are now going to show that, conversely, any bounded symmetric domain can be endowed with a canonical Hermitian metric with respect to which it becomes a HMS of non-compact type.

Let $D\subset \bbC^N$ be any bounded domain. Let $\calH^2(D)$ be the separable Hilbert space consisting of all holomorphic functions on $D$ that are square integrable with respect to the Euclidean
measure $d\mu$ on $D$ (see \cite[Chap. VIII, Cor. 3.2]{Hel}). Choose an orthonormal basis $\{e_n(z)\}_{n\in \bbN}$ of $\calH(D)$ and set
\begin{equation}\label{E:BergKer}
\begin{aligned}
K_D: D\times D & \longrightarrow \bbC, \\
(z,w) & \mapsto K_D(z,w):=\sum_{n\in \bbN} e_n(z)\cdot \ov{e_m(w)},
\end{aligned}
\end{equation}
where the right hand side converges absolutely and uniformly on any compact subset of $D\times D$ (see \cite[Chap. VIII, Thm. 3.3]{Hel}). The function $K_D$, known as the \emph{Bergman kernel function}
of $D$, is independent of the choice of the orthonormal basis $\{e_n(z)\}$ (see \cite[Chap. VIII, Thm. 3.3]{Hel}) and it can be intrinsically characterized (see \cite[Chap. II, \S 6]{Sat})
as the unique function $K_D:D\times D\to \bbC$ such that
\begin{enumerate}[(i)]
\item $K_D(z,w)=\ov{K_D(w,z)}$ for any $z,w\in D$.
\item For any $w\in D$, the function $z\mapsto K_D(z,w)$ belongs to $\calH^2(D)$.
\item For any $f\in \calH^2(D)$, we have that
$$f(z)=\int_D K_D(z,w)f(w)d\mu(w).$$
\end{enumerate}
Fix now coordinates $z=(z^1,\cdots, z^N)$ of $\bbC^N$ and consider the smooth tensor of type $(0,2)$ defined by
\begin{equation}\label{E:Ber-met}
h_D=\sum_{1\leq i,j\leq N} \frac{\partial^2}{\partial z^i\partial \ov z^j} \log  K_D(z,z) dz^i  d\ov z^j.
\end{equation}

\begin{thm}\label{T:Ber-BSD}
Let $D\subset \bbC^N$ be a bounded domain and consider the complex structure $J_D$ on $D$ inherited from $\bbC^N$.
\begin{enumerate}[(i)]
\item \label{T:Ber-BSD1} The tensor $h_D$ defines a Hermitian metric (called the \emph{Bergman metric} of $D$) on the complex manifold $(D, J_D)$, which is invariant under $\Hol(D)$.
In particular, $\Aut(D,J_D,h_D)=\Hol(D)$.
\item \label{T:Ber-BSD2} If $D$ is a bounded symmetric domain, then $(D,J_D,h_D)$ is a HSM of non-compact type.
\end{enumerate}
\end{thm}
\begin{proof}
Part \eqref{T:Ber-BSD1} is proved in  \cite[Chap. VIII, Prop. 3.4, Prop. 3.5]{Hel}.

Part \eqref{T:Ber-BSD2}: by assumption, $\Hol(D)$ acts transitively on $D$ and each point $p\in D$ has a symmetry $s_p\in \Hol(D)$. Since $\Hol(D)=\Aut(D,J_D,h_D)$ by part \eqref{T:Ber-BSD1},
it follows that $\Aut(D,J_D,h_D)$ acts transitively on $D$ and that the symmetry $s_p$ at the point $p\in D$  belongs to $\Aut(D,J_D,h_D)$. Therefore, $(D,J_D,h_D)$ is a Hermitian symmetric manifold.
The fact that $(D,J_D,h_D)$ is of non-compact type is proved in \cite[Chap. VIII, Thm. 7.1(i)]{Hel}.
\end{proof}

\begin{remark}
\noindent
\begin{enumerate}[(i)]
\item The Bergman metric $h_D$ of a bounded domain $D\subset \bbC^N$ is K\"ahler, i.e. $\Im h_D$ is a closed $2$-form (see \cite[Chap. VIII, Prop. 3.4]{Hel}).
\item If $D\subset \bbC^N$ is a homogeneous bounded domain, then $h_D$ is K\"ahler-Einstein, i.e. its Ricci curvature is proportional to the associated Riemannian metric $\Re h_D$
(see \cite[Chap. VIII, Prop. 3.6]{Hel}).
\end{enumerate}
\end{remark}

\begin{example}\label{E:unitdisk}
Consider the open unitary disk
$$\Delta:=\{z\in \bbC : |z|<1 \}\subset \bbC.$$
Clearly, $\Delta$ is a bounded domain. Its Bergman Hermitian metric is equal to (see \cite[Chap. IX, \S 2]{FK})
$$h_{\Delta}=\frac{4}{1-|z|^2} dz d\ov z .$$
The unitary disk $\Delta$ is biholomorphic to the upper half space $\calH$ of Example \ref{E:dim1}\eqref{E:dim1b} via the Cayley transforms
\begin{equation}\label{E:Cayley-1dim}
\begin{aligned}
\phi: \calH & \stackrel{\cong}{\longrightarrow}  \Delta \\
\tau &\mapsto  \frac{\tau-i}{\tau+i}.
\end{aligned}
\end{equation}
The pull-back via $\phi$ of the Riemannian metric $\Re h_{\Delta}$ on $\Delta$ is the hyperbolic metric on $\calH$ introduced in Example \ref{E:dim1}\eqref{E:dim1b}.
Therefore, $\Delta$ is a bounded symmetric domain.
\end{example}

Putting together Theorem \ref{T:im-HC} and Theorem \ref{T:Ber-BSD}, we obtain the following correspondence between HSM of non-compact type and bounded symmetric domains.

\begin{thm}\label{T:HSM-BSD}
The maps
\begin{equation}\label{E:HSM-BSD}
\begin{aligned}
\left\{\text{HSMs of non-compact type}\right\} & \stackrel{}{\longrightarrow} \left\{\text{Bounded symmetric domains}\right\} \\
(M, J, h) & \longrightarrow i_{HC}(M)\subset \p_+\\
(D,J_D, h_D) & \longleftarrow D\subset \bbC^N
\end{aligned}
\end{equation}
are bijections which are inverses of each other. Moreover, the above bijections send irreducible HSMs of non-compact type into irreducible bounded symmetric domains and conversely.
\end{thm}
\begin{proof}
It follows from Theorems \ref{T:im-HC} and \ref{T:Ber-BSD} that the above maps are well-defined. The fact that they are inverses of each other can be extracted from the proof of \cite[Chap. VIII, Thm. 7.1]{Hel}.
The last assertion is obvious.
\end{proof}

The rank of a bounded symmetric domain can be characterized in the following way.

\begin{thm}[Polydisk theorem]\label{T:polydisk}
Let $D$ be a bounded symmetric domain with its Bergman metric $h_D$ and fix a base point $o\in D$. If the rank of $D$ is equal to $r$ then there exists a totally geodesic polydisk
$\Delta^r\subseteq D$ of dimension $r$  such that the restriction of $h_D$ to $\Delta^r$ is equal to the Bergman metric of $\Delta^r$ and $D=\Stab(o)\cdot \Delta^r$.
\end{thm}
\begin{proof}
See \cite[Chap. 5, (1.1), Thm. 1]{Mok}.
\end{proof}

The correspondence in Theorem \ref{T:HSM-BSD} together with the classification of irreducible HSMs of non-compact type recalled in \S\ref{SS:HSMLiegrp} gives a classification of irreducible bounded symmetric domains.  We record the irreducible bounded symmetric domains  in their Harish-Chandra embeddings (together with their complex dimensions and their ranks) into Table \ref{F:BSD},
referring to \S \ref{S:irrHSM} for more details.

\begin{table}[!ht]
\begin{tabular}{|c|c|c|c|c|c|}
\hline
Type  & Bounded symmetric domain & $\dim_{\bbC}$ & Rank\\
\hline \hline
$I_{p,q}$ & $ \{Z\in M_{p,q}(\bbC): \: Z^t\cdot \ov{Z}<I_q \}$ & $pq$ & $q$ \\
\hline
$II_n$ & $\{Z\in M_{n,n}^{\rm skew}(\bbC): \:  Z^t\cdot \ov{Z}<I_n\}$ & $\binom{n}{2}$ & $\lfloor \frac{n}{2} \rfloor $\\
\hline
$III_n$ & $\{ Z\in M_{n,n}^{\rm sym}(\bbC): \:  Z^t \cdot \ov{Z}<I_n\}$& $\binom{n+1}{2}$& $n$ \\
\hline
$IV_n $ & $ \{Z\in \bbC^n : \: 2\ov{Z}^tZ<1+|Z^tZ|^2, \: \ov{Z}^tZ<1\}  $& $n$ & $\min\{2,n\} $\\
\hline
$V$ & $\calD_{V}\subset \bbO_{\bbC}^2 $  & $16$ & $2$ \\
\hline
$VI$ & $\calD_{VI} \subset H_3(\bbO_{\bbC}) $  & $27$ & $3$  \\
\hline
\end{tabular}
\caption{Irreducible bounded symmetric domains in their Harish-Chandra embeddings.}\label{F:BSD}
\end{table}

\begin{remark}\label{R:BSD-Hodge}
Bounded symmetric domains play a crucial role in Hodge theory. Indeed, on one hand if a period domain $D$ is such that its universal family of Hodge structures satisfies Griffiths transversality 
then $D$ is a bounded symmetric domain. On the other hand, Deligne has shown that every bounded  symmetric domain can be realized as the subdomain of a period domain on which 
certain tensors for the universal family are of Hodge type. In particular, every HSM of non-compact type can be realized as a moduli space for Hodge structures plus tensors.
We refer the reader to \cite[Sec. 7]{Mil2} and the references therein. HSM of non-compact type embedded (equivariantly and horizontally) inside period domains have been recently characterized by 
R. Friedmann and R. Laza \cite{FL}.
\end{remark}

\subsection{HSMs of compact type as cominuscle homogeneous varieties}\label{SS:cominu}

As a consequence of the Borel embedding (see Theorem \ref{T:embed}), we can describe HSMs of compact type as cominuscle homogeneous (projective) varieties.

\begin{defi}\label{D:comin}
 A rational homogeneous projective variety $H/Q$, where $H$ is a semisimple complex Lie group (or, equivalently, algebraic group) and $Q$ is a parabolic subgroup of $H$, is said to be
a \textbf{cominuscle homogeneous variety} if the unipotent radical of $Q$ is abelian.
\end{defi}

\begin{remark}\label{R:comin-root}
If $H$ is a simple complex algebraic group, then $H/Q$ is a cominuscle homogeneous variety if and only if $Q$ is, up to conjugation, a standard maximal parabolic subgroup associated to a
\textbf{cominuscle (or special) simple root}, i.e.
a simple root occurring with coefficient $1$ in the simple root decomposition of the highest positive root (see \cite[Lemma 2.2]{RRS} for a proof). In this case, $H/Q$ is called an \emph{irreducible}  cominuscle
homogeneous variety and clearly any cominuscle homogenous variety can be written uniquely as a product of irreducible ones.
\end{remark}

In \S\ref{SS:HC-B}, we have seen that any HSM $M^*=G_c/K$ of compact type is isomorphic to $G_{\bbC}/(P_-\rtimes K_{\bbC})$ (see Theorem \ref{T:embed}), which is a cominuscle homogeneous
variety  since $G_{\bbC}$ is a semisimple complex Lie group and $P_-\rtimes K_{\bbC}$ is a parabolic subgroup whose unipotent radical $P_-$ is abelian.

\begin{thm}\label{T:HSM-comin}
The map
\begin{equation}\label{E:HSM-comin}
\begin{aligned}
\left\{\text{HSMs of compact type}\right\} & \longrightarrow \left\{\text{Cominuscle homogeneous varieties}\right\} \\
G_c/K & \mapsto G_{\bbC}/(P_-\rtimes K_{\bbC})
\end{aligned}
\end{equation}
is a bijection sending irreducible HSMs of compact type into irreducible cominuscle homogeneous varieties.
\end{thm}
\begin{proof}
See e.g. \cite[5.5]{RRS}.
\end{proof}

The rank of a cominuscle homogeneous variety can be characterized in the following way.

\begin{thm}[Polysphere theorem]\label{T:polysphere}
Let $X$ be a cominuscle homogeneous variety with a base point $o$ and let $h$ the Hermitian metric coming from the bijection of Theorem \ref{T:HSM-comin}.
If the rank of $X$ is equal to $r$ then there exists a totally geodesic polysphere
$(\bbP^1)^r\subseteq X$ of dimension $r$  such that the restriction of $h$ to $(\bbP^1)^r$ is equal to the product of the Fubini-Studi metrics on each factor and
$X=\Stab(o)\cdot (\bbP^1)^r$.
\end{thm}
\begin{proof}
See \cite[Chap. 5, (1.1), Thm. 1]{Mok}.
\end{proof}

The correspondence in Theorem \ref{T:HSM-comin} together with the classification of irreducible HSMs of compact type recalled in \S\ref{SS:HSMLiegrp} gives a classification of irreducible cominuscle homogeneous varieties (see also \cite[Sec. 2.1, 3.1]{LM2} and \cite[Sec. 3]{LM1}). We record the irreducible cominuscle homogeneous varieties together with their associated cominuscle simple roots  (see Remark \ref{R:comin-root}) into Table \ref{F:cominus},  referring to \S \ref{S:irrHSM} for more details.

\begin{table}[!ht]
\begin{tabular}{|c|c|c|c|c|c|c|}
\hline
 Cominuscle homogeneous variety & Cominuscle simple roots \\
\hline \hline
  $I_{p,q}: \hspace{1cm} \Gras(q,p+q)$ &
  $\xymatrix{  *{\circ}  \ar@{-}[r]_(0){1} & *{\circ}  \ar@{-}[r]_(0){\cdots} &*{\bullet}  \ar@{-}[r]_(0){q} &*{\circ}  \ar@{-}[r]_(0){\cdots} &*{\bullet}  \ar@{-}[r]_(0){p}& *{\circ}  \ar@{-}[r]_(0){\cdots}_(1){p+q-1} & *{\circ}  & A_{p+q-1}} $   \\
\hline

 $\xymatrix{ \\
II_{n}: \hspace{1cm}  \Gras_{\rm ort}(n,2n)^o\\
 \\
 }$ &
  $\xymatrix{
  & & &  & & *{\bullet} &\\
  *{\circ} \ar@{-}[r]_(0){1} & *{\circ}  \ar@{-}[r] &*{\circ} \ar@{-}[r]_(0){\cdots} &*{\circ} \ar@{-}[r]_(0.8){n-2}&*{\circ}  \ar@{-}[ru]_(1){n-1} \ar@{-}[rd]_(1){n}&  &  D_n\\
  & & & & & *{\bullet} & \\
   } $\\
\hline

 $III_{n}: \hspace{1cm} \Gras_{\rm sym}(n,2n)$ &
  $\xymatrix{  *{\circ} \ar@{-}[r]_(0){1} & *{\circ} \ar@{-}[r] &*{\circ} \ar@{-}[r]_(0){\cdots} &*{\circ} \ar@{-}[r] &*{\circ}  & *{\bullet}\ar@2[l]^(0){n} & C_n } $\\
\hline

 $ \xymatrix{
 \\
 \\
IV_{n}: \hspace{2.4cm} \calQ^n\\
 \\
 }$
 &
  $ \xymatrix{
  & & &   & *{\circ}\\
  *{\bullet} \ar@{-}[r]_(0){1}  &*{\circ}  \ar@{-}[r]_(0){\cdots} &*{\circ}  \ar@{-}[r]_(0.7){\frac{n}{2}-1}&*{\circ}  \ar@{-}[ru]_(1){\frac{n}{2}} \ar@{-}[rd]_(1){\frac{n}{2}+1}& &D_{\frac{n}{2}+1} \text{ if } n \: \text{ is even }  \\
  & & &  & *{\circ} \\
   *{\bullet} \ar@{-}[r]_(0){1}  &*{\circ}  \ar@{-}[r]_(0){\cdots} &*{\circ}  \ar@{-}[r]_(1){\frac{n-1}{2}} &*{\circ}   \ar@2[r]_(1){\frac{n+1}{2}}  & *{\circ} & B_{\frac{n+1}{2}} \text{ if } n \: \text{ is odd }
      }
   $\\
\hline

 $V: \hspace{1.8cm} \begin{aligned} \bbP^2_{\bbO} \\  \text{ Cayley plane } \end{aligned}
 $  &
  $ \xymatrix{
  *{\bullet}  \ar@{-}[r] & *{\circ}  \ar@{-}[r] &*{\circ}  \ar@{-}[r] \ar@{-}[d] &*{\circ}  \ar@{-}[r] & *{\bullet} & E_6 \\
  & & *{\circ} & & & }$  \\
\hline

 $VI: \hspace{1cm} \begin{aligned}  \calF=\Gras_{\omega}(\bbO^3,\bbO^6) \\  \text{ Freudenthal variety } \end{aligned}$  &
  $
   \xymatrix{
   *{\circ}  \ar@{-}[r] & *{\circ}  \ar@{-}[r] &*{\circ}  \ar@{-}[r] \ar@{-}[d] &*{\circ}  \ar@{-}[r] &*{\circ}  \ar@{-}[r] & *{\bullet} & E_7 \\
   & & *{\circ} & & & & \\ } $  \\
\hline
\end{tabular}
\caption{Irreducible cominuscle homogeneous varieties and their associated cominuscle simple roots.}\label{F:cominus}
\end{table}

Note that each of the varieties $\Gras(q,p+q)$, $\Gras_{\rm ort}(n,2n)$ and $\bbP^2(\bbO)$ corresponds to two cominuscle simple roots; this is due to the fact that in each of the above cases there is an automorphism of the Dinkin  diagram  that exchanges the two roots, and hence inducing an outer automorphism of the associated simple complex algebraic group that establishes an isomorphism of their
associated cominuscle homogeneous varieties.

\begin{remark}\label{R:BSD-root}
Theorem \ref{T:HSM-comin} and Remark \ref{R:comin-root}, together with  the correspondence between (irreducible) bounded symmetric domains and (irreducible) cominuscle varieties (see Corollary \ref{C:dual-HSM} and Theorems \ref{T:HSM-BSD}), provide an explicit bijection between irreducible bounded symmetric domain and cominuscle simple roots of Dinkin diagrams, up to the action of the automorphism group of the Dinkin diagram. This correspondence is described explicitly in \cite[Sec. 1]{Mil} and \cite[Sec. 2]{Mil2} (following Deligne), without passing to the compact dual HSM.
\end{remark}

\subsection{Classifying HSMs of non-compact type via Jordan theory}\label{SS:HSM-JTS}

The aim of this subsection is to explain an alternative approach to the classification of HSMs of non-compact type which is based on the Jordan theory, rather than the Lie theory as in \S\ref{SS:HSMLiealg}.

We start by giving the definition of Jordan triple systems, referring the interested reader to \cite[Part V]{FKKLR} for an excellent survey on Jordan triple systems, with special emphasis on Hermitian positive JTSs.

\begin{defi}\label{D:JTS}
A \emph{Jordan triple system} over a field $F$ is a pair $(V,\{.\,,.\,,.\})$ consisting of a (finite-dimensional) $F$-vector space $V$ together with a $F$-multilinear triple product $V\times V\times V\to V$,  which satisfies the following properties:
\begin{enumerate}
\item[(JT1)] $\{x,y,z\}=\{z,y,x\}$ for any $x,y,z\in V$;
\item[(JT2)] $[a\square x, b\square y]=((a\square b)x)\square y -x\square((b\square a)y)$ for any $a,b,x,y\in V$,
where $a\square b$ is the endomorphism of $V$ defined by
$$(a\square b)x:=\{a,b,x\}, $$
and $[\,]$ is the usual bracket among endomorphisms of $V$.
\end{enumerate}
A Jordan triple system $(V,\{.\,,.\,,.\})$ over $F$ is said to be:
\begin{enumerate}[(i)]
\item \emph{semisimple} if the trace form
\begin{equation}\label{E:trform}
\begin{aligned}
\tau: V\times V & \longrightarrow F,\\
(x,y) & \mapsto \tau(x,y):=\tr(x\square y),
\end{aligned}
\end{equation}
is non-degenerate.
\item \emph{simple} if $\tau$ is not identically zero and $(V,\{.\,,.\,,.\})$ does not have proper ideals, i.e. proper subvector spaces $I\subset V$ such that
$\{I,V,V\}\subseteq I$ and $\{V,I,V\}\subseteq I$.
\item \emph{Hermitian} if $F=\bbR$ and $V$ has a complex structure with respect to which $\{.\,,.\,,.\}$ is $\bbC$-linear in the first and third factor and $\bbC$-antilinear with respect to the second factor.
\item \emph{Hermitian positive} (or, for short, a \textbf{Hermitian positive JTS}) if it is Hermitian and
\begin{enumerate}
\item[(JTp)] the trace form $\tau$ is positive definite.
\end{enumerate}
\end{enumerate}
\end{defi}
Note that property (JTp) makes sense since the trace form is a Hermitian form on $V$ (see \cite[p. 55]{Sat}).

\begin{prop}\label{P:JTSsimple}
Any Hermitian positive (resp. complex semisimple) JTS decomposes uniquely as a product of simple Hermitian positive (resp. complex simple) JTSs.
\end{prop}
\begin{proof}
See \cite[Part V, Prop. IV.1.4]{FKKLR} for the case of Hermitian positive JTSs and \cite[Thm. 10.14]{Loo3} for the case of complex semisimple JTSs.
\end{proof}

\begin{remark}\label{R:comJTS}
There is a natural bijection
\begin{equation*}\label{E:bijJTS}
\begin{aligned}
\left\{\text{Hermitian positive JTSs}\right\} & \stackrel{\cong}{\longrightarrow} \left\{\text{Complex semisimple JTSs} \right\} \\
(V,\{.\,,.\,,.\}) & \mapsto (V,\{x,y,x\}':=\{x,\ov{y},z\})
\end{aligned}
\end{equation*}
which preserves the decomposition into the product of simple Jordan algebras.
\end{remark}

Simple Hermitian positive JTSs (or equivalently simple complex JTSs by Remark \ref{R:comJTS}) can be classified (see \cite[Sec. 17.4]{Loo3}).

\begin{thm}\label{T:simpleJTS}
Every simple Hermitian positive JTS is isomorphic to one of the following:
\begin{enumerate}[(i)]
\item \label{T:simpleJTS1}  $ M_{p,q}(\bbC)$ with Jordan triple product $\{M_1,M_2,M_3\}=\frac{1}{2}(M_1\ov{M_2}^tM_3+M_3\ov{M_2}^t M_1) $.
\item  $ M_{n,n}^{\rm skew}(\bbC):=\{M\in M_{n,n}(\bbC): M^t=-M\}$ with the same Jordan triple product of \eqref{T:simpleJTS1}.
\item    $ M_{n,n}^{\rm sym}(\bbC):=\{M\in M_{n,n}(\bbC): M^t=M\}$ with the same Jordan triple product of \eqref{T:simpleJTS1}.
\item  $\bbC^n$ with Jordan triple product  $\{X,Y,Z\}=(X^t\cdot Z)\ov{Y}-(Z^t\cdot \ov{Y})X-(X^t\cdot \ov{Y})Z $.
\item  $ \bbO_{\bbC}^2 $  with Jordan triple product
$\left\{\begin{pmatrix} a_1 \\ a_2 \end{pmatrix},\begin{pmatrix} b_1 \\ b_2 \end{pmatrix}, \begin{pmatrix} c_1 \\ c_2 \end{pmatrix}\right\}=
\begin{pmatrix} (a_1\wt{\ov{b_1}})c_1+(c_1\wt{\ov{b_1}})a_1+(a_1\ov{b_2})\wt{c_2}+(c_1\ov{b_2})\wt{a_2} \\
\wt{a_1}(\ov{b_1}c_2)+\wt{c_1}(\ov{b_1}a_2)+\wt{a_2}(\ov{b_2}c_2)+\wt{c_2}(\ov{b_2}a_2) \end{pmatrix}$.
\item   $ H_3(\bbO_{\bbC}):=\{M\in M_{3,3}(\bbO_{\bbC}): \wt{M}^t=M\} $  with Jordan triple product $\{a,b,c\}=(a|b)c+(c|b)a-(a\times c)\times \ov{b}$.
\end{enumerate}

\end{thm}

We record the simple Hermitian positive JTSs into Table \ref{F:Jor},   referring to \S \ref{S:irrHSM} for more details.
Among the different types of simple Hermitian positive JTSs, there are the same isomorphisms specified in Table \ref{F:small}.

\begin{table}[!ht]
\begin{tabular}{|c|c|c|c|c|c|}
\hline
Type  & Complex vector space $\p_+$ & Jordan triple product $\{.\,,.\,,.\}$ \\
\hline \hline
$I_{p,q}$ & $ M_{p,q}(\bbC)$ &  $\{M_1,M_2,M_3\}=\frac{1}{2}(M_1\ov{M_2}^tM_3+M_3\ov{M_2}^t M_1) $\\
\hline
$II_n$ & $ M_{n,n}^{\rm skew}(\bbC)$ & $\{M_1,M_2,M_3\}=\frac{1}{2}(M_1\ov{M_2}^tM_3+M_3\ov{M_2}^t M_1) $ \\
\hline
$III_n$ & $ M_{n,n}^{\rm sym}(\bbC)$ & $\{M_1,M_2,M_3\}=\frac{1}{2}(M_1\ov{M_2}^tM_3+M_3\ov{M_2}^t M_1) $ \\
\hline
$IV_n $ & $\bbC^n$ &  $\{X,Y,Z\}=(X^t\cdot Z)\ov{Y}-(Z^t\cdot \ov{Y})X-(X^t\cdot \ov{Y})Z $\\
\hline
$V$ & $ \bbO_{\bbC}^2 $  &
$\left\{\begin{pmatrix} a_1 \\ a_2 \end{pmatrix},\begin{pmatrix} b_1 \\ b_2 \end{pmatrix}, \begin{pmatrix} c_1 \\ c_2 \end{pmatrix}\right\}=
\begin{pmatrix} (a_1\wt{\ov{b_1}})c_1+(c_1\wt{\ov{b_1}})a_1+(a_1\ov{b_2})\wt{c_2}+(c_1\ov{b_2})\wt{a_2} \\
\wt{a_1}(\ov{b_1}c_2)+\wt{c_1}(\ov{b_1}a_2)+\wt{a_2}(\ov{b_2}c_2)+\wt{c_2}(\ov{b_2}a_2) \end{pmatrix}$
  \\
\hline
$VI$ & $ H_3(\bbO_{\bbC}) $  &  $\{a,b,c\}=(a|b)c+(c|b)a-(a\times c)\times \ov{b}$  \\
\hline
\end{tabular}
\caption{Simple Hermitian positive JTSs.}\label{F:Jor}
\end{table}

\vspace{0.2cm}

There is a way to construct  a Hermitian positive JTS starting from a Hermitian SLA of non-compact type.
Indeed, given a Hermitian SLA of non-compact type $(\g,\theta,H)$, consider the decomposition
\begin{equation*}
\g_{\bbC}=\k_{\bbC}\oplus \p_+\oplus \p_{-}
\end{equation*}
given in \eqref{E:dec-g} and denote by $x\to \ov{x}$ the complex conjugation on $\g_{\bbC}$ with respect to the real form $\g$ of $\g_{\bbC}$.
Then the complex vector space $\p_+$ endowed with the triple product
\begin{equation}\label{E:Jortriple}
\begin{aligned}
\{.\,,.\,,.\}: \p_+\times \p_+\times \p_+ &\longrightarrow  \p_+, \\
(x,y,z) & \mapsto \{x,y,z\}:=\frac{1}{2}[[x,\ov{y}],z],
\end{aligned}
\end{equation}
is a Hermitian positive JTS (see \cite[p. 55]{Sat}).

\begin{thm}\label{T:LieJor}
The map
\begin{equation}\label{E:bijeJL}
\begin{aligned}
\left\{\text{Hermitian SLAs of non-compact type}\right\} & \stackrel{}{\longrightarrow} \left\{\text{Hermitian positive JTSs}\right\} \\
(\g,\theta,H) & \mapsto (\p_+,\{.\,,.\,,.\})\\
\end{aligned}
\end{equation}
is a bijection sending irreducible Hermitian SLAs of non-compact type into simple Hermitian positive JTSs.
\end{thm}
\begin{proof}
We limit ourself to defining a map in the other direction, referring to \cite[Chap. II, Prop. 3.3]{Sat} for the verification that it is the inverse of the map \eqref{E:bijeJL}.

Let $(V,\{.\,,.\,,.\})$ be a Hermitian positive JTS. Consider the graded complex vector space
$$\gS=\gS_{-1}\oplus \gS_0\oplus \gS_1:=V\oplus (V\square V) \oplus \ov V,$$
where $V\square V:=\{x\square y :  x,y \in V\}\subseteq \End(V)$ and $\ov V$ is the complex conjugate vector space of $V$, i.e. the complex vector space whose underlying real vector space is equal to
the one of $V$ and such that $i\cdot \ov{v}=-i\cdot v$ for any $v\in V$.

Observe that (JT2) of Definition \ref{D:JTS} implies that $V\square V$ is a complex Lie subalgebra of $\End(V)$ with respect to the usual
Lie bracket $[\,]$ on $\End(V)$. Moreover, if we denote by $T^*$ the adjoint of an endomorphism $T\in \End(V)$ with respect to the Hermitian positive definite trace form $\tau$ (see (JTp) of Definition \ref{D:JTS}),
then (JT2) implies that $(x\square y)^*=y\square x$. In particular, $V\square V\subseteq \End(V)$ is closed under the adjoint operator.
Using these observations, we can define a Lie bracket on $\gS=V\oplus (V\square V) \oplus \ov V$ as it follows:
\begin{equation*}\label{E:bracketS}
[(a,T,\ov b), (a', T', \ov{b'})]:=(Ta'-T'a,2a'\square \ov b+[T,T']-2a\square \ov{b'},(T')^*\ov b-T^*\ov{b'}).
\end{equation*}
It turns out that $(\gS,[\,])$ is a graded semisimple complex Lie algebra (see \cite[Chap. I, Prop. 7.1]{Sat})

Consider now the map
\begin{equation*}
\begin{aligned}
\sigma: \gS & \longrightarrow \gS, \\
(a,T,b) & \mapsto (b, -T^*, a).
\end{aligned}
\end{equation*}
It is easily checked that $\sigma$ is a complex conjugation on the graded complex Lie algebra $\gS$, i.e. it is a $\bbC$-antilinear involution  such that $[\sigma(X),\sigma(Y)]=[X,Y]$ for any $X,Y\in \gS$
and $\sigma(\gS_{i})=\gS_{-i}$ for $i=-1,0,1$.
Moreover, the fact that the trace form $\tau$  is positive definite (by (JTp) of Definition \ref{D:JTS}) implies that $\sigma$ is a Cartan involution of the complex semisimple Lie algebra $\gS$
(see \cite[Chap. I, \S 9]{Sat}), or, equivalently, that the real form $\gS_{\sigma}:=\{X\in \gS : \sigma(X)=X\}$ of the complex Lie algebra $\gS$ defined by $\sigma$ is a compact real form of $\gS$
 (see \cite[Prop. 6.14]{Kna}).

Finally, consider the $\bbC$-linear involution on $\gS$ which preserves the grading on $\gS$ and such that
\begin{equation*}
\theta_{|\gS_i}=
\begin{cases}
\id & \text{ if } i=0, \\
-\id & \text{ if } i=-1, 1.
\end{cases}
\end{equation*}
Then the complex conjugation $X\mapsto \theta(\sigma(X))$ on $\gS$ defines another real form $\g:=\{X\in \gS : \theta(\sigma(X))=X\}$ of $\gS$, on which $\theta$ induces a Cartan involution
(see e.g. \cite[Chap. I, \S 4]{Sat}). If $\g=\k\oplus \p$ is the Cartan decomposition relative to $\theta$ (as in \eqref{E:decompo}), then by construction it follows that $\k_{\bbC}=\gS_0$ and
$\p\otimes_{\bbR}\bbC=\gS_{-1}\oplus \gS_1$. Therefore, arguing as in \S\ref{SS:HSMLiealg}, there exists a unique element $H\in Z(\k)$ such that $\ad(H)^2_{|\p}=-\id_{\p}$ and such that
$\p\otimes_{\bbR}\bbC=\gS_{-1}\oplus \gS_1$ is the decomposition into eigenspaces for $\ad(H)_{|\p}$ relative to the eigenvalues $+1$ and $-1$ (as in \S \ref{SS:HC-B}).

Summing up, starting with a Hermitian positive JTS $(V,\{.\,,.\,,.\})$, we have constructed a Hermitian SLA of non-compact type $(\g,\theta,H)$ and this defines the inverse of the map \eqref{E:bijeJL}
(see \cite[Chap. II, Prop. 3.3]{Sat}).

\end{proof}

Note that the correspondence in Theorem \ref{T:LieJor} together with Theorem \ref{T:HSMLiealg} and the classification of simple Hermitian positive JTSs in Theorem \ref{T:simpleJTS} gives a new
approach to the classification of HSMs of non-compact type, as recalled in \S\ref{SS:HSMLiegrp}.

\begin{remark}\label{R:funSLA-JTS}
The bijection \eqref{E:bijeJL} becomes an equivalence of categories if the left hand set is endowed with morphisms of Hermitian SLAs as in Remark \ref{R:funHSM-SLA} and the right hand side
is endowed with morphisms of Hermitian positive JTSs defined as it follows:
a morphism of Hermitian positive JTSs between two Hermitian positive JTSs $(V,\{.\,,.\,,.\})$ and $(V',\{.\,,.\,,.\}')$ is a $\bbC$-linear map $\psi:V\to V'$ such that
$$\psi(\{x,y,z\})=\{\psi(x),\psi(y),\psi(z)\}' \: \text{ for any } x,y,z\in V.$$
See \cite[Chap. II, \S8]{Sat}.
\end{remark}

\begin{remark}
Let $M$ be a HSM of non-compact type and consider its associated Hermitian positive JTS $(\p_+,\{.\,,.\,,.\})$ (via the bijections of  Theorems \ref{T:HSMLiealg} and \ref{T:LieJor}).
Then the image of the Harish-Chandra embedding $i_{HC}:M\hookrightarrow \p_+$ (as in Theorem \ref{T:embed}) is equal to
\begin{equation}\label{E:HC-JTS}
i_{HC}(M)=\{z\in \p_+ : z\square z <\id_{\p_+}\}.
\end{equation}
See  \cite[Chap. II, Thm. 5.9]{Sat} and the references therein.
\end{remark}

Using the above presentation of HSMs, it is possible to give a Jordan theoretic description of the rank of a HSM of non-compact type (as in Definition \ref{D:rank}). Recall that a \emph{tripotent}
(or idempotent) of a Jordan triple system $(V,\{.\,,.\,,.\})$ is an element $e\in V$ such that $\{e,e,e\}=e$.  Two tripotents $e_1$ and $e_2$ are said to be orthogonal if $\{e_1,e_2,x\}=0$ for any $x\in V$.
A \emph{Jordan frame} of $(V,\{.\,,.\,,.\})$ is a maximal collection $\{e_1,\ldots, e_n\}$ of pairwise orthogonal distinct tripotents.

\begin{prop}\label{P:rank-Jor}
Let $M$ be a HSM of non-compact type and let $(\p_+,\{.\,,.\,,.\})$ the corresponding Hermitian positive JTS according to Theorem \ref{T:HSMLiealg} and Theorem \ref{T:LieJor}.
Then the rank of $M$ is equal to the cardinality of every Jordan frame of $(\p_+,\{.\,,.\,,.\})$.
\end{prop}
\begin{proof}
See \cite{Loo4}.
\end{proof}

\section{Irreducible Hermitian symmetric manifolds}\label{S:irrHSM}

Irreducible HSMs of non-compact type (resp. of compact type) have been classified by Cartan in \cite{Car2} (based upon \cite{Car1})  and they are divided in $6$ \textbf{types} which according to the nowadays standard Siegel's notation\footnote{Cartan's original notation permutes Type III and Type IV.}
are called: $I_{p,q}$ (with $p\geq q\geq 1$), $II_n$ (with $n\geq 2$),  $III_n$ (with $n\geq 1$), $IV_n$ (with $2\neq n\geq 1$), $V$ and $VI$. The first four families are called \emph{classical HSMs} while the last two are called \emph{exceptional HSMs}.
For small values of the parameters there are some isomorphisms between  the above types as shown in Table \ref{F:small}.

\begin{table}[!ht]
\begin{tabular}{|c|c|c|}
\hline
Isomorphic Types   & $\bbR$-Dimension & Rank\\
\hline \hline
$I_{1,1}=II_2=III_1=IV_1$ & $2$ & $1$\\
\hline
$I_{3,1}=II_3$ & $6$ & $1$ \\
\hline
$III_2=IV_3$ & $6$ & $2$ \\
\hline
$I_{2,2}=IV_4$ & $8$ & $2$ \\
\hline
$II_4=IV_6$ & $12$ & $2$ \\
\hline
\end{tabular}
\caption{Isomorphic Types.}\label{F:small}
\end{table}

For a modern proof of the above classification, the reader is referred to \cite[Chap. X, \S6]{Hel}, where such a result is deduced as a corollary of the more general classification of irreducible Riemannian symmetric spaces. The notation used by Helgason differs from the one used by Siegel and the translation between the two different notations is given as it follows: $I_{p,q}=A_{III}$,  $II_n=D_{III}$,
$III_n=CI$, $IV_n=BDI(q=2)$, $V=EVI$ and $VI=EVII$. A direct proof which avoids the classification of Riemannian symmetric manifolds appears in \cite{Wol1}.

In the subsections below, we describe in detail each of the above types following mainly \cite{Wol3} and \cite[Chap. 4, \S2]{Mok} for classical HSMs and \cite{Roos} for exceptional HSMs.

\subsection{Type $I_{p,q}$}

The \textbf{bounded symmetric domain} of type $I_{p,q}$ (with $p\geq q\geq 1$) in its \emph{Harish-Chandra embedding} is given by
 \begin{equation}\label{E:BSDI}
 \calD_{I_{p,q}}:= \{Z\in M_{p,q}(\bbC): \: Z^t\cdot \ov{Z}<I_q \}\subset M_{p,q}(\bbC).
  \end{equation}
  Note that, in the special case $q=1$, $\calD_{I_{p,1}}$ is the unitary ball $B^p:=\{(z_1,\cdots,z_p)\in \bbC^p: \sum_i |z_i|^2<1\}\subset \bbC^p$.

Let $\SU(p,q)$ be the connected simple non-compact Lie subgroup of $\SL(p+q,\bbC)$ that leaves invariant the bilinear Hermitian form on $\bbC^{p+q}\times \bbC^{p+q}$ given by
$-x_1\ov{y}_1-\cdots -x_p \ov{y}_p+x_{p+1}\ov{y}_{p+1}+\cdots+x_{p+q}\ov{y}_{p+q}$. More explicitly
\begin{equation*}
\SU(p,q)=\left\{g\in \SL(p+q,\bbC)\: : \: \ov{g}^t \begin{pmatrix} -I_p & 0 \\ 0 & I_q \end{pmatrix} g= \begin{pmatrix} -I_p & 0 \\ 0 & I_q \end{pmatrix} \right\}=
\end{equation*}
\begin{equation*}
=\left\{\begin{pmatrix} A & B \\ C & D \end{pmatrix} \in \SL(p+q,\bbC) \: : \:
\begin{aligned}
& \ov{A}^t A- \ov{C}^t C=I_p \\
& \ov{D}^t D- \ov{B}^t B=I_q \\
& \ov{A}^tB= \ov{C}^t D
\end{aligned}
\right\}.
\end{equation*}
The Lie group $\SU(p,q)$ acts transitively on $\calD_{I_{p,q}}$ via generalized M\"obius transformations
\begin{equation}\label{E:Moebius}
\begin{aligned}
\SU(p,q)\times \calD_{I_{p,q}} & \longrightarrow \calD_{I_{p,q}} \\
\left( \begin{pmatrix} A & B \\ C & D \end{pmatrix} , Z\right) & \mapsto (AZ+B)(CZ+D)^{-1}.
\end{aligned}
\end{equation}
Notice that the center $Z(\SU(p,q))=\left\{\lambda I_{p+q}\: : \: \lambda^{p+q}=1 \right\}$ of $\SU(p,q)$
acts trivially on $\calD_{I_{p,q}}$; indeed, it turns out that the connected component of the group of biholomorphisms of $\calD_{I_{p,q}}$ is given by
\begin{equation*}
\Hol(\calD_{I_{p,q}})^o=\SU(p,q)/Z(\SU(p,q)):=\PSU(p,q),
\end{equation*}
which is the connected non-compact adjoint simple  Lie group of type $A_{p+q-1}$.

The symmetry of $\calD_{I_{p,q}}$ at the base point $0$ is given by the element
\begin{equation}\label{E:symI}
s_0=\left[\begin{pmatrix} -I_p & 0 \\ 0 & I_q \end{pmatrix} \right]\in \PSU(p,q),
\end{equation}
which acts on $\calD_{I_{p,q}}$ by sending $Z$ into $-Z$. The symmetry $s_0$ induces an involution on $\SU(p,q)$
\begin{equation}\label{E:involu1}
\begin{aligned}
\sigma: \SU(p,q) & \longrightarrow \SU(p,q), \\
\begin{pmatrix} A & B \\ C & D \end{pmatrix} & \mapsto
\begin{pmatrix} -I_p & 0 \\ 0 & I_q \end{pmatrix}\begin{pmatrix} A & B \\ C & D \end{pmatrix}\begin{pmatrix} -I_p & 0 \\ 0 & I_q \end{pmatrix}^{-1}=
\begin{pmatrix} A & -B \\ -C & D \end{pmatrix},
\end{aligned}
\end{equation}
whose fixed Lie subgroup is equal to the maximal compact Lie subgroup
\begin{equation*}
\left\{\begin{pmatrix} A & 0 \\ 0 & D \end{pmatrix}\in \SU(p,q) \right\}=\left\{\begin{pmatrix} A & 0 \\ 0 & D \end{pmatrix}\: : \:
\begin{aligned}
& \ov{A}^tA=I_p, \: \ov{D}^t D=I_q \\
&  \det(A) \det(D)=1
\end{aligned}
 \right\}
=:\GS(U_p\times U_q),
\end{equation*}
which is also equal to the stabilizer of $0\in \calD_{I_{p,q}}$. In particular, the pair $(\SU(p,q), $ $ \GS(U_p\times U_q))$ is a Riemannian symmetric pair.
Notice that the involution $\sigma$ descends to an involution of $\PSU(p,q)$ whose fixed locus is the maximal compact Lie subgroup
 $\ov{\GS(U_p\times U_q)}:=\GS(U_p\times U_q)/Z(\SU(p,q))$ of $\PSU(p,q)$. Therefore also the pair $(\PSU(p,q),\ov{\GS(U_p\times U_q)})$ is a Riemannian symmetric pair.

By the above discussion, we get the following presentation of $\calD_{I_{p,q}}$ as an irreducible \emph{HSM of non-compact type}
\begin{equation}\label{E:nocomI}
\calD_{I_{p,q}}\cong \SU(p,q)/\GS(U_p\times U_q)=\PSU(p,q)/\ov{\GS(U_p\times U_q)},
\end{equation}
associated to the Riemannian symmetric pair $(\SU(p,q), \GS(U_p\times U_q))$ (resp. to $(\PSU(p,q),\ov{\GS(U_p\times U_q)}$).
Notice that the last description of $\calD_{I_{p,q}}$ is the one appearing in Theorem \ref{T:Lie-irr}\eqref{T:Lie-irr1}.

The  irreducible \emph{Hermitian SLA of non-compact type} associated to the Riemannian symmetric pair  $(\SU(p,q), \GS(U_p\times U_q))$ (or equivalently to $(\PSU(p,q),\ov{\GS(U_p\times U_q)}$) is given
by the Lie algebra
\begin{equation}\label{E:Lie-nocompI}
\Lie \SU(p,q)=\mathfrak{su}(p,q)=\left\{ M \in \sl(p+q,\bbC) : \:
\ov{ M}^t \begin{pmatrix} -I_p & 0 \\ 0 & I_q \end{pmatrix} =-  \begin{pmatrix} -I_p & 0 \\ 0 & I_q \end{pmatrix} M\right\}=
\end{equation}
\begin{equation*}
=\left\{ \begin{pmatrix}  Z_1 & Z_2 \\  \ov{Z_2}^t & Z_3  \end{pmatrix}\in \gl(p+q,\bbC) \: : \:
\begin{aligned}
\ov{Z_1}^t=-Z_1, \: \: \: &  \ov{Z_3}^t=-Z_3 \\
Z_2  \in M_{p,q}(\bbC), &  \Tr(Z_1)+\Tr(Z_3)=0 \\
\end{aligned}
\right\}
\end{equation*}
endowed with the Cartan involution $\theta= d\sigma$ given by
\begin{equation*}
\theta \begin{pmatrix}  Z_1 & Z_2 \\  \ov{Z_2}^t & Z_3  \end{pmatrix}=\begin{pmatrix}Z_1 & -Z_2 \\ -\ov{Z_2}^t & Z_3\end{pmatrix}
\end{equation*}
and with the element
\begin{equation*}
H=i \begin{pmatrix} \frac{q}{p+q} I_p & 0\\ 0 & \frac{-p}{p+q} I_q\end{pmatrix}\in \Fix(\theta)=\Lie \GS(U_p\times U_q)=\left\{ \begin{pmatrix}  Z_1 & 0 \\  0 & Z_3  \end{pmatrix}\in \mathfrak{su}(p,q)\right\}.
\end{equation*}

The \textbf{cominuscle homogeneous variety } of type $I_{p,q}$ is the Grassmannian $\Gras(q,p+q)$ parametrizing $q$-dimensional subspaces of $\bbC^{p+q}$:
\begin{equation}\label{E:ComI}
\Gras(q,p+q):=\{[W\subset \bbC^{p+q}]\: : \: \dim W=q\}.
\end{equation}
The \emph{Borel embedding} of $\calD_{I_{p,q}}$ into $\Gras(q,p+q)$ is given by
\begin{equation}\label{E:BorI}
\begin{aligned}
\calD_{I_{p,q}}\subset M_{p,q}(\bbC) & \hookrightarrow \Gras(q,p+q), \\
Z & \mapsto \left\{ \langle v_1,\cdots v_q \rangle \: : \: \{v_1,\cdots, v_q\} \text{ are the column vectors of } \: \begin{pmatrix} Z \\ I_q \end{pmatrix} \right\}.
\end{aligned}
\end{equation}

The complex algebraic simple group $\SL(p+q,\bbC)$ of type $A_{p+q-1}$ acts transitively on $\Gras(q,p+q)$ via
\begin{equation*}
\begin{aligned}
\SL(p+q,\bbC)\times \Gras(q,p+q) & \longrightarrow \Gras(q,p+q) \\
(g, [W\subset \bbC^{p+q}]) & \mapsto [g(W)\subset \bbC^{p+q}].
\end{aligned}
\end{equation*}
Note that the center $Z(\SL(p+q,\bbC))=\{\lambda I_{p+q}\: : \: \lambda^{p+q}=1\} $ of $\SL(p+q,\bbC)$
acts trivially on $\Gras(q,p+q)$; indeed, it turns out that the group of automorphisms of the algebraic variety $\Gras(q,p+q)$ is equal to
$$\PSL(p+q,\bbC):=\SL(p+q,\bbC)/Z(\SL(p+q,\bbC)),$$
which is the connected simple adjoint complex algebraic group of type $A_{p+q-1}$ and it is the complexification of the Lie group $\PSU(p,q)$.

Consider now the base point $W_o:=\langle e_{p+1},\cdots, e_{p+q}\rangle\in \Gras(q,p+q)$ with respect to the standard basis $\{e_1,\cdots,e_{p+q}\}$ of $\bbC^{p+q}$.
The stabilizer of $W_o$ is the maximal parabolic subgroup associated to the $q$-th simple root of the Dinkin diagram $A_{p+q-1}$  (which is cominuscle, see Table \ref{F:cominus})
\begin{equation*}
Q_q:=\left\{ \begin{pmatrix} A & 0 \\ C & D \end{pmatrix}\in \SL(p+q,\bbC)\right\}\subset \SL(p+q,\bbC),
\end{equation*}
where $A\in M_{p,p}(\bbC)$, $C\in M_{q,p}(\bbC)$ and $D\in M_{q,q}(\bbC)$. The parabolic group $Q_q$ admits the following Levi decomposition
\begin{equation*}
Q_q=R_u(Q_q)\rtimes L(Q_q):=\left\{ \begin{pmatrix} I_p & 0 \\ C & I_q \end{pmatrix} \right\} \rtimes \left\{ \begin{pmatrix} A & 0 \\ 0 & D \end{pmatrix}\in \SL(p+q,\bbC)\right\},
\end{equation*}
which coincides with the Levi decomposition appearing in Theorem \ref{T:HSM-comin}.

From the above discussion, we obtain the following explicit presentation of $\Gras(q,p+q)$ as a cominuscle homogeneous variety (as in Definition \ref{D:comin})
\begin{equation}\label{E:G/PI}
\Gras(q,p+q)\cong \SL(p+q,\bbC)/Q_q=\PSL(p+q,\bbC)/\ov{Q_q},
\end{equation}
where $\ov{Q_q}:=Q_q/\left\{\lambda I_{p+q}\: : \: \lambda^{p+q}=1 \right\}$.

Consider now the compact real form of $\SL(p+q,\bbC)$, which  is the Lie subgroup $\SU(p+q)\subset \SL(p+q,\bbC)$ that leaves invariant the positive definite Hermitian form
$x_1\ov{y}_1+\cdots+x_{p+q}\ov{y}_{p+q}$ on $\bbC^{p+q}\times \bbC^{p+q}$. More explicitly
\begin{equation*}
\SU(p+q)=\{g\in \SL(p+q,\bbC)\: : \: \ov{g}^tg=I_{p+q}\}=
\end{equation*}
\begin{equation*}
=\left\{\begin{pmatrix} A & B \\ C & D \end{pmatrix} \in \SL(p+q,\bbC) : \:
\begin{aligned}
& \ov{A}^t A+ \ov{C}^t C=I_p \\
& \ov{D}^t D+ \ov{B}^t B=I_q \\
& \ov{A}^tB= -\ov{C}^t D
\end{aligned}
\right\}.
\end{equation*}
Similarly, the quotient of $\SU(p+q)$ by its center
\begin{equation*}
\PSU(p+q):=\SU(p+q)/Z(\SU(p+q))=\SU(p+q)/\{\lambda I_{p+q}\: : \: \lambda^{p+q}=1\}
\end{equation*}
is the compact real form of $\PSL(p+q,\bbC)$.

The restriction of the action of  $\SL(p+q,\bbC)$  on $\Gras(q,p+q)$ to the subgroup $\SU(p+q)\subset \SL(p+q,\bbC)$ is still transitive and the stabilizer of $W_o$ is the maximal
proper connected and compact subgroup
\begin{equation*}
\SU(p+q)\cap Q_q= \left\{\begin{pmatrix} A & 0 \\ 0 & D \end{pmatrix}\in \SU(p+q) \right\}\cong \GS(U_p\times U_q).
\end{equation*}
The action of $\SU(p+q)$ on $\Gras(q,p+q)$ factors through a transitive action of $\PSU(p+q)$ in such a way that the stabilizer of $W_o$ is equal to the maximal
proper connected and compact subgroup
$$\PSU(p+q)\cap \ov{Q_q}= \left\{\left[\begin{pmatrix} A & 0 \\ 0 & D \end{pmatrix}\right]\in \PSU(p,q) \right\}\cong \ov{\GS(U_p\times U_q)}.
$$

The pair $(\SU(p+q),\GS(U_p\times U_q))$ is a Riemannian symmetric pair since $\GS(U_p\times U_q)$ is the fixed subgroup of the involution
\begin{equation*}
\begin{aligned}
\sigma^*: \SU(p+q) & \longrightarrow \SU(p+q) \\
\begin{pmatrix} A & B \\ C & D \end{pmatrix} & \mapsto \begin{pmatrix} A & -B \\ -C & D \end{pmatrix},
\end{aligned}
\end{equation*}
and similarly for the pair $(\PSU(p+q),\ov{\GS(U_p\times U_q)})$.

By the above discussion, we get the following presentation of $\Gras(q,p+q)$ as the irreducible \emph{HSM of compact type}
\begin{equation}\label{E:comI}
\Gras(q,p+q)\cong \SU(p+q)/\GS(U_p\times U_q)=\PSU(p+q)/\ov{\GS(U_p\times U_q)},
\end{equation}
associated to the Riemannian symmetric pair $(\SU(p+q), \GS(U_p\times U_q))$ (resp. to $(\PSU(p+q),\ov{\GS(U_p\times U_q)}$).
In particular, the last description of $\Gras(q,p+q)$ is the one appearing in Theorem \ref{T:Lie-irr}\eqref{T:Lie-irr2}.
Notice that the symmetry at the base point $W_o$ of $ \Gras(q,p+q)$, seen as a Hermitian symmetric manifold, is given by the element
\begin{equation*}
s_{W_o}=\left[\begin{pmatrix} -I_p & 0 \\ 0 & I_q \end{pmatrix} \right]\in \PSU(p+q).
\end{equation*}

The irreducible \emph{Hermitian SLA of compact type} associated to the Riemann symmetric pair $(\SU(p+q),\GS(U_p\times U_q))$ is given by the Lie algebra
\begin{equation}\label{E:Lie-compI}
\Lie \SU(p+q)=\mathfrak{su}(p+q)=\left\{ \begin{pmatrix}  Z_1 & Z_2 \\  -\ov{Z_2}^t & Z_3  \end{pmatrix}  : \:
\begin{aligned}
\ov{Z_1}^t=-Z_1, \:\:\:  & \ov{Z_3}^t=-Z_3 \\
Z_2 \in M_{p,q}(\bbC), \: & \Tr(Z_1)+\Tr(Z_3)=0 \\
\end{aligned}
\right\}
\end{equation}
endowed with the involution $\theta^*=d\sigma^*$
\begin{equation*}
\theta^* \begin{pmatrix}  Z_1 & Z_2 \\ - \ov{Z_2}^t & Z_3  \end{pmatrix}=\begin{pmatrix}Z_1 & -Z_2 \\ \ov{Z_2}^t & Z_3\end{pmatrix}
\end{equation*}
and with the element
\begin{equation*}
H^*=i \begin{pmatrix} \frac{q}{p+q} I_p & 0\\ 0 & \frac{-p}{p+q} I_q\end{pmatrix}\in \Fix(\theta^*)=\left\{ \begin{pmatrix}  Z_1 & 0 \\  0 & Z_3  \end{pmatrix}\in \mathfrak{su}(p+q)\right\}\cong
  \Lie \GS(U_p\times U_q).
\end{equation*}
Notice that the Hermitian SLA $(\su(p+q),\theta^*,H^*)$ is the dual of the Hermitian SLA $(\su(p,q),\theta,H)$ in the sense of  \S\ref{SS:dual}.

The complexification of the Lie algebras $\su(p,q)$ and $\su(p+q)$ is the complex simple Lie algebra of type $A_{p+q-1}$
\begin{equation*}
\Lie \SL(p+q,\bbC)=\sl(p+q,\bbC)=\left\{ \begin{pmatrix}  Z_1 & Z_2 \\  Z_4 & Z_3  \end{pmatrix}\in \gl(p+q,\bbC) \: : \:  \Tr(Z_1)+\Tr(Z_3)=0 \right\}.
\end{equation*}
The decomposition \eqref{E:dec-g} of $\sl(p+q,\bbC)$ is given by
\begin{equation*}
\sl(p+q,\bbC)=\left\{ \begin{pmatrix}  Z_1 & 0 \\  0 & Z_3  \end{pmatrix}\: : \:  \Tr(Z_1)+\Tr(Z_3)=0  \right\} \oplus \left\{ \begin{pmatrix}  0 & Z_2 \\  0 & 0  \end{pmatrix} \right\} \oplus
 \left\{ \begin{pmatrix}  0 & 0 \\  Z_4 & 0  \end{pmatrix} \right\}.
\end{equation*}
In particular, we have the identification
\begin{equation}\label{E:p+I}
\begin{aligned}
M_{p,q}(\bbC)& \stackrel{\cong}{\longrightarrow} \p_+\\
M & \mapsto \begin{pmatrix}  0 & M \\  0 & 0  \end{pmatrix}.
\end{aligned}
\end{equation}
Using the above identification and the formula \eqref{E:Jortriple}, $M_{p,q}(\bbC)$ becomes a \emph{Hermitian positive JTS} with respect to the triple product
\begin{equation}\label{E:JorI}
\{M_1,M_2,M_3\}=\frac{1}{2}\left(M_1\ov{M_2}^tM_3+M_3\ov{M_2}^t M_1\right).
\end{equation}

\subsection{Type $II_{n}$}

The \textbf{bounded symmetric domain} of type $II_{n}$ ($n\geq 2$) in its \emph{Harish-Chandra embedding} is given by
 \begin{equation}\label{E:BSDII}
 \calD_{II_{n}}:= \{Z\in M_{n,n}^{\rm skew}(\bbC): \: Z^t\cdot \ov{Z}<I_n \}\subset M_{n,n}^{\rm skew}(\bbC):=\{Z\in M_{n,n}(\bbC)\: : \: Z^t=-Z\}.
  \end{equation}
Let $\SO(2n,\bbC)$ be the connected complex Lie subgroup of $\SL(2n,\bbC)$ that leaves invariant the bilinear symmetric form on $\bbC^{2n}\times \bbC^{2n}$ given by
$S(\un{x},\un{y})=x_1y_{n+1}+\cdots+x_{n}y_{2n}+\cdots + x_{n+1}y_1+\cdots+x_{2n}y_n$\footnote{Usually, one defines $\SO(2n,\bbC)$ with respect to the standard bilinear symmetric form on $\bbC^{2n}\times \bbC^{2n}$ given by $x_1y_1+\cdots+x_{2n}y_{2n}$. However, for our purposes it will be more convenient to use this alternative presentation.}. The group $\SO(2n.\bbC)$ is simple of type $D_n$
and it is explicitly given in $n\times n$ block notation as
\begin{equation*}
\SO(2n,\bbC)=\left\{g\in \SL(2n,\bbC) : \: g^t S_n  g= S_n  \right\}=
\end{equation*}
\begin{equation*}
=\left\{\begin{pmatrix} A & B \\ C & D \end{pmatrix} \in \SL(2n,\bbC) : \:
\begin{aligned}
& A^t C= - C^t A \\
& B^t D=- D^t B \\
& A^tD+C^tB= I_n
\end{aligned}
\right\}.
\end{equation*}
Consider the non-compact real form $\SOnc(2n)$ of $\SO(2n, \bbC)$ consisting of all the elements of $\SO(2n,\bbC)$ that leave invariant the bilinear Hermitian form on
$\bbC^{2n}\times \bbC^{2n}$ given by $-x_1\ov{y}_1-\cdots -x_n \ov{y}_n+x_{n+1}\ov{y}_{n+1}+\cdots+x_{2n}\ov{y}_{2n}$. Explicitly,
\begin{equation*}
\SOnc(2n)=\SO(2n,\bbC)\cap \SU(n,n)=\left\{g\in \SO(2n,\bbC) : \: \ov{g}^t \begin{pmatrix} I_n & 0 \\ 0 & -I_n\end{pmatrix}  g= \begin{pmatrix} I_n & 0 \\ 0 & -I_n\end{pmatrix} \right\}=
\end{equation*}
\begin{equation*}
=\left\{\begin{pmatrix} A & B \\ -\ov{B} & \ov{A} \end{pmatrix} \in \SL(2n,\bbC) : \:
\begin{aligned}
& \ov{A}^t A - B^t \ov{B}=I_n \\
& \ov{A}^tB+B^t\ov{A}= 0
\end{aligned}
\right\}.
\end{equation*}
Note that the group $\SOnc(2n)$ is isomorphic to the classical real Lie group $\SO^*(2n)$ via the conjugation inside $\SO(2n,\bbC)$ given by (see \cite[p. 74]{Mok})
\begin{equation}\label{E:iso-IInc}
\begin{aligned}
\SOnc(2n) & \stackrel{\cong}{\longrightarrow} \SO^*(2n):=\{g \in \SL(2n,\bbC) : \: g^t g=I_{2n} \text{ and }Ê\ov{g}^tJ_ng=J_n\}\\
h & \mapsto \begin{pmatrix}I_n & i I_n \\ iI_n & I_n  \end{pmatrix} h \begin{pmatrix}I_n & i I_n \\ iI_n & I_n  \end{pmatrix}^{-1}.
\end{aligned}
\end{equation}

The Lie group $\SOnc(2n)$ acts transitively on $\calD_{II_{n}}$ via generalized M\"obius transformations, as in \eqref{E:Moebius}.
Notice that the center $Z(\SOnc(2n))=\left\{\pm I_{2n} \right\}$ of $\SOnc(2n)$
acts trivially on $\calD_{II_{n}}$; indeed, it turns out that the connected component of the group of biholomorphisms of $\calD_{II_{n}}$ is given by
\begin{equation*}
\Hol(\calD_{II_{n}})^o=\SOnc(2n)/Z(\SOnc(2n)):=\PSOnc(2n),
\end{equation*}
which is the connected non-compact adjoint simple  Lie group of type $D_n$.

The symmetry of $\calD_{II_{n}}$ at the base point $0$ is given by the element
\begin{equation}\label{E:symII}
s_0=\left[\begin{pmatrix} iI_n & 0 \\ 0 & -iI_n \end{pmatrix} \right]\in \PSOnc(2n),
\end{equation}
which acts on $\calD_{II_{n}}$ by sending $Z$ into $-Z$. The symmetry $s_0$ induces an involution on $\SOnc(2n)$
\begin{equation*}
\begin{aligned}
\sigma: \SOnc(2n) & \longrightarrow \SOnc(2n) \\
\begin{pmatrix} A & B \\ C & D \end{pmatrix} & \mapsto
\begin{pmatrix} i I_n & 0 \\ 0 & -i I_n \end{pmatrix}\begin{pmatrix} A & B \\ C & D \end{pmatrix}\begin{pmatrix}  i I_n & 0 \\ 0 & - i I_n \end{pmatrix}^{-1}=
\begin{pmatrix} A & -B \\ -C & D \end{pmatrix},
\end{aligned}
\end{equation*}
whose fixed Lie subgroup is equal to the maximal compact Lie subgroup
\begin{equation*}
\left\{\begin{pmatrix} A & 0 \\ 0 & D \end{pmatrix}\in \SOnc(2n) \right\}=\left\{\begin{pmatrix} A & 0 \\ 0 & \ov{A} \end{pmatrix}\: : \: \ov{A}^t A=I_n \right\}
=:\U(n),
\end{equation*}
which is also equal to the stabilizer of $0\in \calD_{II_{n}}$. In particular, the pair $(\SOnc(2n), $ $ \U(n))$ is a Riemannian symmetric pair.
Notice that the involution $\sigma$ descends to an involution of $\PSOnc(2n)$ whose fixed locus is the maximal compact  Lie subgroup
 $\ov{\U(n)}:=\U(n)/\{\pm I_n\}$ of $\PSOnc(2n)$. Therefore, also the pair $(\PSOnc(2n),\ov{\U(n)})$ is a Riemannian symmetric pair.

By the above discussion, we get the following presentation of $\calD_{II_{n}}$ as an irreducible \emph{HSM of non-compact type}
\begin{equation}\label{E:nocomII}
\calD_{II_{n}}\cong \SOnc(2n)/\U(n)=\PSOnc(2n)/\ov{\U(n)},
\end{equation}
associated to the Riemannian symmetric pair $(\SOnc(2n), \U(n))$ (resp. to $(\PSOnc(2n),$ $\ov{\U(n)}$).
Notice that the last description of $\calD_{II_{n}}$ is the one appearing in Theorem \ref{T:Lie-irr}\eqref{T:Lie-irr1}.

The  irreducible \emph{Hermitian SLA of non-compact type} associated to the Riemannian symmetric pair  $(\SOnc(2n), \U(n))$ (or equivalently to $(\PSOnc(2n),$
$ \ov{U(n)}$) is given  by the Lie algebra
\begin{equation}\label{E:Lie-nocompII}
\sonc(2n):=\Lie \SOnc(2n)=\left\{ M \in \sl(2n,\bbC) : \:
\begin{aligned}
& M^t S_n=-S_n M \\
& \ov{ M}^t \begin{pmatrix} -I_n & 0 \\ 0 & I_n \end{pmatrix} =-  \begin{pmatrix} -I_n & 0 \\ 0 & I_n \end{pmatrix} M\\
\end{aligned}
\right\}=
\end{equation}
\begin{equation*}
=\left\{ \begin{pmatrix}  Z_1 &Z_2 \\ \ov{Z}_2^t & - Z_1^t  \end{pmatrix}\in \gl(2n,\bbC)  : \:
 \ov{Z_1}^t=-Z_1, \: Z_2^t=-Z_2\right\}\footnote{Note that $\sonc(2n)$ is isomorphic to the classical real Lie algebra $\so^*(2n)=\Lie \SO^*(2n)$ via the same conjugation map as in \eqref{E:iso-IInc}.}
\end{equation*}
endowed with the Cartan involution $\theta= d\sigma$ given by
\begin{equation*}
\theta \begin{pmatrix}  Z_1 &Z_2 \\ \ov{Z}_2^t & - Z_1^t  \end{pmatrix}= \begin{pmatrix}  Z_1 &-Z_2 \\ -\ov{Z}_2^t & - Z_1^t  \end{pmatrix}
\end{equation*}
and with the element
\begin{equation*}
H=\frac{i}{2} \begin{pmatrix} I_n & 0\\ 0 & - I_n\end{pmatrix}\in \Fix(\theta)=\left\{ \begin{pmatrix}  Z_1 & 0 \\  0 & -Z_1^t  \end{pmatrix}: \ov{Z}_1^t=-Z_1\right\}\cong \Lie U(n).
\end{equation*}

The \textbf{cominuscle homogeneous variety } of type $II_{n}$ is the connected component $\Gras_{\rm ort}(n,2n)^o$ containing $[W_o:=\langle e_{n+1}, \cdots, e_{2n}\rangle \subset \bbC^{2n}]$ of the orthogonal Grassmannian $\Gras_{\rm ort}(n,2n)$ parametrizing Lagrangian $n$-dimensional subspaces of $\bbC^{2n}$:
\begin{equation}\label{E:ComII}
\Gras_{\rm ort}(n,2n):=\{[W\subset \bbC^{2n}]\: : \: \dim W=p, S_{|W\times W}\equiv 0\},
\end{equation}
where $S$ is the bilinear symmetric non-degenerate form on $\bbC^{2n}$ which is represented by the matrix $S_n=\begin{pmatrix} 0 & I_n \\ I_n & 0 \end{pmatrix} $ in the standard basis of $\bbC^{2n}$.

The \emph{Borel embedding} of $\calD_{II_n}$ into $\Gras_{\rm ort}(n,2n)^o$ is given by
\begin{equation}\label{E:BorII}
\begin{aligned}
\calD_{II_n}\subset M_{n,n}^{\rm skew}(\bbC) & \hookrightarrow \Gras_{\rm ort}(n,2n)^o\\
Z & \mapsto \left\{ \langle v_1,\cdots v_n \rangle : \: \{v_1,\cdots, v_n\} \text{ are the column vectors of } \begin{pmatrix} Z \\ I_n \end{pmatrix} \right\}.
\end{aligned}
\end{equation}
The complex algebraic simple group $\SO(2n,\bbC)$ acts transitively on $\Gras_{\rm ort}(n,2n)^o$ via
\begin{equation*}
\begin{aligned}
\SO(2n,\bbC)\times \Gras_{\rm ort}(n,2n)^o & \longrightarrow \Gras_{\rm ort}(n,2n)^o \\
(g, [W\subset \bbC^{2n}]) & \mapsto [g(W)\subset \bbC^{2n}].
\end{aligned}
\end{equation*}
Note that the center $Z(\SO(2n,\bbC))=\{\pm I_{2n} \} $ of $\SO(2n,\bbC)$
acts trivially on $\Gras_{\rm ort}(n,2n)^o$; indeed, it turns out that the group of automorphisms of the algebraic variety $\Gras_{\rm ort}(n,2n)^o$ is equal to
$$\PSO(2n,\bbC):=\PSO(2n,\bbC)/\{\pm I_{2n}\},$$
which is the connected simple adjoint complex algebraic group of type $D_n$ and it is the complexification of the Lie group $\PSOnc(2n)$.

The stabilizer of $[W_o=\langle e_{n+1}, \cdots, e_{2n}\rangle\subset \bbC^{2n}] \in \Gras_{\rm ort}(n,2n)^o$ is the maximal parabolic subgroup associated to the $n$-th simple root of the Dinkin diagram $D_n$  (which is a cominuscle simple root of $D_n$,
see Table \ref{F:cominus})\footnote{As it is seen from Table \ref{F:cominus}, we could have chosen the $(n-1)$-th simple root  and we would have gotten an isomorphic (although non conjugate) parabolic subgroup.}
\begin{equation*}
Q_n:=\left\{ \begin{pmatrix} A & 0 \\ C & D \end{pmatrix}\in \SO(2n,\bbC)\right\}\subset \SO(2n,\bbC),
\end{equation*}
where $A\in M_{n,n}(\bbC)$, $C\in M_{n,n}(\bbC)$ and $D\in M_{n,n}(\bbC)$. The parabolic group $Q_n$ admits the following Levi decomposition
\begin{equation*}
Q_n=R_u(Q_n)\rtimes L(Q_n):=\left\{ \begin{pmatrix} I_p & 0 \\ C & I_q \end{pmatrix} : \: C^t=-C \right\} \rtimes \left\{ \begin{pmatrix} A & 0 \\ 0 & D \end{pmatrix}\in \SO(2n,\bbC)\right\},
\end{equation*}
which coincides with the Levi decomposition appearing in Theorem \ref{T:HSM-comin}.

From the above discussion, we obtain the following explicit presentation of $\Gras_{\rm ort}(n,2n)^o$ as a cominuscle homogeneous variety (as in Definition \ref{D:comin})
\begin{equation}\label{E:G/PII}
\Gras_{\rm ort}(n,2n)^o\cong \SO(2n,\bbC)/Q_n=\PSO(n,\bbC)/\ov{Q_n},
\end{equation}
where $\ov{Q_n}:=Q_n/\{\pm I_{2n}\}$.

Consider now the compact real form $\SOc(2n,\bbC)$ of $\SO(2n,\bbC)$ consisting of all the elements of $\SO(2n,\bbC)$ that leaves invariant the positive definite Hermitian form
$x_1\ov{y}_1+\cdots+x_{2n}\ov{y}_{2n}$ on $\bbC^{2n}$. More explicitly
\begin{equation*}
\SOc(2n):=\SO(2n,\bbC)\cap \SU(2n)=\left\{g\in \SO(2n,\bbC) : \: \ov{g}^t   g= I_{2n}  \right\}=
\end{equation*}
\begin{equation*}
=\left\{\begin{pmatrix} A & B \\ \ov{B} & \ov{A} \end{pmatrix} \in \SL(2n,\bbC) : \:
\begin{aligned}
 &  A^t \ov{A} + \ov{B}^t B=I_n\\
& A^t\ov{B}+\ov{B}^t A = 0
\end{aligned}
\right\}.
\end{equation*}
Note that the group $\SOc(2n)$ is isomorphic to the real orthogonal group $\SO(2n)$ via the same conjugation map as in \eqref{E:iso-IInc}.
Similarly, the quotient of $\SOc(2n)$ by its center
\begin{equation*}
\PSOc(2n):=\SOc(2n)/Z(\SOc(2n))=\SOc(2n)/\{\pm I_{2n} \}
\end{equation*}
is a compact real form of $\PSO(2n,\bbC)$.

The restriction of the action of  $\SO(2n,\bbC)$  on $\Gras_{\rm ort}(n,2n)^o$ to the subgroup $\SOc(2n)\subset \SO(2n,\bbC)$ is still transitive and the stabilizer of $W_o$ is the maximal
proper connected and compact subgroup
\begin{equation*}
\SOc(2n)\cap Q_n= \left\{\begin{pmatrix} A & 0 \\ 0 & D \end{pmatrix}\in \SOc(2n) \right\}= \left\{\begin{pmatrix} A & 0 \\ 0 & \ov{A} \end{pmatrix}: \ov{A}^t A=I_{2n}\right\}\cong U(n).
\end{equation*}
The action of $\SOc(2n)$ on $\Gras_{\rm ort}(n,2n)^o$ factors through a transitive action of $\PSOc(2n)$ in such a way that the stabilizer of $W_o$ is equal to the maximal
proper connected and compact subgroup
$$\PSOc(2n)\cap \ov{Q_n}= \left\{\left[\begin{pmatrix} A & 0 \\ 0 & D \end{pmatrix}\right]\in \PSOc(2n) \right\} = \left\{\left[\begin{pmatrix} A & 0 \\ 0 & \ov{A} \end{pmatrix}\right]: \ov{A}^t A=I_{2n}\right\}\cong
\ov{U(n)}.
$$

The pair $(\SOc(2n),U(n))$ is a Riemannian symmetric pair since $U(n)$ is the fixed subgroup of the involution
\begin{equation*}
\begin{aligned}
\sigma^*: \SOc(2n) & \longrightarrow \SOc(2n) \\
\begin{pmatrix} A & B \\ C & D \end{pmatrix} & \mapsto \begin{pmatrix} A & -B \\ -C & D \end{pmatrix},
\end{aligned}
\end{equation*}
and similarly for the pair $(\PSOc(2n),\ov{U(2n)})$.

By the above discussion, we get the following presentation of $\Gras_{\rm ort}(n,2n)^o$ as the irreducible \emph{HSM of compact type}
\begin{equation}\label{E:comII}
\Gras_{\rm ort}(n,2n)^o\cong \SOc(2n)/U(n)=\PSOc(2n)/\ov{U(n)},
\end{equation}
associated to the Riemannian symmetric pair $(\SOc(2n), U(n))$ (resp. to $(\PSOc(2n),$ $ \ov{U(n)}$).
In particular, the last description of $\Gras_{\rm ort}(n,2n)^o$ is the one appearing in Theorem \ref{T:Lie-irr}\eqref{T:Lie-irr2}.
Notice that the symmetry at the base point $W_o$ of $ \Gras_{\rm ort}(n,2n)^o$, seen as a Hermitian symmetric manifold, is given by the element
\begin{equation*}
s_{W_o}=\left[\begin{pmatrix} -I_p & 0 \\ 0 & I_q \end{pmatrix} \right]\in \PSOc(2n).
\end{equation*}

The irreducible \emph{Hermitian SLA of compact type} associated to the Riemann symmetric pair $(\SOc(2n),U(2n))$ is given by the Lie algebra
\begin{equation}\label{E:Lie-compII}
\soc(2n):=\Lie \SOc(2n)=\left\{ M \in \sl(p+q,\bbC) : \:
\begin{aligned}
& M^t S_n=-S_n M \\
& \ov{ M}^t  =- M\\
\end{aligned}
\right\}=
\end{equation}
\begin{equation*}
=\left\{ \begin{pmatrix}  Z_1 &Z_2 \\ -\ov{Z}_2^t & - Z_1^t  \end{pmatrix}\in \gl(p+q,\bbC)  : \:
 \ov{Z_1}^t=-Z_1, \: Z_2^t=-Z_2\right\}\footnote{Note that $\soc(2n)$ is isomorphic to the classical real Lie algebra $\so(2n)=\Lie \SO(2n)$ via the same conjugation map as in \eqref{E:iso-IInc}}
\end{equation*}
endowed with the involution $\theta^*=d\sigma^*$
\begin{equation*}
\theta^* \begin{pmatrix}  Z_1 & Z_2 \\ - \ov{Z_2}^t & -Z_1^t  \end{pmatrix}=\begin{pmatrix}Z_1 & -Z_2 \\ \ov{Z_2}^t & - Z_1^t\end{pmatrix}
\end{equation*}
and with the element
\begin{equation*}
H^*=\frac{i}{2} \begin{pmatrix} I_n & 0\\ 0 &-I_n\end{pmatrix}\in \Fix(\theta^*)=\left\{ \begin{pmatrix}  Z_1 & 0 \\  0 & -Z_1^t  \end{pmatrix}: \ov{Z}_1^t=-Z_1\right\}\cong \Lie U(n).
\end{equation*}
Notice that the Hermitian SLA $(\soc(2n),\theta^*,H^*)$ is the dual of the Hermitian SLA $(\sonc(2n),\theta,H)$ in the sense of  \S\ref{SS:dual}.

The complexification of the Lie algebras $\sonc(2n)$ and $\soc(2n)$ is the complex simple Lie algebra of type $D_n$
\begin{equation*}
\Lie \SO(2n,\bbC)=\so(2n,\bbC)=\left\{ \begin{pmatrix}  Z_1 & Z_2 \\  Z_3 & -Z_1^t  \end{pmatrix}\in \gl(p+q,\bbC) \: : \:  Z_2^t=-Z_2, \: Z_3^t=-Z_3 \right\}.
\end{equation*}
The decomposition \eqref{E:dec-g} of $\so(p+q,\bbC)$ is given by
\begin{equation*}
\sl(p+q,\bbC)=\left\{ \begin{pmatrix}  Z_1 & 0 \\  0 & - Z_1^t  \end{pmatrix} \right\} \oplus \left\{ \begin{pmatrix}  0 & Z_2 \\  0 & 0  \end{pmatrix} : \: Z_2^t=-Z_2\right\} \oplus
 \left\{ \begin{pmatrix}  0 & 0 \\  Z_3 & 0  \end{pmatrix}: \: Z_3^t=-Z_3 \right\}.
\end{equation*}
In particular, we have the identification
\begin{equation}\label{E:p+II}
\begin{aligned}
M_{n,n}^{\rm skew}(\bbC)& \stackrel{\cong}{\longrightarrow} \p_+\\
M & \mapsto \begin{pmatrix}  0 & M \\  0 & 0  \end{pmatrix}.
\end{aligned}
\end{equation}
Using the above identification and the formula \eqref{E:Jortriple}, $M_{n,n}^{\rm skew}(\bbC)$ becomes a \emph{Hermitian positive JTS} with respect to the triple product
\begin{equation}\label{E:JorII}
\{M_1,M_2,M_3\}=\frac{1}{2}\left(M_1\ov{M_2}^tM_3+M_3\ov{M_2}^t M_1\right).
\end{equation}

\subsection{Type $III_n$}

The \textbf{bounded symmetric domain} of type $III_{n}$ ($n\geq 1$) in its \emph{Harish-Chandra embedding} is given by
 \begin{equation}\label{E:BSDIII}
 \calD_{III_{n}}:= \{ Z\in M_{n,n}^{\rm sym}(\bbC): \: Z^t \cdot \ov{Z}<I_n\}\subset M_{n,n}^{\rm sym}(\bbC):= \{ Z\in M_{n,n}(\bbC): \: Z^t=Z\}.
  \end{equation}
Let $\Sp(n,\bbC)$ be the connected complex Lie subgroup of $\SL(2n,\bbC)$ that leaves invariant the bilinear alternating form on $\bbC^{2n}\times \bbC^{2n}$ given by
$S(\un{x},\un{y})=x_1y_{n+1}+\cdots+x_{n}y_{2n}+\cdots - x_{n+1}y_1-\cdots-x_{2n}y_n$. The group $\Sp(n.\bbC)$ is simple of type $C_n$
and it is explicitly given in $n\times n$ block notation as
\begin{equation*}
\Sp(n,\bbC)=\left\{g\in \SL(2n,\bbC)\: : \: g^t J_n  g= J_n  \right\}=
\end{equation*}
\begin{equation*}
=\left\{\begin{pmatrix} A & B \\ C & D \end{pmatrix} \in \SL(2n,\bbC) \: : \:
\begin{aligned}
& A^t C= C^t A \\
& B^t D=D^t B \\
& A^tD-C^tB= I_n
\end{aligned}
\right\}.
\end{equation*}
Consider the non-compact real form $\Spnc(n)$ of $\Sp(n, \bbC)$ consisting of all the elements of $\Sp(n,\bbC)$ that leave invariant the bilinear Hermitian form on
$\bbC^{2n}\times \bbC^{2n}$ given by $-x_1\ov{y}_1-\cdots -x_n \ov{y}_n+x_{n+1}\ov{y}_{n+1}+\cdots+x_{2n}\ov{y}_{2n}$. Explicitly,
\begin{equation*}
\Spnc(n)=\Sp(n,\bbC)\cap \SU(n,n)=\left\{g\in \Sp(n,\bbC) : \: \ov{g}^t \begin{pmatrix} I_n & 0 \\ 0 & -I_n\end{pmatrix}  g= \begin{pmatrix} I_n & 0 \\ 0 & -I_n\end{pmatrix} \right\}=
\end{equation*}
\begin{equation*}
=\left\{\begin{pmatrix} A & B \\ \ov{B} & \ov{A} \end{pmatrix} \in \SL(2n,\bbC) : \:
\begin{aligned}
& \ov{A}^t A-B^t \ov{B}=I_n  \\
& \ov{A}^tB-B^t\ov{A}=0
\end{aligned}
\right\}.
\end{equation*}
Note that the Lie group $\Spnc(n)$ is isomorphic to the real symplectic group $\Sp(n, \bbR)$ via the conjugation inside $\Sp(n,\bbC)$ given by (see \cite[p. 71]{Mok})
\begin{equation}\label{E:isoIII}
\begin{aligned}
\Spnc(n)& \stackrel{\cong}{\longrightarrow}  \Sp(n,\bbR):=\left\{g\in \GL(2n,\bbR): \: g^t J_n g=J_n \right\}\\
h & \mapsto \begin{pmatrix}I_n & i I_n \\ iI_n & I_n  \end{pmatrix} h \begin{pmatrix}I_n & i I_n \\ iI_n & I_n  \end{pmatrix}^{-1}.
\end{aligned}
\end{equation}
The Lie group $\Spnc(n)$ acts transitively on $\calD_{III_{n}}$ via generalized M\"obius transformations, as in \eqref{E:Moebius}.
Notice that the center $Z(\Spnc(n))=\left\{\pm I_{2n} \right\}$ of $\Spnc(n)$
acts trivially on $\calD_{III_{n}}$; indeed, it turns out that the connected component of the group of biholomorphisms of $\calD_{III_{n}}$ is given by
\begin{equation*}
\Hol(\calD_{III_{n}})^o=\Spnc(n)/Z(\Spnc(n)):=\PSpnc(n),
\end{equation*}
which is the connected non-compact adjoint simple  Lie group of type $C_n$.

The symmetry of $\calD_{III_{n}}$ at the base point $0$ is given by the element
\begin{equation}\label{E:symIII}
s_0=\left[\begin{pmatrix} iI_n & 0 \\ 0 & -iI_n \end{pmatrix} \right]\in \PSpnc(n),
\end{equation}
which acts on $\calD_{III_{n}}$ by sending $Z$ into $-Z$. The symmetry $s_0$ induces an involution on $\Spnc(n)$
\begin{equation*}
\begin{aligned}
\sigma: \Spnc(n) & \longrightarrow \Spnc(n) \\
\begin{pmatrix} A & B \\ C & D \end{pmatrix} & \mapsto
\begin{pmatrix} i I_n & 0 \\ 0 & -i I_n \end{pmatrix}\begin{pmatrix} A & B \\ C & D \end{pmatrix}\begin{pmatrix}  i I_n & 0 \\ 0 & - i I_n \end{pmatrix}^{-1}=
\begin{pmatrix} A & -B \\ -C & D \end{pmatrix},
\end{aligned}
\end{equation*}
whose fixed Lie subgroup is equal to the maximal compact Lie subgroup
\begin{equation*}
\left\{\begin{pmatrix} A & 0 \\ 0 & D \end{pmatrix}\in \Spnc(2n) \right\}=\left\{\begin{pmatrix} A & 0 \\ 0 & \ov{A} \end{pmatrix}\: : \: \ov{A}^t A=I_n \right\}
=:\U(n),
\end{equation*}
which is also equal to the stabilizer of $0\in \calD_{III_{n}}$. In particular, the pair $(\Spnc(n), $ $ \U(n))$ is a Riemannian symmetric pair.
Notice that the involution $\sigma$ descends to an involution of $\PSpnc(n)$ whose fixed locus is the maximal compact  Lie subgroup
 $\ov{\U(n)}:=\U(n)/\{\pm I_n\}$ of $\PSpnc(n)$. Therefore, also the pair $(\PSpnc(n),\ov{\U(n)})$ is a Riemannian symmetric pair.

By the above discussion, we get the following presentation of $\calD_{III_{n}}$ as an irreducible \emph{HSM of non-compact type}
\begin{equation}\label{E:nocomIII}
\calD_{III_{n}}\cong \Spnc(n)/\U(n)=\PSpnc(n)/\ov{\U(n)},
\end{equation}
associated to the Riemannian symmetric pair $(\Spnc(n), \U(n))$ (resp. to $(\PSpnc(n),$ $\ov{\U(n)}$).
Notice that the last description of $\calD_{III_{n}}$ is the one appearing in Theorem \ref{T:Lie-irr}\eqref{T:Lie-irr1}.

The  irreducible \emph{Hermitian SLA of non-compact type} associated to the Riemannian symmetric pair  $(\Spnc(n), \U(n))$ (or equivalently to $(\PSpnc(n),$
$ \ov{U(n)}$) is given  by the Lie algebra
\begin{equation}\label{E:Lie-nocompIII}
\spnc(n):=\Lie \Spnc(n)=\left\{ M \in \sl(2n,\bbC) : \:
\begin{aligned}
& M^t S_n=-S_n M \\
& \ov{ M}^t \begin{pmatrix} -I_n & 0 \\ 0 & I_n \end{pmatrix} =-  \begin{pmatrix} -I_n & 0 \\ 0 & I_n \end{pmatrix} M\\
\end{aligned}
\right\}=
\end{equation}
\begin{equation*}
=\left\{ \begin{pmatrix}  Z_1 &Z_2 \\ \ov{Z}_2^t & - Z_1^t  \end{pmatrix}\in \gl(2n,\bbC)  : \:
 \ov{Z_1}^t=-Z_1, \: Z_2^t=Z_2\right\}\footnote{Note that $\spnc(n)$ is isomorphic to the classical real Lie algebra $\sp(n,\bbR)$ via the same conjugation given in formula \eqref{E:isoIII}}.
\end{equation*}
endowed with the Cartan involution $\theta= d\sigma$ given by
\begin{equation*}
\theta \begin{pmatrix}  Z_1 &Z_2 \\ \ov{Z}_2^t & - Z_1^t  \end{pmatrix}= \begin{pmatrix}  Z_1 &-Z_2 \\ -\ov{Z}_2^t & - Z_1^t  \end{pmatrix}
\end{equation*}
and with the element
\begin{equation*}
H=\frac{i}{2} \begin{pmatrix} I_n & 0\\ 0 & - I_n\end{pmatrix}\in \Fix(\theta)=\left\{ \begin{pmatrix}  Z_1 & 0 \\  0 & -Z_1^t  \end{pmatrix}: \ov{Z}_1^t=-Z_1\right\}\cong \Lie U(n).
\end{equation*}

The \textbf{cominuscle homogeneous variety } of type $III_{n}$ is the symplectic Grassmannian $\Gras_{\rm sym}(n,2n)$ parametrizing Lagrangian $n$-dimensional subspaces of $\bbC^{2n}$:
\begin{equation}\label{E:ComIII}
\Gras_{\rm sym}(n,2n):=\{[W\subset \bbC^{2n}]\: : \: \dim W=p, J_{|W\times W}\equiv 0\},
\end{equation}
where $J$ is the standard symplectic form on $\bbC^{2n}$ which is represented by the matrix $J_n=\begin{pmatrix} 0 & I_n \\ -I_n & 0 \end{pmatrix} $ in the standard basis of $\bbC^{2n}$.

The \emph{Borel embedding} of $\calD_{III_n}$ into $\Gras_{\rm sym}(n,2n)$ is given by
\begin{equation}\label{E:BorIII}
\begin{aligned}
\calD_{III_n}\subset M_{n,n}^{\rm sym}(\bbC) & \hookrightarrow \Gras_{\rm sym}(n,2n), \\
Z & \mapsto \left\{ \langle v_1,\cdots v_n \rangle : \: \{v_1,\cdots, v_n\} \text{ are the column vectors of } \: \begin{pmatrix} Z \\ I_n \end{pmatrix} \right\}.
\end{aligned}
\end{equation}
The complex simple algebraic group $\Sp(n,\bbC)$ acts transitively on $\Gras_{\rm sym}(n,2n)$ via
\begin{equation*}
\begin{aligned}
\Sp(n,\bbC)\times \Gras_{\rm sym}(n,2n) & \longrightarrow \Gras_{\rm sym}(n,2n) \\
(g, [W\subset \bbC^{2n}]) & \mapsto [g(W)\subset \bbC^{2n}].
\end{aligned}
\end{equation*}
Note that the center $Z(\Sp(n,\bbC))=\{\pm I_{2n} \} $ of $\Sp(n,\bbC)$
acts trivially on $\Gras_{\rm sym}(n,2n)$; indeed, it turns out that the group of automorphisms of the algebraic variety $\Gras_{\rm sym}(n,2n)$ is equal to
$$\PSp(n,\bbC):=\Sp(n,\bbC)/\{\pm I_{2n}\},$$
which is a connected semisimple complex algebraic group of adjoint type and it is the complexification of the Lie group $\PSp(n,\bbR)$.

Consider now the base point $W_o:=\langle e_{n+1}, \cdots, e_{2n}\rangle \in \Gras_{\rm sym}(n,2n)$ with respect to the standard basis $\{e_1,\cdots,e_{2n}\}$ of $\bbC^{2n}$
(recall that we have normalized $J$ so that it is represented by the standard symplectic matrix $J_n$ with respect to this basis).
The stabilizer of $W_o$ is the maximal parabolic subgroup associated to the $n$-th simple root of the Dinkin diagram $C_n$  (which is the unique cominuscle simple root of $C_n$,
see Table \ref{F:cominus})
\begin{equation*}
Q_n:=\left\{ \begin{pmatrix} A & 0 \\ C & D \end{pmatrix}\in \Sp(n,\bbC)\right\}\subset \Sp(n,\bbC),
\end{equation*}
where $A\in M_{n,n}(\bbC)$, $C\in M_{n,n}(\bbC)$ and $D\in M_{n,n}(\bbC)$. The parabolic group $Q_n$ admits the following Levi decomposition
\begin{equation*}
Q_n=R_u(Q_n)\rtimes L(Q_n):=\left\{ \begin{pmatrix} I_n & 0 \\ C & I_n \end{pmatrix} : \: C^t=C \right\} \rtimes \left\{ \begin{pmatrix} A & 0 \\ 0 & D \end{pmatrix}\in \Sp(n,\bbC)\right\},
\end{equation*}
which coincides with the Levi decomposition appearing in Theorem \ref{T:HSM-comin}.

From the above discussion, we obtain the following explicit presentation of $\Gras_{\rm sym}(n,2n)$ as a cominuscle homogeneous variety (as in Definition \ref{D:comin})
\begin{equation}\label{E:G/PIII}
\Gras_{\rm sym}(n,2n)\cong \Sp(n,\bbC)/Q_n=\PSp(n,\bbC)/\ov{Q_n},
\end{equation}
where $\ov{Q_n}:=Q_n/\{\pm I_{2n}\}$.

Consider now the compact real form of $\Sp(n,\bbC)$, which  is the Lie subgroup  $\Sp(n):=\Sp(n,\bbC)\cap \SU(2n) \subset \Sp(n,\bbC)$ that leaves invariant the positive definite Hermitian form
$x_1\ov{y}_1+\cdots+x_{2n}\ov{y}_{2n}$ on $\bbC^{2n}$. More explicitly\footnote{The Lie group $\Sp(n)$ admits another natural description in terms of matrices with coefficients in $\bbH$.
Namely, there an isomorphism of Lie group
\begin{equation*}
\begin{aligned}
\Sp(n) & \stackrel{\cong}{\longrightarrow} \U(n,\bbH):=\{g\in \GL(n,\bbH)\: : \: \ov{g}^tg=I_n\}\\
\begin{pmatrix} A & B \\ -\ov{B} & \ov{A} \end{pmatrix} & \mapsto A-j\ov{B}.
\end{aligned}
\end{equation*}
}

\begin{equation*}
\Sp(n)=\{g\in \Sp(n, \bbC)\: : \: \ov{g}^tg=I_{2n}\}=
\end{equation*}
\begin{equation*}
=\left\{\begin{pmatrix} A & B \\ -\ov{B} & \ov{A} \end{pmatrix} \in \SL(2n,\bbC) : \:
\begin{aligned}
& \ov{A}^t A+ \ov{B}^t B=I_n \\
& \ov{A}^tB= -B^t \ov{A}
\end{aligned}
\right\}.
\end{equation*}
Similarly, the quotient of $\Sp(n)$ by its center
\begin{equation*}
\PSp(n):=\Sp(n)/Z(\Sp(n))=\Sp(n)/\{\pm I_{2n} \}
\end{equation*}
is the compact real form of $\PSp(n,\bbC)$.

The restriction of the action of  $\Sp(n, \bbC)$  on $\Gras(q,p+q)$ to the subgroup $\Sp(n)\subset \Sp(n, \bbC)$ is still transitive and the stabilizer of $W_o$ is the maximal
proper connected and compact subgroup
\begin{equation*}
\Sp(n)\cap Q_n= \left\{\begin{pmatrix} A & 0 \\ 0 & A \end{pmatrix} : \: \ov{A}^tA=I_n  \right\}\cong \U(n).
\end{equation*}
The action of $\Sp(n)$ on $\Gras(q,p+q)$ factors through a transitive action of $\PSp(n)$ in such a way that the stabilizer of $W_o$ is equal to the maximal
proper connected and compact subgroup
$$\PSp(n)\cap \ov{Q_n}=\left\{\left[\begin{pmatrix} A & 0 \\ 0 & A \end{pmatrix}\right] : \: \ov{A}^tA=I_n  \right\} \cong \ov{\U(n)}=\U(n)/\{\pm I_n\}.
$$

The pair $(\Sp(n),\U(n))$ is a Riemannian symmetric pair since $\U(n)$ is the fixed subgroup of the involution
\begin{equation*}
\begin{aligned}
\sigma^*: \Sp(n) & \longrightarrow \Sp(n) \\
\begin{pmatrix} A & B \\ -\ov{B} & \ov{A} \end{pmatrix} & \mapsto \begin{pmatrix} A & -B \\ \ov{B} & \ov{A} \end{pmatrix},
\end{aligned}
\end{equation*}
and similarly for the pair $(\PSp(n),\ov{\U(n)})$.

By the above discussion, we get the following presentation of $\Gras_{\rm sym}(n,2n)$ as the irreducible \emph{HSM of compact type}
\begin{equation}\label{E:comIII}
\Gras_{\rm sym}(n,2n)\cong \Sp(n)/\U(n)=\PSp(n)/\ov{\U(n)},
\end{equation}
associated to the Riemannian symmetric pair $(\Sp(n), \U(n))$ (resp. to $(\PSp(n),\ov{\U(n)}$).
In particular, the last description of $\calD_{III_n}$ is the one appearing in Theorem \ref{T:Lie-irr}\eqref{T:Lie-irr2}.
Notice that the symmetry at the base point $W_o$ of $ \Gras_{\rm sym}(n,2n)$, seen as a Hermitian symmetric manifold, is given by the element
\begin{equation*}
s_{W_o}=\left[\begin{pmatrix}  iI_n & 0 \\ 0 & -iI_n \end{pmatrix} \right]\in \PSp(n).
\end{equation*}

The irreducible \emph{Hermitian SLA of compact type} associated to the Riemann symmetric pair $(\Sp(n),\U(n))$ is given by the Lie algebra
\begin{equation}\label{E:Lie-compIII}
\Lie \Sp(n)=\sp(n)=\left\{ \begin{pmatrix}  Z_1 & Z_2 \\  -\ov{Z_2}^t & -Z_1^t  \end{pmatrix}\: : \: \ov{Z_1}^t=-Z_1, \:\: Z_2^t=Z_2 \right\}
\end{equation}
endowed with the involution $\theta^*=d\sigma^*$
\begin{equation*}
\theta^* \begin{pmatrix}  Z_1 & Z_2 \\  -\ov{Z_2}^t & -Z_1^t  \end{pmatrix} =\begin{pmatrix}  Z_1 & -Z_2 \\  \ov{Z_2}^t & -Z_1^t  \end{pmatrix}
\end{equation*}
and with the element
\begin{equation*}
H^*=\frac{1}{2} \begin{pmatrix} I_n & 0\\ 0 & -I_n \end{pmatrix}\in \Fix(\theta^*)=\left\{ \begin{pmatrix}  Z_1 & 0 \\  0 & -Z_1^t  \end{pmatrix}: \: \ov{Z_1}^t=-Z_1 \right\}\cong \u(n)=\Lie \U(n)
\end{equation*}
Notice that the Hermitian SLA $(\sp(n),\theta^*,H^*)$ is the dual of the Hermitian SLA $(\spnc(n),\theta,H)$ in the sense of  \S\ref{SS:dual}.

The complexification of the Lie algebras $\spnc(n)$ and $\sp(n)$ is the complex simple Lie algebra of type $C_n$
\begin{equation*}
\Lie \Sp(n, \bbC)=\sp(n,\bbC)=\left\{ M \in \gl(2n,\bbC) \: : \:
M^t J_n =- J_n  M\right\}=
\end{equation*}
\begin{equation*}
=\left\{ \begin{pmatrix}  Z_1 & Z_2 \\  Z_3 & -Z_1^t  \end{pmatrix}\in \gl(2n,\bbC)\: : \: Z_2^t =Z_2 \text{ and } Z_3^t=Z_3  \right\}.
\end{equation*}
The decomposition \eqref{E:dec-g} of $\sp(n,\bbC)$ is given by
\begin{equation*}
\sp(n,\bbC)=\left\{ \begin{pmatrix}  Z_1 & 0 \\  0 & -Z_1^t  \end{pmatrix}\right\} \oplus \left\{ \begin{pmatrix}  0 & Z_2 \\  0 & 0  \end{pmatrix} : \: Z_2^t=Z_2 \right\} \oplus
 \left\{ \begin{pmatrix}  0 & 0 \\  Z_3 & 0  \end{pmatrix} : \: Z_3^t=Z_3 \right\}.
\end{equation*}
In particular, we have the identification
\begin{equation}\label{E:p+III}
\begin{aligned}
M_{n,n}^{\rm sym}(\bbC)& \stackrel{\cong}{\longrightarrow} \p_+\\
M & \mapsto \begin{pmatrix}  0 & M \\  0 & 0  \end{pmatrix}.
\end{aligned}
\end{equation}
Using the above identification and the formula \eqref{E:Jortriple}, $M_{n,n}^{\rm sym}(\bbC)$ becomes a \emph{Hermitian positive JTS} with respect to the triple product
\begin{equation}\label{E:JorIII}
 \{M_1,M_2,M_3\}=\frac{1}{2}\left(M_1\ov{M_2}^tM_3+M_3\ov{M_2}^t M_1\right).
\end{equation}

\subsection{Type $IV_n$}

The \textbf{bounded symmetric domain} of type $IV_{n}$ ($1\leq n\neq 2$) in its \emph{Harish-Chandra embedding} is given by
 \begin{equation}\label{E:BSDIV}
 \calD_{IV_{n}}:= \{Z\in \bbC^n : \: 2\ov{Z}^tZ<1+|Z^tZ|^2, \: \ov{Z}^tZ<1\}  \subset \bbC^n.
  \end{equation}
Note that the first inequality, together with the fact that  $ |Z^tZ|^2\leq (\ov{Z}^tZ)^2$,  implies that
$$2\ov{Z}^tZ<1+|Z^tZ|^2 \Longrightarrow 2\ov{Z}^tZ<1+(\ov{Z}^tZ)^2 \Longleftrightarrow 0<(1-\ov{Z}^tZ)^2.$$
Therefore the open subset $\{ Z\in \bbC^n : \: 2\ov{Z}^tZ<1+|Z^tZ|^2\}\subset \bbC^n$ is the disjoint union of two connected components defined by, respectively, $\ov{Z}^tZ<1$ and
by $\ov{Z}^tZ>1$. The first connected component is the one that contains the origin $0\in \bbC^n$ and it coincides with the domain $\calD_{IV_n}$.

The domain $\calD_{IV_n}$ (which is also called the Lie ball) admits another real analytic incarnation in terms of $2\times n$ real matrices, namely we have a real analytic
diffeomorphism (see \cite[Sec. 12 and 13]{Hua})
\begin{equation}\label{E:BSDIVreal}
\begin{aligned}
\calD_{IV_n} & \stackrel{\cong}{\longrightarrow} \{M\in M_{2,n}(\bbR):M\cdot M^t<I_2\}\\
Z & \mapsto 2 \begin{pmatrix}Z^tZ+1 & i(Z^tZ-1) \\\ov{Z^tZ}+1 & -i(\ov{Z^tZ}-1) \end{pmatrix}^{-1}\cdot \begin{pmatrix}Z \\ \ov{Z} \end{pmatrix}.
\end{aligned}
\end{equation}


Consider the subgroup $\SOnc(n,2)$ of $\SO(n+2, \bbC)$ consisting of all the elements of $\SO(2+n,\bbC)$ that leave invariant the bilinear Hermitian form on
$\bbC^{n+2}\times \bbC^{2+n}$ given by $-x_1\ov{y}_1-\cdots -x_n \ov{y}_n+x_{n+1}\ov{y}_{n+1}+x_{n+2}\ov{y}_{n+2}$. Explicitly,
\begin{equation*}
\SOnc(n,2)=\SO(n+2,\bbC)\cap \U(n,2)=
\end{equation*}
\begin{equation*}
=\left\{g\in \SL(n+2,\bbC) : \: g^tg=I_{n+2}, \: \ov{g}^t \begin{pmatrix} I_n & 0 \\ 0 & -I_2\end{pmatrix}  g= \begin{pmatrix} I_n & 0 \\ 0 & -I_2\end{pmatrix} \right\}=
\end{equation*}
\begin{equation*}
=\left\{\begin{pmatrix} A & B \\ C & D \end{pmatrix} \in \SL(n+2,\bbC) : \:
\begin{aligned}
& A^t A+C^t C=I_n , & \ov{A}^t A- \ov{C}^t C=I_n \\
&  D^t D+B^t B=I_2 ,  &  \ov{D}^t D- \ov{B}^t B=I_2\\
& A^tB= -C^t D , & \ov{A}^tB= \ov{C}^t D
\end{aligned}
\right\}
\end{equation*}

The Lie group $\SOnc(n,2)$ acts transitively on $\calD_{IV_{n}}$ via
\begin{equation*}
\begin{aligned}
\SOnc(n,2)\times \calD_{IV_{n}} & \longrightarrow \calD_{IV_{n}} \\
\left( \begin{pmatrix} A & B \\ C & D \end{pmatrix} , Z\right) & \mapsto
\frac{2iAZ+B\begin{pmatrix}1+Z^t\cdot Z \\ i-iZ^t\cdot ZÊ\end{pmatrix}}{(1, i)\cdot \left(2i CZ+D\begin{pmatrix}1+Z^t\cdot Z \\ i-iZ^t\cdot ZÊ\end{pmatrix} \right) }.
\end{aligned}
\end{equation*}

Notice that the center $Z(\SOnc(n,2))=\left\{\pm I_{n+2} \right\}$ of $\SOnc(n,2)$
acts trivially on $\calD_{IV_{n}}$; indeed, it turns out that the connected component of the group of biholomorphisms of $\calD_{IV_{n}}$ is given by
\begin{equation*}
\Hol(\calD_{IV_{n}})^o=\SOnc(n,2)/ Z(\SOnc(n,2)):=\PSOnc(n,2),
\end{equation*}
which is the connected non-compact adjoint simple  Lie group of type $D_{n/2+1}$ if $n$ is even and $B_{(n+1)/2}$ if $n$ is odd.

The symmetry of $\calD_{IV_{n}}$ at the base point $0$ is given by the element
\begin{equation}\label{E:symIV}
s_0=\left[\begin{pmatrix} I_n & 0 \\ 0 & -I_2 \end{pmatrix} \right]\in \PSOnc(n,2),
\end{equation}
which acts on $\calD_{IV_{n}}$ by sending $Z$ into $-Z$. The symmetry $s_0$ induces an involution on $\SOnc(n,2)$
\begin{equation*}
\begin{aligned}
\sigma: \SOnc(n,2) & \longrightarrow \SOnc(n,2) \\
\begin{pmatrix} A & B \\ C & D \end{pmatrix} & \mapsto
\begin{pmatrix} I_n & 0 \\ 0 & - I_2 \end{pmatrix}\begin{pmatrix} A & B \\ C & D \end{pmatrix}\begin{pmatrix}  I_n & 0 \\ 0 & -  I_2 \end{pmatrix}^{-1}=
\begin{pmatrix} A & -B \\ -C & D \end{pmatrix},
\end{aligned}
\end{equation*}
whose fixed Lie subgroup is equal to the maximal compact Lie subgroup
\begin{equation*}
\left\{\begin{pmatrix} A & 0 \\ 0 & D \end{pmatrix}\in \SOnc(n,2) \right\}=\left\{\begin{pmatrix} A & 0 \\ 0 & D \end{pmatrix}\: : \:
\begin{aligned} \ov{A}^t A=A^t A=I_n \\  \ov{D}^t D=D^t D=I_2 \end{aligned}Ê\right\}
=\SO(n)\times \SO(2),
\end{equation*}
which is also equal to the stabilizer of $0\in \calD_{IV_{n}}$. In particular, the pair $(\SOnc(n,2), $ $ \SO(n)\times \SO(2))$ is a Riemannian symmetric pair.
Notice that the involution $\sigma$ descends to an involution of $\PSOnc(n,2)$ whose fixed locus is the maximal compact  Lie subgroup
 $\ov{\SO(n)\times \SO(2)}:=\SO(n)\times \SO(2)/Z(\SOnc(n,2))$ of $\PSOnc(n,2)$. Therefore, also the pair $(\PSOnc(n,2),\ov{\SO(n)\times \SO(2)})$ is a Riemannian symmetric pair.

By the above discussion, we get the following presentation of $\calD_{IV_{n}}$ as an irreducible \emph{HSM of non-compact type}
\begin{equation}\label{E:nocomIV}
\calD_{IV_{n}}\cong \SOnc(n,2)/(\SO(n)\times \SO(2))=\PSOnc(n,2)/\ov{\SO(n)\times \SO(2)},
\end{equation}
associated to the Riemannian symmetric pair $(\SOnc(n,2), \SO(n)\times \SO(2))$ (resp. to $(\PSOnc(n,2),$ $\ov{\SO(n)\times \SO(2)}$).
Notice that the last description of $\calD_{IV_{n}}$ is the one appearing in Theorem \ref{T:Lie-irr}\eqref{T:Lie-irr1}.

The  irreducible \emph{Hermitian SLA of non-compact type} associated to the Riemannian symmetric pair  $(\SOnc(n,2), \SO(n)\times \SO(2))$ (or equivalently to $(\PSOnc(n,2),$
$ \ov{\SO(n)\times \SO(2)}$) is given  by the Lie algebra
\begin{equation}\label{E:Lie-nocompIV}
\sonc(n,2):=\Lie \SOnc(n,2)=\left\{ M \in \sl(n+2,\bbC) : \:
\begin{aligned}
& M^t=-M \\
& \ov{ M}^t \begin{pmatrix} I_n & 0 \\ 0 & -I_2 \end{pmatrix} = - \begin{pmatrix} I_n & 0 \\ 0 & -I_2 \end{pmatrix} M\\
\end{aligned}
\right\}=
\end{equation}
\begin{equation*}
=\left\{ \begin{pmatrix}  X_1 &iX_2 \\ -iX_2^t & X_3 \end{pmatrix}\in \gl(n+2,\bbC)  : \:
\begin{aligned}
& \ov{X}_1=X_1,\: \ov{X}_2=X_2,\: \ov{X}_3=X_3\\
& X_1^t=-X_1, \: X_3^t=-X_3
 \end{aligned}
 \right\}
\end{equation*}
endowed with the involution $\theta=d\sigma$
\begin{equation*}
\theta \begin{pmatrix}  X_1 &iX_2 \\ -iX_2^t & X_3 \end{pmatrix}=\begin{pmatrix}  X_1 &-iX_2 \\ iX_2^t & X_3 \end{pmatrix}
\end{equation*}
and with the element
\begin{equation*}
H= \begin{pmatrix} 0 & 0\\ 0 & J_1 \end{pmatrix}\in \Fix(\theta)=\left\{ \begin{pmatrix}  X_1 & 0 \\  0 & X_3  \end{pmatrix}\in \gl(n+2,\bbR): X_1^t=-X_1, \: X_3^t=-X_3\right\}\cong \Lie (\SO(n)\times \SO(2)).
\end{equation*}

The \textbf{cominuscle homogeneous variety } of type $IV_{n}$ is the complex quadric hypersurface of dimension $n$:
\begin{equation}\label{E:ComIV}
\calQ^n:=\{[v]\in \bbP^{n+1}: \: Q(v,v)=0\} \subset \bbP^{n+1}
\end{equation}
where $Q$ is the bilinear symmetric non-degenerate form on $\bbC^{n+2}$ given by $Q(v)=v_1^2+\ldots+v_{n+2}^2$.

Observe that the complex quadric hypersurface $\calQ^n\subset \bbP^{n+1}$ admits another real analytic incarnation. Namely, $\calQ^n$ is real analytic diffeomorphic to the oriented real Grassmannian $\Gras_{\bbR}^+(2,n+2)$ parametrizing $2$-dimensional oriented subspaces of $\bbR^{n+2}$ via the map
(see \cite[Appendix \S6]{Sat})
\begin{equation}\label{E:ComIVreal}
\begin{aligned}
\Gras_{\bbR}^+(2,n+2) & \stackrel{\cong}{\longrightarrow} \calQ^n \\
\langle v_1,v_2\rangle & \mapsto [v_1+iv_2].
\end{aligned}
\end{equation}

The \emph{Borel embedding} of $\calD_{IV_n}$ into $\calQ^n$ is given by
\begin{equation}\label{E:BorIV}
\begin{aligned}
\calD_{IV_n}\subset \bbC^n & \hookrightarrow \calQ^n\\
Z & \mapsto \left[ \begin{pmatrix}  2iZ \\ 1+Z^t\cdot Z \\ i-i Z^t\cdot Z \end{pmatrix}  \right].
\end{aligned}
\end{equation}
The complex algebraic simple group $\SO(n+2,\bbC)$ acts transitively on $\calQ^n$ via
\begin{equation*}
\begin{aligned}
\SO(n+2,\bbC)\times \calQ^n & \longrightarrow \calQ^n \\
(g, [v]) & \mapsto [g(v)].
\end{aligned}
\end{equation*}
Note that the center
\begin{equation}\label{E:centerIV}
Z(\SO(n+2,\bbC))=
\begin{cases}
\{\pm I_{n+2} \} & \text{ if } n \: \text{ is even,} \\
\{I_{n+2}\} & \text{ if } n \: \text{ is odd,}
\end{cases}
\end{equation}
acts trivially on $\calQ^n$; indeed, it turns out that the group of automorphisms of the algebraic variety $\calQ^n$ is equal to
$$\PSO(2n,\bbC):=\PSO(2n,\bbC)/Z(\SO(n+2,\bbC)),$$
which is the connected simple adjoint complex algebraic group of type $D_{n/2+1}$ if $n$ is even and $B_{(n+1)/2}$ if $n$ is odd.

The stabilizer of $v_o=[(0,\cdots, 0, 1, i)] \in \calQ^n$ is the maximal parabolic subgroup associated to the first simple root of the Dinkin diagram $D_{n/2+1}$ if $n$ is even and  of the Dinkin
diagram $B_{(n+1)/2}$ if $n$ is odd (which are cominuscle simple roots, see Table \ref{F:cominus})
\begin{equation*}
Q_1:=\left\{ \begin{pmatrix} A & B \\ C & D \end{pmatrix}\in \SO(n+2,\bbC)\: :
\begin{aligned}
& B=(B',iB') \text{ for some } B'\in M_{n,1}(\bbC) \\
& D=\begin{pmatrix} a & b \\ c & d \end{pmatrix}\text{ such that } i a-b=c+i d
\end{aligned}
\right\},
\end{equation*}
where $A\in M_{n,n}(\bbC)$, $B\in M_{n,2}(\bbC)$, $C\in M_{2,n}(\bbC)$ and $D\in M_{2,2}(\bbC)$.

From the above discussion, we obtain the following explicit presentation of $\calQ^n$ as a cominuscle homogeneous variety (as in Definition \ref{D:comin})
\begin{equation}\label{E:G/PIV}
\calQ^n\cong \SO(n+2,\bbC)/Q_1=\PSO(n,\bbC)/\ov{P}_1,
\end{equation}
where $\ov{P}_1:=Q_1/Z(\SO(n+2,\bbC))$.

Consider now the compact real form $\SO(n+2)$ of $\SO(n+2,\bbC)$ consisting of all the real matrices in $\SO(n+2,\bbC)$ or, equivalently, of all the elements in $\SO(n+2,\bbC)$
that leave invariant the positive definite Hermitian form
$x_1\ov{y}_1+\cdots+x_{n+2}\ov{y}_{n+2}$ on $\bbC^{n+2}$. More explicitly
\begin{equation*}
\SO(n+2):=\SO(n+2,\bbC)\cap \SU(n+2)=\left\{g\in \SO(n+2,\bbC) : \: \ov{g}=   g  \right\}=
\end{equation*}
\begin{equation*}
=\left\{\begin{pmatrix} A & B \\ C & D \end{pmatrix} \in \SL(n+2,\bbR) : \:
\begin{aligned}
& A^t A+C^t C=I_n  \\
&  D^t D+B^t B=I_2 \\
& A^tB= -C^t D
\end{aligned}
\right\}.
\end{equation*}
Similarly, the quotient of $\SO(n+2)$ by its center (which is given by \eqref{E:centerIV})
\begin{equation*}
\PSO(n+2):=\SO(n+2)/Z(\SO(n+2))
\end{equation*}
is a compact real form of $\PSO(2n,\bbC)$.

The restriction of the action of  $\SO(n+2,\bbC)$  on $\calQ^n$ to the subgroup $\SO(n+2)\subset \SO(n+2,\bbC)$ is still transitive and the stabilizer of $v_o$ is the maximal
proper connected and compact subgroup
\begin{equation*}
\SO(n+2)\cap Q_1= \left\{\begin{pmatrix} A & 0 \\ 0 & D \end{pmatrix}:
\begin{aligned}
A^tA=I_n, &\: \det(A)=1 \\
D^t D=I_2, &\: \det(D)=1
\end{aligned}
\right\}\cong \SO(n)\times \SO(2).
\end{equation*}

The pair $(\SO(n+2),\SO(n)\times \SO(2))$ is a Riemannian symmetric pair since $\SO(n)\times \SO(2)$ is the connected component of  the fixed subgroup of the involution
\begin{equation*}
\begin{aligned}
\sigma^*: \SO(n+2) & \longrightarrow \SO(n+2) \\
\begin{pmatrix} A & B \\ C & D \end{pmatrix} & \mapsto \begin{pmatrix} A & -B \\ -C & D \end{pmatrix},
\end{aligned}
\end{equation*}
and similarly for the pair $(\PSO(n+2),\ov{\SO(n)\times \SO(2)})$, where $\ov{\SO(n)\times \SO(2)}$ is the image of $\SO(n)\times SO(2)$ in $\PSO(n+2)$.

By the above discussion, we get the following presentation of $\calQ^n$ as the irreducible \emph{HSM of compact type}
\begin{equation}\label{E:comIV}
\calQ^n\cong \SO(n+2)/\SO(n)\times \SO(2)=\PSO(n+2)/\ov{\SO(n)\times \SO(2)},
\end{equation}
associated to the Riemannian symmetric pair $(\SO(n+2), \SO(n)\times \SO(2))$ (resp. to $(\PSO(n+2),$ $ \ov{\SO(n)\times \SO(2)}$).
In particular, the last description of $\calQ^n$ is the one appearing in Theorem \ref{T:Lie-irr}\eqref{T:Lie-irr2}.
Notice that the symmetry at the base point $v_o$ of $ \calQ^n$, seen as a Hermitian symmetric manifold, is given by the element
\begin{equation*}
s_{v_o}=\left[\begin{pmatrix} I_n & 0 \\ 0 & -I_2 \end{pmatrix} \right]\in \PSO(n+2).
\end{equation*}

The irreducible \emph{Hermitian SLA of compact type} associated to the Riemann symmetric pair $(\SO(n+2),\SO(n)\times \SO(2))$ is given by the Lie algebra
\begin{equation}\label{E:Lie-compIV}
\so(n+2):=\Lie \SO(n+2)=\left\{ M \in \sl(p+q,\bbR) : \: M^t=-M\right\}=
\end{equation}
\begin{equation*}
=\left\{ \begin{pmatrix}  X_1 &X_2 \\ -X_2^t & X_3 \end{pmatrix}\in \gl(p+q,\bbR)  : \:
 X_1^t=-X_1, \: X_3^t=-X_3\right\}
\end{equation*}
endowed with the involution $\theta^*=d\sigma^*$
\begin{equation*}
\theta^* \begin{pmatrix}  X_1 &X_2 \\ -X_2^t & X_3 \end{pmatrix}=\begin{pmatrix}  X_1 &-X_2 \\ X_2^t & X_3 \end{pmatrix}
\end{equation*}
and with the element
\begin{equation*}
H^*= \begin{pmatrix} 0 & 0\\ 0 & J_1 \end{pmatrix}\in \Fix(\theta^*)=\left\{ \begin{pmatrix}  X_1 & 0 \\  0 & X_3  \end{pmatrix}: X_1^t=-X_1, \: X_3^t=-X_3\right\}\cong \Lie (\SO(n)\times \SO(2)).
\end{equation*}
Notice that the Hermitian SLA $(\so(n+2),\theta^*,H^*)$ is the dual of the Hermitian SLA $(\sonc(n,2),\theta,H)$ in the sense of  \S\ref{SS:dual}.

The complexification of the Lie algebras $\sonc(n,2)$ and $\so(n+2)$ is the complex simple Lie algebra of type $D_{n/2+1}$ if $n$ is even and
 $B_{(n+1)/2}$ if $n$ is odd:
\begin{equation*}
\Lie \SO(n+2,\bbC)=\so(n+2,\bbC)=\left\{ \begin{pmatrix}  Z_1 & Z_2 \\  -Z_2^t & Z_3  \end{pmatrix}\in \gl(n+2,\bbC) \: : \:  Z_1^t=-Z_1, \: Z_3^t=-Z_3 \right\}.
\end{equation*}
The decomposition \eqref{E:dec-g} of $\so(n+2,\bbC)$ is given by
\begin{equation*}
\so(n+2,\bbC)=\left\{ \begin{pmatrix}  Z_1 & 0 \\  0 & Z_3  \end{pmatrix}:  \begin{aligned}Ê & Z_1^t=-Z_1 \\  & Z_3^t=-Z_3 \end{aligned}Ê\right\} \oplus
\left\{ \begin{pmatrix}  0 & (iZ',Z') \\  -(iZ',Z')^t & 0  \end{pmatrix} \right\} \oplus  \left\{ \begin{pmatrix}  0 & (Z'',iZ'') \\  -(Z'',iZ'')^t & 0  \end{pmatrix} \right\} .
\end{equation*}
In particular, we have the identification
\begin{equation}\label{E:p+IV}
\begin{aligned}
\bbC^n & \stackrel{\cong}{\longrightarrow} \p_+\\
Z & \mapsto \begin{pmatrix}  0 & (iZ,Z) \\  -(iZ,Z)^t & 0  \end{pmatrix}.
\end{aligned}
\end{equation}
Using the above identification and the formula \eqref{E:Jortriple}, $\bbC^{n}$ becomes a \emph{Hermitian positive JTS} with respect to the triple product
\begin{equation}\label{E:JorIV}
\{X,Y,Z\}=(X^t\cdot Z)\ov{Y}-(Z^t\cdot \ov{Y})X-(X^t\cdot \ov{Y})Z.
\end{equation}

\subsection{Type $VI $}\label{S:TypeVI}

Let $\bbO$ be the $\bbR$-algebra of octonions or Cayley algebra (we refer the reader to \cite{Bae} for a beautiful introduction to the octonions).
 Recall that $\bbO$ is the alternative $\bbR$-algebra (neither associative nor commutative)  of dimension $8$ whose underlying vector space
is equal to $\bbH\times \bbH$ and whose multiplication is equal to
\begin{equation*}
(a_1,b_1)\cdot (a_2,b_2):=(a_1a_2-b_2\wt{b_1}, \wt{a_1}b_2+a_2b_1),
\end{equation*}
where $\bbH$ is the division $\bbR$-algebra of quaternions and $\,\wt{}\,$ denotes its involution
$$
\begin{aligned}
\,\wt{}\,: \bbH & \longrightarrow \bbH \\
x_0+ix_1+jx_2+kx_3 & \mapsto x_0-ix_1-jx_2-kx_3.
\end{aligned}
$$
The algebra $\bbO$ is endowed with the unity element $e=(1,0)$ and with an involutive anti-automorphism
\begin{equation*}
(a,b)  \mapsto \wt{(a,b)}:=(\wt{a},-b).
\end{equation*}
The above involution gives rise to a norm
\begin{equation*}
\begin{aligned}
|.|^2: \bbO & \longrightarrow \bbR \\
(a,b) & \mapsto |(a,b)|^2:=(a,b)\cdot \wt{(a,b)}=a\wt{a}+b\wt{b}
\end{aligned}
\end{equation*}
which is a positive define quadratic form and it is multiplicative (i.e. $|(a_1,b_1)\cdot (a_2, b_2)|^2=|(a_1,b_1)|^2|(a_2,b_2)|^2$).
Therefore the pair $(\bbO, |.|^2)$ is a Euclidean composition algebra of dimension $8$ and indeed it is the unique such algebra.
We will denote by $\langle ,\rangle$ the bilinear form associated to the quadratic form $|.|^2$, i.e.
$$\langle x,y\rangle:=|x+y|^2-|x|^2-|y|^2,$$
for any $x,y\in \bbO$.

Let $\bbO_{\bbC}:=\bbO\otimes_{\bbR} \bbC$ be the  complexification of $\bbO$ (it is called the complex Cayley algebra). The involution $\wt{}$ and the quadratic form $|.|^2$ on $\bbO$ extend naturally
on  $\bbO_{\bbC}$ (by a slight abuse of notation, we will continue to denote them by the same symbols). Moreover, $\bbO_{\bbC}$ is endowed with a complex conjugation with respect to its real form $\bbO$:
\begin{equation*}
\lambda\otimes x  \mapsto \ov{\lambda\otimes x}:=\ov{\lambda}\otimes x,
\end{equation*}
where $\lambda\in \bbC$ and $x\in \bbO$.

Consider the complex vector space $H_3(\bbO_{\bbC})$ consisting of Hermitian $3\times 3$-matrices with entries in $\bbO_{\bbC}$
\begin{equation}\label{E:H3}
H_3(\bbO_{\bbC}):=\left\{ a\in M_{3,3}(\bbO_{\bbC})\,:\, \wt{a}^t=a\right\}=
\end{equation}
\begin{equation*}
=\left\{\begin{pmatrix}\alpha_1 & a_3 & \wt{a_2} \\ \wt{a_3} & \alpha_2 & a_1\\ a_2 & \wt{a_1} & \alpha_3 \end{pmatrix}: \alpha_1,\alpha_2,\alpha_3\in \bbC;\, a_1,a_2,a_3\in \bbO_{\bbC} \right\}.
\end{equation*}
The complex vector space $H_3(\bbO_{\bbC})$ is endowed with a product (called the Freudenthal product) defined by
\begin{equation}\label{E:Freud}
a\times b:=\begin{pmatrix}\alpha_1 & a_3 & \wt{a_2} \\ \wt{a_3} & \alpha_2 & a_1\\ a_2 & \wt{a_1} & \alpha_3 \end{pmatrix}\times
\begin{pmatrix}\beta_1 & b_3 & \wt{b_2} \\ \wt{b_3} & \beta_2 & b_1\\ b_2 & \wt{b_1} & \beta_3 \end{pmatrix}:=
\end{equation}
\begin{equation*}
=\begin{pmatrix}
\alpha_2\beta_3+\alpha_3\beta_2-\langle a_1,b_1\rangle & a_1b_2+b_1a_2-\alpha_3\wt{b_3}-\beta_3\wt{a_3} &  \wt{b_1}\wt{a_3}+\wt{a_1}\wt{b_3}-\alpha_2b_2-\beta_2a_2 \\
\wt{b_2}\wt{a_1} +\wt{a_2} \wt{b_1} -\alpha_3 b_3-\beta_3 a_3 & \alpha_3\beta_1+\alpha_1\beta_3-\langle a_2,b_2\rangle & a_2b_3+b_2a_3-\alpha_1\wt{b_1}-\beta_1\wt{a_1} \\
a_3b_1+b_3a_1-\alpha_2\wt{b_2}-\beta_2\wt{a_2} &\wt{b_3} \wt{a_2} +\wt{a_3} \wt{b_2} -\alpha_1 b_1-\beta_1a_1  & \alpha_1\beta_2+\alpha_2\beta_1-\langle a_3,b_3\rangle \end{pmatrix}.
\end{equation*}
Moreover, $H_3(\bbO_{\bbC})$ is endowed with a positive definite Hermitian form defined by
\begin{equation}\label{E:posHerm}
\left(a | b \right):=\sum_{i=1}^3\alpha_i \ov{\beta_i} +\sum_{j=1}^3 \langle a_j, \ov{b_j}\rangle,
\end{equation}
where $a, b\in H_3(\bbO_{\bbC})$ are written as in \eqref{E:Freud}.
Using the Freudenthal product and the above positive define Hermitian form, we can define a Jordan triple product on $H_3(\bbO_{\bbC})$ via
\begin{equation}\label{E:JorVI}
\{a,b,c\}:=(a|b)c+(c|b)a-(a\times c)\times \ov{b},
\end{equation}
where $\ov{b}$ is the element of $H_3(\bbO_{\bbC})$ obtained by conjugating all the entries of $b$ with respect to the complex conjugation of $\bbO_{\bbC}$.
The pair $(H_3(\bbO_{\bbC}),\{.\,,.\,,.\})$ is an irreducible  Hermitian positive JTS of dimension $27$ (see \cite[Sec. 2.2]{Roos}), called sometimes the exceptional Hermitian positive JTS of dimension $27$ or the
\textbf{Hermitian positive JTS of type VI}.

From the above explicit description of the Hermitian positive JTS $(H_3(\bbO_{\bbC}),\{.\,,.\,,.\})$ and formula  \eqref{E:HC-JTS}, we can deduce an explicit  expression of the associated bounded symmetric domain in its
Harish-Chandra embedding. In order to do that, we need to introduce the determinant and the adjoint of an element of $H_3(\bbO_{\bbC})$.
The determinant is defined by
\begin{equation}\label{E:detVI}
\begin{aligned}
\det: H_3(\bbO_{\bbC}) & \longrightarrow \bbO_{\bbC}, \\
a & \mapsto \frac{1}{3!}(a\times a|\ov{a})=\alpha_1\alpha_2\alpha_3-\sum_{i=1}^3\alpha_i |a_i|^2+a_1(a_2a_3)+(\wt{a_3}\wt{a_2})a_1,
\end{aligned}
\end{equation}
where $a\in H_3(\bbO_{\bbC})$ is written as in \eqref{E:Freud}. The adjoint of $a\in H_3(\bbO_{\bbC})$ is defined by
\begin{equation}\label{E:adjVI}
(a)^{\sharp}:=\frac{a\times a}{2}.
\end{equation}
The relation between the determinant and the adjoint is given by the following formulas (see \cite[Sec. 2.1]{Roos})
\begin{equation}\label{E:detadj}
\begin{sis}
& (x^{\sharp}|\ov{x})=3 \det(x), \\
& (x^{\sharp})^{\sharp}=\det(x) x.
\end{sis}
\end{equation}

The \textbf{bounded symmetric domain} of type $VI$ in its \emph{Harish-Chandra embedding} is given by (see \cite[Sec. 3.1]{Roos})
 \begin{equation}\label{E:BSDVI}
 \calD_{VI}:= \left\{a\in H_3(\bbO_{\bbC}): \:
 \begin{aligned}
& 1-(a|a)+(a^{\sharp}|a^{\sharp})-|\det(a)|^2>0\\
& 3- 2(a|a) +(a^{\sharp}|a^{\sharp})>0\\
& 3-(a|a)>0\\
\end{aligned}\right\}
\subset H_3(\bbO_{\bbC}).
   \end{equation}

The \textbf{cominuscle homogeneous variety } of type $VI$ is the Freudenthal variety
\begin{equation}\label{E:ComVI}
\calF:=\left\{[\lambda,x,y,\mu]\in \bbP(\bbC\oplus H_3(\bbO_{\bbC})\oplus H_3(\bbO_{\bbC})\oplus \bbC): y^{\sharp}=\mu x, \, x^{\sharp}=\lambda y, \, (x|\ov{y})=3\lambda\mu \right\}.
\end{equation}

The \emph{Borel embedding} of $\calD_{VI}$ into $\calF$ is given by
\begin{equation}\label{E:BorVI}
\begin{aligned}
\calD_{VI}\subset H_3(\bbO_{\bbC}) & \stackrel{j}{\hookrightarrow} \calF\\
x & \mapsto \left[ (1,x, x^{\sharp}, \det(x)) \right].
\end{aligned}
\end{equation}
Using the relations \eqref{E:detadj}, it is easy to see that $j$ is an open embedding and that  $j( H_3(\bbO_{\bbC}))$ is the Zariski open subset of $\calF$ defined by $\{\lambda\neq 0\}$.

\subsection{Type $V$}\label{S:TypeV}

In this subsection, we are going to use the notation introduced in \S\ref{S:TypeVI}.

The \textbf{Hermitian positive JTS of type V} (sometimes also called the exceptional Hermitian positive JTS of dimension $16$) is the simple Hermitian positive JTS $(\bbO_{\bbC}^2,\{.\,,.\,,.\})$ where the Jordan triple product $\{.\,,.\,,.\}$ is defined by

\begin{equation}\label{E:JorV}
\left\{\begin{pmatrix} a_1 \\ a_2 \end{pmatrix},\begin{pmatrix} b_1 \\ b_2 \end{pmatrix}, \begin{pmatrix} c_1 \\ c_2 \end{pmatrix}\right\}:=
\begin{pmatrix} (a_1\wt{\ov{b_1}})c_1+(c_1\wt{\ov{b_1}})a_1+(a_1\ov{b_2})\wt{c_2}+(c_1\ov{b_2})\wt{a_2} \\
\wt{a_1}(\ov{b_1}c_2)+\wt{c_1}(\ov{b_1}a_2)+\wt{a_2}(\ov{b_2}c_2)+\wt{c_2}(\ov{b_2}a_2) \end{pmatrix}.
\end{equation}

The \textbf{bounded symmetric domain} of type $V$ in its \emph{Harish-Chandra embedding} is given by (see \cite[Sec. 3.1]{Roos})
 \begin{equation}\label{E:BSDV}
 \calD_{V}:= \left\{x=\begin{pmatrix}x_1 \\ x_2\end{pmatrix} \in \bbO_{\bbC}^2: \:
 \begin{aligned}
& 1-\sum_{i=1}^2 \langle x_i, \ov{x_i} \rangle +\sum_{i=1}^2 (|x_i|^2)^2+\langle x_2x_3,\ov{x_2x_3}\rangle >0\\
& 2-\sum_{i=1}^2 \langle x_i, \ov{x_i} \rangle  >0\\
\end{aligned}\right\}
\subset \bbO_{\bbC}^2.
   \end{equation}


The \textbf{cominuscle homogeneous variety } of type $V$ is the Cayley plane
\begin{equation}\label{E:ComV}
\bbP^2_{\bbO}:=\left\{[a]\in \bbP(H_3(\bbO_{\bbC})):  a^{\sharp}=0  \right\}=
\end{equation}
\begin{equation*}
\left\{[a]\in \bbP(H_3(\bbO_{\bbC})):  \:
\begin{aligned}
 \alpha_2\alpha_3=|a_1|^2, \: \: & \alpha_3\alpha_1=|a_2|^2,  &  \alpha_1\alpha_2=|a_3|^3\\
a_1a_2=\alpha_3\wt{a_3}, \: \: & a_2a_3=\alpha_1\wt{a_1}, & a_3a_1=\alpha_2\wt{a_2}
\end{aligned}
  \right\},
\end{equation*}
where $a\in H_3(\bbO_{\bbC})$ is written as in \eqref{E:H3}.

The Cayley plane $\bbP^2_{\bbO}$ is homogeneous with respect to the natural action of
the subgroup $\SL_3(\bbO_{\bbC})\subset \GL(H_3(\bbO_{\bbC}))=\GL_{27}(\bbC)$ consisting of the elements preserving the determinant \eqref{E:detVI}
(see \cite[Sec. 6.2]{LM1}). The group $\SL_3(\bbO_{\bbC})$ is a complex simple Lie group of type $E_6$.
Moreover, the stabilizer of any element is isomorphic to the maximal parabolic subgroup $P_6$ corresponding to the $6$-th simple root of the diagram $E_6$
(which is a cominuscle simple root, see Table \ref{F:cominus}). Therefore, we obtain the following explicit presentation of $\bbP^2_{\bbO}$ as a cominuscle homogeneous variety (as in Definition \ref{D:comin})
\begin{equation}\label{E:G/PV}
\bbP^2_{\bbO}\cong \SL_3(\bbO_{\bbC})/P_6.
\end{equation}

The \emph{Borel embedding} of $\calD_{V}$ into $\bbP^2_{\bbO}$ is given by
\begin{equation}\label{E:BorV}
\begin{aligned}
\calD_{V}\subset \bbO_{\bbC}^2 & \stackrel{j}{\hookrightarrow} \bbP^2_{\bbO}\\
\begin{pmatrix} x_1 \\ x_2 \end{pmatrix} & \mapsto \left[ \begin{pmatrix}1 & x_2 & \wt{x_1} \\ \wt{x_2} & |x_2|^2 & \wt{x_2}\wt{x_1} \\
x_1 & x_1x_2 & |x_1|^2 \end{pmatrix} \right].
\end{aligned}
\end{equation}
Indeed, it is easily checked that $j(\bbO_{\bbC}^2)$ is the Zariski open subset of $\bbP^2_{\bbO}$ consisting of all the matrices $[a]\in \bbP^2_{\bbO}$ whose $(1,1)$-entry is non-zero.

\section{Boundary components}\label{S:bound}

The aim of this section is to define and study the boundary components of a Hermitian symmetric manifold of non-compact type, or, equivalently, of a bounded symmetric domain, see \S\ref{SS:domain}.

Let $D\stackrel{i_{HC}}{\hookrightarrow} \bbC^N$ be a bounded symmetric domain in its Harish-Chandra embedding and let $D\stackrel{i_{HC}}{\hookrightarrow} \bbC^N\stackrel{j}{\hookrightarrow} D^c$ be the Borel embedding into the compact dual $D^c$ of $D$ (see Theorem \ref{T:embed}). Denote by $\ov D$ the closure of $D$ inside $\bbC^N$ with respect to the  Euclidean topology.

\begin{defi}\label{D:bound}
Consider the following equivalence relation $\sim$ on $\ov D$: $p\sim q$ if and only if there exist
holomorphic maps $\lambda_1,\cdots,\lambda_m:\Delta=\{z\in \bbC: |z|<1\}\to \ov D$ (for some $m\in \bbN$) such that
\begin{itemize}
\item $\lambda_1(0)=p$ and $\lambda_m(0)=q$;
\item $\Im \lambda_i \cup \Im \lambda_{i+1}\neq \emptyset$ for any $1\leq i\leq m-1$.
\end{itemize}
A \emph{boundary component }Ê$F$ of $D$ is an equivalence class for the above equivalence relation $\sim$ on $\ov D$.

Given two boundary components $F_1$ and $F_2$ of $D$, we say that $F_1$ dominates $F_2$ (and we write $F_2\leq F_1$) if $F_2\subseteq \ov{F_1}$.
\end{defi}
In other words, two points $p$ and $q$ of $\ov D$ belong to the same boundary component if they can be connected by a finite chain of holomorphic disks contained in $\ov D$.
Note that $D$ is always a boundary component of itself and that for every boundary component $F$ of $D$ it holds that $F\leq D$.

\begin{thm}\label{T:bound}
Let $D\subset \bbC^N$ be a bounded symmetric domain in its Harish-Chandra embedding. Then
\begin{enumerate}[(i)]
\item \label{T:bound1} $\ov D=\coprod_{F\leq D} F$ and $G=\Hol(D)^o$ preserves this decomposition.
\item \label{T:bound2} Let $F\leq D$ and denote by  $\langle F\rangle$ be the smallest linear subspace of $\bbC^N$ containing $F$, by $\ov F$ be the Euclidean closure of $F$ inside $\bbC^N$ (or equivalently inside $\langle F \rangle$)  and by  $F^c$ the Zariski closure of $F$ inside $D^c$.
Then $F$ is a Hermitian symmetric manifold of non-compact type such that
\begin{itemize}
\item $F\subset \langle F\rangle$ is the Harish-Chandra embedding of $F$;
\item  $F\subset \langle F\rangle\subset \ov F$ is the Borel embedding of $F$.
\end{itemize}
Moreover the following diagram of inclusions is Cartesian
\begin{equation}
\xymatrix{
F\ar@{^{(}->}[r] & \ov F \ar@{^{(}->}[r] \ar@{^{(}->}[d] \ar@{}[dr]|{\square} & \langle F\rangle \ar@{^{(}->}[r]  \ar@{^{(}->}[d] \ar@{}[dr]|{\square}& F^c \ar@{^{(}->}[d]\\
D\ar@{^{(}->}[r] & \ov D\ar@{^{(}->}[r] & \bbC^N \ar@{^{(}->}[r]  & D^c \\
}
\end{equation}
\item \label{T:bound3} If $F\leq D$ and $F'\leq F$ then $F'\leq D$.
\item \label{T:bound4} If $D=D_1\times \cdots \times D_r$ is the decomposition of $D$ into irreducible bounded symmetric domains, then the boundary components of $D$ are the product of the
boundary components of the $D_i$'s.
\end{enumerate}
\end{thm}Ê
\begin{proof}
See \cite[Chap. III, Thm. 3.3]{AMRT}.
\end{proof}

\begin{remark}\label{R:BergSil}
It has been proved by Bott-Kor\'anyi (see \cite[\S 3]{KW}) that the union of the $0$-dimensional boundary components of a bounded symmetric domain $D\subset \bbC^N$ is the Bergman-Silov
boundary of $D$, i.e. the smallest closed subset of the boundary $\partial D:=\ov{D}\setminus D$ on which the absolute value of any function continuous on $\ov{D}$ and holomorphic on $D$
achieves its maximum.
\end{remark}

\subsection{The normalizer subgroup of a boundary component}\label{SS:norma}

The aim of this subsection is to study the normalizer subgroup of a boundary component $F$ of $D$.

\begin{defi}\label{D:norma}
The \emph{normalizer} subgroup of a boundary component $F\leq D$ is the subgroup
$$N(F):=\{g\in G=\Hol(D)^o: \: g F=F\}\subseteq G.$$
\end{defi}

We can classify the boundary components of $D$ in terms of their normalizer subgroups.

\begin{thm}\label{T:clas-norm}
Let $D=D_1\times \cdots \times D_s$ the decomposition of $D$ into its irreducible bounded symmetric domains and let $\Hol(D)^o=G=G_1\times \cdots \times G_s=
\Hol(D_1)^o\times \cdots \times \Hol(D_s)^o$ the associated decomposition of $G$ into its simple factors. Then there is a bijection
$$\begin{aligned}
\left\{\text{ÊBoundary components } F\leq D \right\}Ê& \stackrel{\cong}{\longrightarrow}
\left\{\begin{aligned}Ê
\text{ Subgroups }ÊP_1\times \cdots \times P_s\subseteq G_1\times \cdots \times G_s  \text{ such that}Ê \\
P_i=G_i \text{ or } P_i \text{ is a maximal parabolic subgroup of }ÊG_i
 \end{aligned} \right\}ÊÊ\\
 F & \mapsto N(F).
\end{aligned} $$
\end{thm}
\begin{proof}Ê
See \cite[Chap. III, Prop. 3.9]{AMRT}.
\end{proof}Ê

We want now to take a closer look at the structure of the normalizer subgroup associated to a boundary component of $D$.
We will need the following technical result.

\begin{lemma}\label{L:pair}
Let $F$ be a boundary component of $D$ and fix a base point $o\in D$. Then there exists a unique pair
\begin{equation}\label{E:pair}
\begin{aligned}
& f_F: \Delta\to D,\\
& \phi_F: \bbS^1\times \SL_2(\bbR)\to G=\Hol(D)^o,
\end{aligned}
\end{equation}
such that
\begin{enumerate}[(i)]
\item \label{E:pair1} $f_F$ is a symmetric morphism (in the sense of Remark \ref{R:funHSM-SLA}) such that $f_F(0)=o\in D$ and $o_F:=f_F(1):=\lim_{z\to 1}Êf_F(z)\in F$;
\item \label{E:pair2} $\phi_F$ is a morphism of Lie groups such that
$$h_o(e^{i\theta}):=\phi_F\left(e^{i\theta}, \begin{pmatrix} \cos \theta & \sin \theta \\ -\sin \theta & \cos \theta \end{pmatrix} \right) $$
belongs to $K=\Stab(o)\subset G$ and it acts on $T_o D$ as multiplication by $e^{2i\theta}$.
\item \label{E:pair3} $f_F$ is equivariant with respect to the morphism $\phi_F$ and the natural actions of $G$ on $D$ and of $\bbS^1\times \SL_2(\bbR)$ on $\Delta$ via
$$
\begin{aligned}Ê
\left[\bbS^1\times SL_2(\bbR)\right] \times \Delta & \longrightarrow \Delta \\
\left(\left[e^{i\theta},\begin{pmatrix} a & b \\ c& d\end{pmatrix} \right], z\right) & \mapsto  \frac{[i(a+d)-(b-c)]z+[i(a-d)+(b+c)]}{[i(a-d)-(b+c)]z+[i(a+d)+(b-c)]}.
\end{aligned}Ê
$$
\end{enumerate}
\end{lemma}
\begin{proof}
See \cite[Chap. III, Thm. 3.3(v) and Thm. 3.7]{AMRT}.
\end{proof}

\begin{remark}\label{R:pairs}
\noindent
\begin{enumerate}[(i)]
\item Since $f_F:\Delta \to D$ is a symmetric morphism, it extends uniquely to a morphism between  their Harish-Chandra and Borel embeddings (see \cite[Chap. III, Sec. 2.2]{AMRT})
$$\xymatrix{
\Delta \ar[r]^{f_F} \ar@{^{(}->}[d] & D \ar@{^{(}->}[d] \\
 \bbC \ar@{^{(}->}[d]  \ar[r]^{\ov{f_F}} & \bbC^N \ar@{^{(}->}[d] \\
 \Delta^c=\bbP^1 \ar[r]^{f_F^c} & D^c
} $$
In particular $f_F(1)=\ov{f_F}(1)\in \ov D$.
\item  For any non-Euclidean Hermitian symmetric domain $M$ with a fixed base point $o$ there exists a unique morphism $u_o:\bbS^1\to G=\Aut(M)^o$ such that
$\Im u_o \subset K=\Stab(o)\subset G$ and $u_o(e^{i\theta})$ induces the multiplication by $e^{i\theta}$ on $T_oM$ (see \cite[Chap. III, Sec. 2.1]{AMRT}).
Therefore, part \eqref{E:pair2} is equivalent to saying that $h_o^2=u_o$.
\item The action of $\SL_2(\bbR)$ on $\Delta$ in part \eqref{E:pair3} is equivalent to the action of $\SL_2(\bbR)$ on the upper half space $\calH$ via Mo\"ebius transformations (see
Example \ref{E:dim1}\eqref{E:dim1b}) using the Cayley isomorphism $\calH\cong \Delta$ of \eqref{E:Cayley-1dim}.Ê
\end{enumerate}
\end{remark}

The connected component of the normalizer subgroup of a boundary component of $D$ admits the following decomposition, know as the $5$-\emph{term decomposition}.

\begin{thm}\label{T:5deco}
Let $F$ be a boundary component of $D$ and let $N(F)$ its associated normalizer subgroup. Consider the one-parameter subgroup of $G=\Hol(D)^o$
\begin{equation}\label{E:wF}
\begin{aligned}Ê
w_F: \Gm & \longrightarrow G\\
t & \mapsto \phi_F\left(1,\begin{pmatrix} t & 0 \\ 0 & t^{-1}\end{pmatrix} \right).
\end{aligned}
\end{equation}
\begin{enumerate}[(i)]
\item \label{T:5deco1} The  normalizer subgroup $N(F)$ of $F$ is equal to
$$N(F)=\{g\in G: Ê\lim_{t\to 0}w_F(t)\cdot g\cdot  w_F(t)^{-1} \text{ exists}Ê\}.$$
\item \label{T:5deco2} The connected component $N(F)^o$ of $N(F)$ is equal to the semidirect product
$$N(F)^o=Z(w_F)^o\ltimes W(F),$$
where:

$\bullet$ $W(F)$ is the unipotent radical of $N(F)^o$ and it is equal to
$$W(F):=\{g\in G:  \lim_{t\to 0}w_F(t)\cdot g\cdot  w_F(t)^{-1}=1\in G\};$$

$\bullet$ $Z(w_F)^o$ is a Levi subgroup of $N(F)^o$ and it is the connected component of the centralizer $Z(w_F)$ of $w_F$, i.e.
$$ Z(w_F):=\{g\in G:  w_F(t)\cdot g\cdot  w_F(t)^{-1}=g \text{ for any } t\in \Gm\}.$$

\item \label{T:5deco3} $W(F)$ is a $2$-step unipotent group which is given as an extension of two abelian groups
$$0\to U(F)\to W(F) \to V(F) \to 0,$$
where $U(F)$ is the center of $W(F)$ and $V(F):=W(F)/U(F)$.

\item \label{T:5deco4} $Z(w_F)^o$ is a reductive group which is equal to the product modulo finite subgroups
$$Z(w_F)^o=G_h(F)\cdot G_l(F) \cdot M(F),$$
where

$\bullet$ $M(F)$ is compact and semisimple;

$\bullet$ $G_h(F)$ is semisimple  and it satisfies $G_h(F)/Z_{G_h(F)}=\Aut(F)^o$;

$\bullet$ $G_l(F)$ is reductive without compact factors.

\end{enumerate}Ê
\end{thm}
\begin{proof}
See \cite[Chap. III. Thm. 3.7, Thm. 3.10, \S 4.1]{AMRT}.
\end{proof}Ê

\subsection{The decomposition of $D$ along a boundary component}\label{SS:decom}

The aim of this subsection is to deduce from the $5$-term decomposition of the normalizer $N(F)$ of a boundary component $F$ of $D$ (see Theorem \ref{T:5deco})
a decomposition of $D$ along $F$.

\begin{prop}\label{P:transi}
Let $D$ be a bounded symmetric domain and fix a boundary component $F\leq D$.
Then $N(F)^o$ acts transitively on $D$.
\end{prop}
\begin{proof}
See \cite[Chap. III, \S 4.3]{AMRT}.
\end{proof}

Recall that, fixing a base point $o\in D$,  $D$ is diffeomorphic to $G/K$ where $G=\Hol(D)^o$ and $K=\Stab(o)\subset G$ is a maximal compact subgroup (see Theorem \ref{T:LieHerm}).
 Using this and the above Proposition \ref{P:transi}, we have a diffeomorphism
 \begin{equation}\label{E:newdif1}
 D\cong N(F)^o/(K\cap N(F)^o).
 \end{equation}
From the $5$-term decomposition of $N(F)^o$ (see Theorem \ref{T:5deco}), it follows that
\begin{equation}\label{E:K-N(F)}
K\cap N(F)^o=K\cap Z(w_F)=K_l(F)\cdot K_h(F)\cdot M(F)\subset G_l(F)\cdot G_h(F)\cdot M(F)=Z(w_F),
\end{equation}
where $K_l(F)\subset G_l(F)$  and $K_h(F)\subset G_h(F)$ are maximal compact subgroups. Substituting  \eqref{E:K-N(F)} into \eqref{E:newdif1}, we get the diffeomorphism
\begin{equation}\label{E:newdif2}
D   \stackrel{\cong}{\longrightarrow} \frac{N(F)^o}{K\cap N(F)^o}= \frac{[G_h(F)\cdot G_l(F)\cdot M(F)]\ltimes W(F)}{K_h(F)\cdot K_l(F)\cdot M(F)}=\frac{G_h(F)}{K_h(F)}\times \frac{G_l(F)}{K_l(F)}\times W(F).
\end{equation}

In order to get a better description of the above diffeomorphism, we will now describe more geometrically the right hand side of \eqref{E:newdif2}.

\begin{thm}\label{T:F-cone}
Notations as above.
\begin{enumerate}[(i)]
\item \label{T:F-cone1} Under the natural action of $G_h(F)\subset G$ on $\ov D\subset \bbC^N$, the orbit of the point $o_F:=f_F(1)$ is equal to $F\subset \ov D$ and its stabilizer is equal to $K_h(F)$.
Therefore, we have a diffeomorphism
\begin{equation}\label{E:iso-F}
\frac{G_h(F)}{K_h(F)}\cong F.
\end{equation}
\item \label{T:F-cone2} Under the action of $G_l(F)\subset N(F)^o$ on $U(F)$ by conjugation, the orbit of the point $\Omega_F:=\phi_F\left(1,\begin{pmatrix} 1 & 1 \\ 0 & 1\end{pmatrix}Ê\right)\in U(F)$
is an open cone $C(F)\subset U(F)$ and its stabilizer is equal to $K_l(F)$. Therefore, we have a diffeomorphism
\begin{equation}\label{E:iso-C(F)}
\frac{G_l(F)}{K_l(F)}\cong C(F).
\end{equation}
\end{enumerate}
\end{thm}Ê
\begin{proof}
Part \eqref{T:F-cone1} follows from \cite[Chap. III, Thm. 3.10 and Lemma 4.6]{AMRT}.
Part \eqref{T:F-cone2} Êfollows from \cite[Chap. III, Thm. 4.1]{AMRT}.
\end{proof}

\begin{remark}\label{R:idem}
Let $o\in D\subseteq \p_+$ be a bounded symmetric domain (with a fixed base point $o$) in its Harish-Chandra embedding (see Theorem \ref{T:embed}).
Consider the Jordan triple product $\{.\,,.\,,.\}$  of \eqref{E:Jortriple} with respect to which $(\p_+,\{.\,,.\,,.\})$ is a Hermitian positive JTS (see \S\ref{SS:HSM-JTS}).
Then there is a bijection (see \cite[Chap. III, \S8, Rmk. 2]{Sat})
\begin{equation}\label{E:bound-idem}
\begin{aligned}
\left\{\text{Boundary components of }ÊDÊ\right\} & \stackrel{\cong}{\longrightarrow}Ê \left\{\text{Tripotents of } (\p_+,\{.\,,.\,,.\}) \right\} \\
F &\longmapsto o_F
\end{aligned}
\end{equation}
where a tripotent (or idempotent)  of $(\p_+,\{.\,,.\,,.\})$ is an element $e\in \p_+$ such that $\{e,e,e\}=e$.
For a detailed study of the tripotents of a Jordan triple system, we refer the reader to \cite[Chap. V, Part V]{FKKLR}.
\end{remark}

Using the above Theorem \ref{T:F-cone}, the diffeomorphism \eqref{E:newdif2} can be written as
 \begin{equation}\label{E:newdif3}
\begin{aligned}
D  & \stackrel{\cong}{\longrightarrow} F\times C(F) \times W(F) \\
x & \mapsto  (\pi_F(x), \Phi_F(x), w(x)).
\end{aligned}
\end{equation}
The above smooth maps $w$, $\pi_F$ and $\Phi_F$ can be described explicitly as it follows. By \eqref{E:newdif1}Ê and \eqref{E:K-N(F)}, we can write $x\in D$ as
$$x=g_hg_lw\cdot o,$$
for some unique $w\in W(F)$ and for some $g_h\in G_h(F)$ (resp. $g_l\in G_l(F)$) which is unique up to multiplication by an element of $K_h(F)$ (resp. $K_l(F)$).
Then we have that (see \cite[Chap. III, \S 4.3]{AMRT}):
\begin{equation}\label{E:expl-maps}
\begin{sis}
& w(x):=w\in W(F), \\
& \pi_F(x):=g_h \cdot o_F\in F,\\
& \Phi_F(x):=g_l\cdot \Omega_F\in C(F).
\end{sis}
\end{equation}

The maps $\pi_F$ and $\Phi_F$ are closely related as we are now going to explain. Recall that $D$ embeds as an open subset in its compact dual $D^c$ via the Borel embedding
and that $D^c$ is a homogeneous projective variety with respect to the action of the complexification $G_{\bbC}$ of $G=\Hol(D)^o$ (see Theorem \ref{T:embed}).
We will denote by $U(F)_{\bbC}\subset G_{\bbC}$ the complexification of $U(F)\subset N(F)^o\subset G$.

\begin{defi}\label{E:D(F)}
Notations as above.  Denote by $D(F)$ the analytic open subset of $D^c$:
\begin{equation*}
D\subset D(F):=U(F)_{\bbC}\cdot D=\bigcup_{g\in U(F)_{\bbC}} g \cdot D\subset D^c.
\end{equation*}
\end{defi}

The open subset $D(F)$  admits the following Lie-theoretic description.

\begin{lemma}\label{L:Lie-D(F)}
We have a diffeomorphism
$$D(F)\cong \frac{N(F)^o\cdot U(F)_{\bbC}}{K_h(F)\cdot G_l(F)\cdot M(F)},
$$
where $N(F)^o\cdot U(F)_{\bbC}$ is the subgroup of $G_{\bbC}$ generated by $N(F)^o$ and $U(F)_{\bbC}$.
\end{lemma}
\begin{proof}
The group $N(F)^o\cdot U(F)_{\bbC}$ acts transitively on $D(F)$ by Proposition \ref{P:transi} and Definition \ref{E:D(F)}. The stabilizer of the point
\begin{equation*}
o_F^o:=s_o(o_F)=f_F(-1)\in D(F)
\end{equation*}
is equal to $K_h(F)\cdot G_l(F)\cdot M(F)$ by \cite[Chap. III, Lemma 4.6]{AMRT}. Hence, we get the desired diffeomorphism.
\end{proof}

Using the diffeomorphism in Lemma \ref{L:Lie-D(F)}, we can define two smooth and surjective maps
 \begin{equation}\label{E:ext-maps}
 \begin{sis}
  \wt{\pi}_F:D(F)\cong \frac{N(F)^o\cdot U(F)_{\bbC}}{K_h(F)\cdot G_l(F)\cdot M(F)} & \longrightarrow \frac{G_h(F)}{K_h(F)}\cong F, \\
  \wt{\Phi}_F: D(F) \cong \frac{N(F)^o\cdot U(F)_{\bbC}}{K_h(F)\cdot G_l(F)\cdot M(F)} & \longrightarrow \frac{N(F)^o\cdot U(F)_{\bbC}}{N(F)^o} \cong U(F),
 \end{sis}
 \end{equation}
where the last diffeomorphism is obtained by projecting onto $i\cdot U(F)\subset U(F)_{\bbC}$.

\begin{thm}\label{T:type-dom}
Notations as above.
\begin{enumerate}[(i)]
\item \label{T:type-dom1} The smooth maps   $\wt{\pi}_F $ and  $\wt{\Phi}_F$ fit into the following commutative diagram
\begin{equation*}
\xymatrix{
C(F) \ar@{^{(}->}[r] \ar@{}[dr]|{\square} & U(F) \\
D  \ar@{^{(}->}[r] \ar[u]_{\Phi_F}Ê\ar[rd]_{\pi_F}& D(F)\ar[u]^{\wt{\Phi}_F} \ar[d]^{\wt{\pi}_F}\\
& F\\
}
\end{equation*}
where the upper square is Cartesian.
\item \label{T:type-dom2} The smooth map $\wt{\pi}_F$ factors as
\begin{equation*}
D(F) \stackrel{\pi'_F}{\longrightarrow} D(F)':=D(F)/U(F)_{\bbC} \stackrel{p_F}{\longrightarrow} F=D(F)'/V(F),
\end{equation*}
in such a way that $\pi'_F$ is a trivial holomorphic $U(F)_{\bbC}$-torsor and $p_F$ is a smooth $V(F)$-torsor and, at the same time, a trivial complex vector bundle. In particular, we have
a diffeomorphism
\begin{equation}\label{E:diff-D(F)}
D(F)\cong U(F)_{\bbC}\times \bbC^k\times F,
\end{equation}
for some $k\in \bbN$.
\item \label{T:type-dom3}
Under the diffeomorphism \eqref{E:diff-D(F)}, the smooth map $\wt{\Phi}_F$ can be written as
\begin{equation*}
\begin{aligned}
\wt{\Phi}_F: D(F)\cong  U(F)_{\bbC}\times \bbC^k\times F & \longrightarrow U(F) \\
(x,y,z) & \mapsto \Im x -h_z(y,y),
\end{aligned}
\end{equation*}
for some bilinear symmetric form $h_z:\bbC^k\times \bbC^k\to U(F)$ varying smoothly with $z\in F$.
\end{enumerate}
\end{thm}Ê
\begin{proof}
See \cite[Chap. III, \S 4.3]{AMRT}.
\end{proof}

From the above Theorem, we deduce the following presentation of $D$ as a \emph{Siegel domain of the third kind}.

\begin{cor}\label{C:SiegIII}
Notations as above. We have a diffeomorphism
\begin{equation*}
D\cong \{(x,y,z)\in U(F)_{\bbC}\times \bbC^k\times F: \Im x-h_z(y,y)\in C(F)\}.Ê
\end{equation*}
\end{cor}

\subsection{Symmetric cones and Euclidean Jordan algebras}\label{SS:symcon}

The cone $C(F)$ associated to a boundary component $F\leq D$ (see Theorem \ref{T:F-cone}\eqref{T:F-cone2}) belongs to a special class of cones, namely the symmetric cones, that we now introduce.

\begin{defi}\label{D:symcon}
Let $V$ be a real (finite-dimensional) vector space endowed with a scalar product $\langle , \rangle$ (i.e. an Euclidean space).
An open (pointed and convex) cone $C\subset V$ is said to be:
\begin{enumerate}[(i)]
\item \emph{homogeneous}Ê if the group of automorphisms of $C$:
$$\G(C):=\{g\in \GL(V): g \cdot C=C\}\subset \GL(V)$$
acts transitively on $C$.
\item \emph{symmetric}Ê if it is homogeneous and and self-dual, i.e. $C$ is equal to the its dual cone
$$C^*=\{x\in V: \langle x, y\rangle>0 \text{ for any } y\in \ov C\}.$$
\end{enumerate}
\end{defi}

Some basic properties of homogeneous and symmetric cones are contained in the following

\begin{thm}\label{T:str-cones}
Let $C\subset (V,\langle , \rangle)$ be a homogeneous cone and fix a base point $o\in C$.
\begin{enumerate}[(i)]
\item \label{T:str-cones1}
The stabilizer subgroup of $o$
$$\G(C)_o:=\{g\in G(C): g(o)=o\}\subset G(C)$$
is a maximal compact subgroup of $\G(C)$ and, conversely, every maximal compact subgroup of $\G(\Omega)$ is the stabilizer subgroup of some point of $\Omega$
\item \label{T:str-cones1b}
 We have a diffeomorphism
$$\frac{\G(C)}{\G(C)_o}\cong  C.$$
\item \label{T:str-cones2}
$C$ is symmetric if and only if $\G(C)$ is equal to its dual group
$$\G(C)^*=\{g^*: g\in \G(C)\},$$
where $g^*$ denote the adjoint of the element $g\in \GL(V)$ with respect to the scalar product $\langle , \rangle$.
In particular, in this case, $\G(C)$ is a reductive Lie group.
\end{enumerate}
\end{thm}
\begin{proof}
For part \eqref{T:str-cones1}, see \cite[Chap. I, Prop. 8.4]{Sat}. For part \eqref{T:str-cones1b}, see \cite[Chap. I, \S4]{FK}.
For part \eqref{T:str-cones2}, see \cite[Chap. I, Lemma 8.3]{Sat}.Ê
\end{proof}

\begin{remark}
It can be shown that symmetric cones are Riemannian symmetric manifolds in the sense of Remark  \ref{R:symHerm}\eqref{R:symHerm4};
see \cite[Chap. I, \S4]{FK} for a proof.
\end{remark}

It turns out (see \cite[Prop. III.4.5]{FK}) that any symmetric cone decomposes uniquely as the product of irreducible symmetric cones, defined as it follows.

\begin{defi}\label{D:irrcone}
A symmetric cone $\Omega\subset $ is said to be \emph{irreducible} if and only if there does not exist a non-trivial decomposition $V=V_1\oplus V_2$ and two symmetric cones
$\Omega_1\subset V_1$ and $\Omega_2\subset V_2$ such that $\Omega=\Omega_1+\Omega_2$ (in this case, we say that $\Omega$ is the product of $\Omega_1$ and $\Omega_2$).
\end{defi}

Symmetric cones can be classified via Euclidean Jordan algebras, which we now introduce.

\begin{defi}\label{D:Jalg}
A \textbf{Jordan algebra} over a field $F$ is a (finite-dimensional) algebra $(A, \circ)$ over $F$ such that
\begin{enumerate}[(i)]
\item[(J1)] $x\circ y=y\circ x$ for any $x,y\in A$;
\item[(J2)] $T_x$ and $T_{x\circ x}$ commutes, where for any $x\in A$ we denote by $T_x$ the endomorphism of $A$ given by
$$
\begin{aligned}
T_x: A & \longrightarrow A, \\
y & \mapsto T_x(y):=x\circ y.
\end{aligned}
$$
\end{enumerate}
A Jordan $F$-algebra $(A,\circ)$ is said to be
\begin{enumerate}[(i)]
\item \emph{semisimple} Êif
the trace form
\begin{equation}\label{E:trace}
\begin{aligned}
\tau: A\times A & \longrightarrow F, \\
(x,y) & \mapsto \tau(x,y):=\tr(T_{x\circ y}).
\end{aligned}
\end{equation}
is non-degenerate.
\item \emph{simple}Ê  if $\tau$ is not identically zero and $A$ does not contain proper ideals, i.e. proper subvector spaces $I\subset A$ such that for any $x\in I$ and $y\in A$ we have that
$x\circ y\in I$.

\item  \emph{Euclidean} Êif $F=\bbR$ and
the trace form $\tau$ is positive definite.
\end{enumerate}
\end{defi}

The following properties of Jordan algebras follow quite easily from the axioms (J1) and (J2).

\begin{lemma}\label{L:Jprop}
Let $(A,\circ)$ be a Jordan $F$-algebra. Then
\begin{enumerate}[(i)]
\item \label{L:Jprop1} $(A,\circ)$ is power-associative, i.e. if we define inductively $x^p:=x\circ x^{p-1}$ (for any $p\in \bbZ_{>0}$) then we have that $x^p\circ x^q=x^{p+q}$ for any $p,q\in \bbZ_{>0}$.
\item \label{L:Jprop2} For any $x\in A$ and any $p,q\in \bbZ_{>0}$ the endomorphisms $T_{x^p}$ and $T_{x^q}$ commute.
\item \label{L:Jprop3} The trace form $\tau$ is associative, i.e.
$$\tau(x\circ y, z)=\tau(x, y\circ z),$$
for any $x,y,z\in A$. In particular, if $(A,\circ)$ is semisimple then, for any $y\in A$, the endomorphism $T_y$ is self-adjoint with respect to $\tau$.
\item \label{L:Jprop4} If $(A,\circ)$ is semisimple then $(A,\circ)$ has a unique unit element $e\in A$, i.e. an element $e\in A$ such that
$e\circ x=x$ for any $x\in A$.
\end{enumerate}Ê
\end{lemma}
\begin{proof}
For \eqref{L:Jprop1} and \eqref{L:Jprop2}, see \cite[Prop. II.1.2]{FK}. For \eqref{L:Jprop3}, see \cite[Prop. 2.4.3]{FK}.
Part \eqref{L:Jprop4}: since $\tau$ is non-degenerate, there exists a unique $e\in A$ such that $\tau(e,x)=\tr T_x$ for any $x\in A$.
Using the associativity of $\tau$, we get (for any $x,y\in A$)
$$\tau(x,e\circ y)=\tau(x\circ y,e)=\tr T_{x\circ y}=\tau(x,y),$$
which implies (again by the non-degeneracy of $\tau$) that $e\circ y=y$, q.e.d.
\end{proof}

\begin{example}\label{E:Jor}
\noindent
\begin{enumerate}
\item \label{E:Jor1} Let $(A,\cdot)$ be an associative $F$-algebra. Then $A$ becomes a Jordan algebra with respect to the Jordan product
$$x\circ y:=(x\cdot y+y\cdot x).$$
 \item \label{E:Jor2} Let $W$ be a $F$-vector space and let $B$ be a symmetric bilinear form on $W\times W$. Then $A=F\oplus W$ becomes a Jordan algebra with respect to the Jordan product
 \begin{equation}\label{E:circB}
 (\lambda, u)\circ_B (\mu, v):=(\lambda\mu+B(u,v), \lambda v+\mu u).
 \end{equation}
It is easily checked that  the Jordan algebra $(F\oplus W,\circ_B)$ is semisimple if and only if $B$ is non-degenerate and that it is Euclidean if and only if $F=\bbR$ and $B$ is positive definite.
 \item \label{E:Jor3} Let $D$ be equal to $\bbR, \bbC$ or $ \bbH$ and denote by $x\mapsto \ov x$ the natural involution. For any $n\geq 2$, the real vector space of Hermitian matrices of order $n$ with
 entries in $D$
$$\Herm_n(D):=\{M\in M_{n,n}(D): \ov{M}^t=M\}$$
becomes a Euclidean Jordan algebra with respect to the Jordan product (see \cite[Chap. 5, \S 2]{FK})
\begin{equation}\label{E:JorHerm}
M_1\circ M_2=\frac{1}{2}(M_1M_2+M_2M_1).
\end{equation}
If $D$ is equal to the algebra of octonions $\bbO$, then $\Herm_n(\bbO)$ with the product \eqref{E:JorHerm} is an Euclidean Jordan algebra if $m\leq 3$ (see \cite[Chap. 5, \S 2]{FK}).
In particular, $\Herm_3(\bbO)$ is an Euclidean Jordan algebra of dimension $27$, known as the Albert algebra.

The Jordan algebras $\Herm_2(D)$ for $D=\bbR,\bbC,\bbH$ or $\bbO$  are isomorphic to the Jordan algebras associated to  a suitable bilinear symmetric form as in Example \ref{E:Jor2}; more precisely, we have that
\begin{equation}\label{E:Herm2}
\begin{sis}
&\Herm_2(\bbR)\cong (\bbR\oplus \bbR^2,\circ_Q) \\
& \Herm_2(\bbC)\cong (\bbR\oplus \bbR^3,\circ_Q) \\
& \Herm_2(\bbH)\cong (\bbR\oplus \bbR^5,\circ_Q)\\
& \Herm_2(\bbO)\cong (\bbR\oplus \bbR^9,\circ_Q)
\end{sis}
\end{equation}
where $\circ_Q$ is defined in Example \ref{E:circB} with respect to the positive definite symmetric bilinear form $B$ on the suitable vector space.

\item \label{E:JA-JTS} Let $(A,\circ)$ be a Jordan algebra over $F$. Then $A$ becomes a Jordan triple system with respect to the triple product $\{.\,,.\,,.\}_{\circ}$ defined by
(see \cite[Chap. 1,\S6]{Sat})
$$
\{x,y,z\}_{\circ}:=(x\circ y)\circ z+ (z\circ y)\circ x-(x\circ z)\circ y.
$$
It turns out that $x\square y=T_{x\circ y}+[T_x,T_y]$ for any $x,y\in A$ which implies that the trace form of the Jordan algebra $(A,\circ)$ as defined in \eqref{E:trace} coincides with the trace
form Êof the JTS $(A,\{.\,,.\,,.\}_{\circ})$ as defined in \eqref{E:trform}. In particular, $(A,\circ)$ is semisimple if and only if $(A,\{.\,,.\,,.\}_{\circ})$ is semisimple.

If $(A,\circ)$ is a Jordan algebra over $\bbR$, then $A_{\bbC}:=A\otimes_{\bbR}\bbC$ becomes a Hermitian JTS with respect to the triple product
$$
\{x,y,z\}'_{\circ}:=(x\circ \ov y)\circ z+ (z\circ \ov y)\circ x-(x\circ z)\circ \ov y,
$$
where $y\mapsto \ov y$ is the complex conjugation corresponding to the real form $A\subset A_{\bbC}$ and the Jordan product $\circ$ is extended linearly to $\bbA_{\bbC}$.
Then $(A,\circ)$ is Euclidean if and only if $(A_{\bbC},\{.\,,.\,,.\}'_{\circ})$ is a  positive Hermitian JTS.


\end{enumerate}
\end{example}

Observe that simple Jordan algebras are semisimple. Moreover, the direct sum of simple Jordan algebras is semisimple, where the direct sum of two Jordan algebras
$(A_1,\circ_1)$ and $(A_2,\circ_2)$ is the vector space $A:=A_1\oplus A_2$ endowed with the component-wise Jordan product $(x_1,x_2)\circ (y_1,y_2):=(x_1\circ y_1,x_1\circ y_2)$.
Conversely, we have the following decomposition theorem.

\begin{prop}\label{P:simple}
Any semisimple (resp. Euclidean) Jordan algebra decomposes uniquely as the product of simple (resp. Euclidean and simple) Jordan algebras.
\end{prop}
\begin{proof}[Sketch of the Proof]
Let $(A,\circ)$ be a semisimple (resp. Euclidean) Jordan algebra. If $A$ is not simple, then there exists a proper ideal $I\subset A$. Consider the orthogonal complement of $I$
$$I^{\perp}:=\{x\in A\: : \tau(x,y)=0 \text{ for any }Êy\in I\}.$$
It is possible to prove  (see \cite[Prop. III.4.4]{FK}) that
\begin{enumerate}[(i)]
\item $I^{\perp}$ is an ideal of $A$;
\item $I$ and $I^{\perp}$ are semisimple (resp. Euclidean) Jordan algebras;
\item $A=I\oplus I^{\perp}$ as Jordan algebras.
\end{enumerate}
Iterating this construction for $I$ and $I^{\perp}$, we get  the existence of the decomposition. For the unicity, see loc. cit.
\end{proof}

Simple Euclidean Jordan algebras were classified by Jordan-Neumann-Wigner (see \cite[Chap. V]{FK} and the references therein).

\begin{thm}\label{T:clasJor}
Every simple Euclidean Jordan algebra is isomorphic to one of the following Jordan algebras
\begin{enumerate}[(i)]
\item $(\bbR\oplus \bbR^n,\circ_Q)$ where $Q$ is the standard scalar product on $\bbR^n$;
\item $\Herm_n(\bbR)$ for $n\geq 3$;
\item $\Herm_n(\bbC)$ for $n\geq 3$;
\item $\Herm_n(\bbH)$ for $n\geq 3$;
\item $\Herm_3(\bbO)$ (the \emph{Albert algebra}).
\end{enumerate}
\end{thm}

\begin{table}[!ht]
\begin{tabular}{|c|c|c|c|c|c|}
\hline
 Real vector space $V$ & Jordan product  $\circ$ \\
\hline \hline
 $ \bbR\times \bbR^{n}$ &  $\un x\circ\un y=(x_0y_0+\sum_{i=1}^nx_iy_i, x_0y_1+y_0x_1,\cdots,x_0y_n+y_0x_n)$\\
\hline
 $ \Herm_n(\bbR)\:  (n\geq 3) $  & $A\circ B=\frac{1}{2}(AB+BA) $ \\
\hline
 $ \Herm_n(\bbC)\:  (n\geq 3) $ & $A\circ B=\frac{1}{2}(AB+BA) $ \\
\hline
 $ \Herm_n(\bbH)\:  (n\geq 3) $ & $A\circ B=\frac{1}{2}(AB+BA) $ \\
\hline
 $ \Herm_3(\bbO) $ & $A\circ B=\frac{1}{2}(AB+BA) $ \\
\hline
\end{tabular}
\caption{Simple Euclidean Jordan algebras}\label{F:EucJor}
\end{table}

\begin{remark}\label{R:Eu-cmplx}
The complexification of a real Jordan algebra $(A,\circ)$ is the complex vector space $A^{\bbC}:=A\otimes_{\bbR}\bbC$ endowed with the Jordan product $\circ^{\bbC}$ obtained by
extending linearly the Jordan product $\circ$ on $A$. The complexification induces a bijection
\begin{equation}\label{E:compl}
\begin{aligned}
\left\{\text{Euclidean Jordan algebras} \right\}Ê& \stackrel{\cong}{\longrightarrow} \left\{\text{Semisimple Jordan $\bbC$-algebras}Ê\right\}Ê\\
(A,\circ) & \mapsto (A^{\bbC},\circ^{\bbC})
\end{aligned}
\end{equation}
which preserves the decomposition into the product of simple Jordan algebras (see \cite[Chap. VIII]{FK}).
\end{remark}

We now explain the relationship between Euclidean Jordan algebras and symmetric cones, due to the work of  Koecher and Vinberg.

Let $(V,\circ)$ be a Euclidean Jordan algebra and denote by $e\in V$ the unit element of $V$ (see Lemma \ref{L:Jprop}\eqref{L:Jprop4}). Denote by $V^*$ the set of invertible elements
of $V$, i.e. the elements $x\in V$ for which there exists $y\in V$ such that $x\circ y=e$.
Then we define an open cone inside $V$ by
\begin{equation}\label{E:coneJor}
\Omega_{(V,\circ)}:=\{x^2 : x \in V^*\}=\{x: x \in V^*\}^o=\{x\in V: T_x>0\},
\end{equation}
where $\{.\,,.\,,.\}^o$ denotes the connected component containing the identity $e\in V$. Indeed, $\Omega_{(V,\circ)}$ is a symmetric cone with respect to the positive definite form
$\langle,\rangle:=\tau(,)$ (see \cite[Chap. III, \S2]{FK}).

Conversely, let $\Omega\subset V$ be a symmetric cone with respect to a scalar product $\langle,\rangle$ on $V$ and choose a base point $e\in \Omega$.
Let $\g$ be the Lie algebra of $\G(\Omega)$, $\k$ the Lie algebra of the maximal compact subgroup $\G(\Omega)_e\subset \G(\Omega)$ and $\g=\k\oplus \p$ be the associated Cartan decomposition.
The action of $\G(\Omega)$ on $V$ induces an action of $\g$ on $V$. Clearly an element $X\in \g$ belongs to $\k$ if and only if $X\cdot e=e$. Therefore, the map $\p\to V$ sending $X$ into
$X\cdot e$ is a bijection; hence, for any $x\in V$, there exists a unique $L_x\in \k$ such that $L_x\cdot e=x$. Define a product $\circ_{\Omega}$ on $V$ by
\begin{equation}\label{E:Jorcone}
x\circ_{\Omega} y:=L_x\cdot y.
\end{equation}
The pair $(V,\circ_{\Omega})$ is an Euclidean Jordan algebra with unit element $e$ (see \cite[Chap. III,\S 3]{FK}).

\begin{thm}\label{T:bijJorcon}
There is a bijection
\begin{equation}
\begin{aligned}
\left\{\text{Euclidean Jordan algebras} \right\}Ê& \stackrel{\cong}{\longleftrightarrow} \left\{\text{Symmetric cones}Ê\right\}Ê\\
(V,\circ) & \longrightarrow \Omega_{(V,\circ)}\subset V\\
(V,\circ_{\Omega})& \longleftarrow \Omega\subset V
\end{aligned}
\end{equation}
preserving the decomposition of Euclidean Jordan algebras into simple ones and the decomposition of symmetric cones into irreducible ones.
\end{thm}
\begin{proof}
See \cite[Chap. III]{FK} or \cite[Chap. I, \S 8]{Sat}.
\end{proof}

\begin{remark}\label{R:funJorcon}
The bijection of Theorem \ref{T:bijJorcon} becomes an equivalence of categories if the two sets are endowed with the following morphisms (see \cite[Chap. I, \S9]{Sat}):
\begin{enumerate}[(i)]
\item A \emph{unital Jordan algebra} homomorphism between two Euclidean Jordan algebras $(V, \circ)$ and $(V', \circ')$  is a linear map $f:V\to V'$ such that
$$\begin{aligned}
& f(x\circ y)=f(x)\circ' f(y) & \text{ for any } x,y \in V, \\
&f(e)=e',
\end{aligned}$$
where $e$ (resp. $e'$) is the unit element of $(V,\circ)$ (resp. $(V',\circ')$).

\item An \emph{equivariant} morphism between two symmetric cones $\Omega\subset (V,\langle,\rangle)$ and $\Omega'\subset (V', \langle,\rangle')$ is a linear map $\phi:V\to V'$ sending $\Omega$ into $\Omega'$ and such that there
exists  a morphism of Lie algebras $\rho:\Lie \G(\Omega) \to \Lie \G(\Omega')$ satisfying
$$\begin{aligned}
& \phi(T\cdot x)=\rho(T)\cdot \phi(x) & \text{ for any } x\in V \text{ and any } T\in \Lie \G(\Omega), \\
& \rho(T^t)=\rho(T)^t \\
\end{aligned}$$
where ${}^t$ denotes the transpose with respect to either $\langle,\rangle$ or $ \langle,\rangle'$.
\end{enumerate}
\end{remark}

As a consequence of the bijection between symmetric cones and Euclidean Jordan algebras in Theorem \ref{T:bijJorcon} and the classification of simple Euclidean Jordan algebra
given in Theorem \ref{T:clasJor}, we get the following classification of irreducible symmetric cones (see \cite[Chap. 1, \S 8]{Sat}).

\begin{thm}\label{T:clascone}
Every irreducible symmetric cone  is isomorphic to one of the following cones
\begin{enumerate}[(i)]
\item $\P(1,n):= \{\un x\in \bbR\oplus \bbR^{n}: x_0>\sqrt{x_1^2+\cdots +x_n^2}\}\subset \bbR\oplus \bbR^n$ for $n\geq 1$ (the \emph{Lorentz or light cone});
\item $\P_n(\bbR)=\Herm_n^{>0}(\bbR):=\{M\in \Herm_n(\bbR): M>0\}\subset \Herm_n(\bbR)$ for $n\geq 3$;
\item $\P_n(\bbC)=\Herm_n^{>0}(\bbC):=\{M\in \Herm_n(\bbC): M>0\}\subset \Herm_n(\bbC)$ for $n\geq 3$;
\item $\P_n(\bbH)=\Herm_n^{>0}(\bbH):=\{M\in \Herm_n(\bbH): M>0\}\subset \Herm_n(\bbH)$ for $n\geq 3$;
\item $\P_3(\bbO)=\Herm_3^{>0}(\bbO):=\{M\in \Herm_3(\bbO): M>0\}\subset \Herm_3(\bbO)$.
\end{enumerate}
\end{thm}

\begin{table}[!ht]
\begin{tabular}{|c|c|c|c|c|c|}
\hline
Rank & Dimension & Cone \\
\hline \hline
$2$ & $ n+1$ &  $\P(1,n)= \{\un x\in \bbR\oplus \bbR^{n}: x_0>\sqrt{x_1^2+\cdots +x_n^2}\} $ \\
\hline
$n$ & $\binom{n+1}{2}$ &   $\P_n(\bbR)= \Herm_n^{>0}(\bbR) \:  (n\geq 3)$  \\
\hline
$n$ & $n^2$ &   $\P_n(\bbC)= \Herm_n^{>0}(\bbC) \:  (n\geq 3)$  \\
\hline
$n$ & $n(2n-1)$ &   $\P_n(\bbH)= \Herm_n^{>0}(\bbH) \:  (n\geq 3)$  \\
\hline
$3$ & $ 27$ &   $ \P_3(\bbO)=\Herm_3^{>0}(\bbO) $  \\
\hline
\end{tabular}
\caption{Irreducible symmetric cones}\label{F:irrecones}
\end{table}

The closure $\ov \Omega$ of a symmetric cone $\Omega\subset V$ can be decomposed into a disjoint union of boundaries components, which we are now going to define.
Recall first that an \emph{idempotent}Ê of an Euclidean Jordan algebra $(V,\circ)$ is an element $e\in V$ such that $e\circ e=e$. The operator $T_e$ (see Definition \ref{D:Jalg}) is
self-adjoint with respect to the positive definite
scalar product $\tau$ given by the trace form (see Lemma \ref{L:Jprop}\eqref{L:Jprop3}) and its eigenvalues are  $0$, $1/2$ and $1$ (see \cite[Prop. III.1.3]{FK}). Therefore, we get an orthogonal decomposition of $V$  into eigenspaces
\begin{equation}\label{E:eigenLe}
V=V(e,1)\oplus V(c,1/2)\oplus V(c,0),
\end{equation}
relative to, respectively, the eigenvalues $0$, $1/2$ and $1$.

\begin{lemma}\label{L:Ve1}
For any idempotent $e\in (V,\circ)$, we have that $V(e,1)$ is an Euclidean Jordan subalgebra of $(V,\circ)$ such that $e\in V(e,1)$ is the identity element.
\end{lemma}
\begin{proof}
See \cite[Prop. IV.1.1]{FK}.
\end{proof}

Consider now the Euclidean Jordan algebra $(V,\circ_{\Omega})$ corresponding to a symmetric cone $\Omega\subset V$ according to Theorem \ref{T:bijJorcon}.
\begin{defi}\label{D:boun-con}
For an idempotent $e\in (V,\circ_{\Omega})$, we define the \emph{boundary component}Ê of $\Omega$ associated to $e$ as the symmetric cone
$$\Omega(e):=\Omega_{(V(c,1),\circ_{\Omega})}\subset V(c,1)$$
corresponding to the Euclidean Jordan algebra $(V(c,1),\circ_{\Omega})$ according to Theorem \ref{T:bijJorcon}.
\end{defi}

The closure $\ov \Omega$ of the symmetric cone $\Omega$ in $V$ can be partitioned into the disjoint union of boundaries components as it follows.

\begin{thm}\label{T:boun-con}
Notations as before.
\begin{enumerate}[(i)]
\item \label{T:boun-con1} For any idempotent $e\in (V, \circ_{\Omega})$, the intersection of $\ov\Omega\subset V$ with the subspace $V(e,1)\subset V$ is equal to the closure $\ov\Omega(e)$.
In particular, $\Omega(e)$ is contained in $\ov \Omega$.
\item \label{T:boun-con2} We have that
\begin{equation}\label{E:decboun}
\ov \Omega=\coprod_{e} \Omega(e),
\end{equation}
where the disjoint union varies over all the idempotents $e\in (V,\circ_{\Omega})$ and, for any an idempotent $e$,  the closure of $\Omega(e)$ is a disjoint union of boundary components.
\item \label{T:boun-con3} The group $G(\Omega)^o$ of automorphisms of $\Omega$ acts on $\ov \Omega$ by permuting its boundary components.
\end{enumerate}
\end{thm}
\begin{proof}
See \cite[Chap. II,\S 3]{AMRT}.
\end{proof}
Using \eqref{T:boun-con2}, we can introduce an order relation on the set of idempotents of $(V, \omega_{\circ})$ by saying that $e\geq e'$ if and only if $\ov \Omega(e) \supseteq \Omega(e')$.
Indeed, it turns out that $e\geq e'$ if and only if $e=e'+e''$ for a certain idempotent $e''$ such that $e'\circ_{\Omega} e''=0$.

\begin{example}
\noindent
\begin{enumerate}[(i)]
\item For $F=\bbR,\bbC,\bbH,\bbO$ and $n\geq 2$ (with the convention that $n\leq 3$ if $F=\bbO$), consider the symmetric cone $\P_n(F)= \Herm_n^{>0}(F)\subset \Herm_n(F)$
(as in Theorem \ref{T:clascone}) which is associated to the Euclidean Jordan algebra $(\Herm_n(F),\circ)$ of Example \ref{E:Jor}\eqref{E:Jor3}.
Every idempotent of $\Herm_n(F)$ is conjugate by $G(\P_n(F))^o$ to the idempotent
$$e_p:=\begin{pmatrix}I_p & 0\\ 0 & 0 \end{pmatrix}$$
for some $0\leq p\leq n$. The eigenspaces \eqref{E:eigenLe}Ê of $T_{e_p}$ are given by
$$
\begin{sis}
& V(e_p,1) =\left\{\begin{pmatrix}ÊA & 0 \\ 0 & 0\end{pmatrix}: A\in \Herm_p(F) \right\},\\Ê
& V(e_p,1/2)=\left\{\begin{pmatrix}Ê0 &  B \\ \ov{B}^t  & 0\end{pmatrix}: B\in M_{p,n-p}(F) \right\},\\Ê
& V(e_p,0) =\left\{\begin{pmatrix}Ê0 & 0 \\ 0 & C \end{pmatrix}: C\in \Herm_{n-p}(F) \right\}.
\end{sis}
$$
The boundary component associated to $e_p$ is equal to
$$\P_n(F)(e_p) =\left\{\begin{pmatrix}ÊA & 0 \\ 0 & 0\end{pmatrix}: A\in \Herm_p^{>0}(F) \right\}\subset \Herm_n^{\geq 0}(F)=\ov{\P_n(F)}.$$
Note that the above idempotents $\{e_p\}$ are such that  $0=e_0\leq \cdots \leq e_p \leq \cdots \leq e_n=I_n$.
\item Consider the Lorentz cone $\P(1,n)$ of Theorem \ref{T:clascone} which is associated to the Euclidean Jordan algebra $(\bbR\oplus \bbR^n,\circ_{B})$ of Example \ref{E:Jor}\eqref{E:Jor2},
where $B((x_1,\ldots,x_n)):=x_1^2+\cdots+x_n^2$ is the standard quadratic form on $\bbR^n$.  By abuse of notation, we will denote also by $B$ the symmetric bilinear form associated to the quadratic
form $B$. Every non-trivial idempotent (i.e. different from $(0,0)$ and $(1,0)$) is of the form
$$e_w=\left(\frac{1}{2},w\right) \text{ where }ÊB(w)=\frac{1}{4}.$$
The eigenspaces \eqref{E:eigenLe}Ê of $T_{e_w}$ are given by
$$
\begin{sis}
& V(e_w,1) =\bbR\cdot e_w,\\Ê
& V(e_w,1/2)=\left\{(0,v): B(v,w)=0\right\},\\Ê
& V(e_w,0) =\bbR\cdot e_{-w}.
\end{sis}
$$
The boundary component associated to $e_w$ is equal to
$$\P_{1,n}(e_w) = \bbR_{>0}\cdot e_w \subset \ov{\P_{1,n}}.$$
\end{enumerate}
\end{example}

For the symmetric cone $C(F)$ associated to a boundary component $F$ of a bounded symmetric domain $D$, as in Theorem \ref{T:F-cone}\eqref{T:F-cone2},  we can explicitly describe
its boundary components in terms of boundary components of $D$ that dominates $F$.

\begin{thm}\label{T:bounCF}
Let $D$ be a bounded symmetric domain and let $F\leq D$ be a boundary component.
\begin{enumerate}[(i)]
\item \label{T:bounCF1}
If $F\leq F'$ then $U(F)\supseteq U(F')$ and we have the equality $\ov{C(F')}=\ov{C(F)}\cap U(F')\subset U(F)$. Moreover, $C(F')$ is a boundary component of the symmetric cone $C(F)$.
\item \label{T:bounCF2}
There is an order-reversing bijection
\begin{equation}\label{E:revboun}
\begin{aligned}
\left\{F'\leq D: F\leq F'\leq D \right\}Ê& \stackrel{\cong}{\longrightarrow} \left\{\text{Boundary components of } C(F)Ê\right\}\\
F' & \mapsto C(F')\subset \ov{C(F)}.
\end{aligned}
\end{equation}
\end{enumerate}
\end{thm}
\begin{proof}
See \cite[Chap. III, Thm. 4.8]{AMRT}.
\end{proof}

\subsection{Siegel domains}\label{SS:Siegel}

The presentation of a bounded symmetric domain $D$ as a Siegel domain of the third kind with respect to a given boundary component $F\leq D$ (see
Corollary \ref{C:SiegIII}) assumes a nicer form when the boundary component $F$ is a point, in which case it gives rise to a presentation of $D$
as a \emph{Siegel domain of the second type}, or simply a \emph{Siegel domain} (following the terminology of \cite{Sat}).
The aim of this subsection is to introduce and study Siegel domains.

\begin{defi}\label{D:Sieg}
Let $\Omega$ be an open (convex and pointed) cone in a real vector space $U$. Let $V$ be a complex vector space and let $H:V\times V\to U_{\bbC}$ be a Hermitian map ($\bbC$-linear in the second variable and $\bbC$-antilinear in the first variable). Assume that $H$ is $\Omega$-positive, i.e.
$$H(v,v)\in \Omega\setminus \{0\} \text{ for any } 0\neq v\in V.$$
The \textbf{Siegel domain} associated to $(U,V,\Omega,H)$ is given by
\begin{equation}\label{E:Sieg2}
\SS=\SS(U,V,\Omega,H):=\{(u,v)\in U_{\bbC}\times V: \Im u-H(v,v)\in \Omega\}\subset U_{\bbC}\times V.
\end{equation}
In the special case where $V=\{0\}$, then
\begin{equation}\label{E:Sieg1}
\SS=\SS(U,\Omega):=\{u\in U_{\bbC}: \Im u\in \Omega\}\subset U_{\bbC}
\end{equation}
is called a \emph{Siegel domain of the first kind}, or a \textbf{tube domain}.
\end{defi}

The following result is due to Pyateskii-Shapiro \cite{PS} (see also \cite[Chap. III, Prop. 6.1]{Sat}).

\begin{thm}
Every Siegel domain is holomorphically equivalent to a bounded domain.
\end{thm}

There is a nice characterization (due to Satake) of the Siegel domains that are holomorphically equivalent to a bounded symmetric domain.
In order to present such a characterization, we need to introduce some notations. Consider the setting of Definition \ref{D:Sieg}.
Assume furthermore that $\Omega\subset U$ is symmetric with respect to a scalar product $\langle,\rangle$ on $U$ (see Definition \ref{D:symcon})
and extend $\langle, \rangle$ to a $\bbC$-bilinear symmetric form on $U_{\bbC}\times U_{\bbC}$.
Choose a base point $e\in \Omega$ such that
\begin{equation}\label{E:point-e}
\G(\Omega)_e=\G(\Omega)\cap \O(V,\langle, \rangle),
\end{equation}
which is possible by Theorem \ref{T:str-cones}\eqref{T:str-cones1}. Denote by $\circ_{\Omega}$ the Jordan product on $U$ defined by mean of \eqref{E:Jorcone}
and extend it linearly to $U_{\bbC}$.  Recall that $(U,\circ_{\Omega})$ is an Euclidean Jordan algebra with unit element $e$ (see Theorem \ref{T:bijJorcon}).
Define now a positive definite Hermitian form $h$ on $V$ by
\begin{equation}\label{E:form-h}
h(v,v')=\langle e,H(v,v')\rangle \text{ for any } v,v'\in V.
\end{equation}
Using $h$, we can define for any $u\in U_{\bbC}$ an endomorphism $R_u\in \End(V)$ by mean of the formula
\begin{equation}\label{E:def-R}
\langle u, H(v,v')\rangle= 2h(v, R_u v') \text{ for any } v,v'\in V.
\end{equation}
It is easily checked from \eqref{E:def-R} that $\ov{R}^*=R_{\ov u}$, where the adjoint ${}^*$  is with respect to the Hermitian form $h$.
 In particular, if $u\in U$ then
$$R_u\in \Herm(V,h):=\{f\in \End(V): f^*=f\}.$$

\begin{thm}\label{T:Siegsym}
Notations as above.
The Siegel domain $\SS(U,V,\Omega,H)$ is biholomorphic to a bounded symmetric domain (in which case we say that it is \emph{symmetric}) if and only if
\begin{enumerate}[(i)]
\item \label{T:Siegsym1} $\Omega\subset U$ is a symmetric cone with respect to a scalar product $\langle,\rangle$ on $U$;
\item \label{T:Siegsym2} $u\circ_{\Omega} H(v,v')=H(R_u v, v')+H(v,R_u v')$ for any $u\in U$ and any $v,v'\in V$;
\item \label{T:Siegsym3} $H(R_{H(v'',v')}v,v'')=H(v',R_{H(v,v'')}v'')$ for any $v,v',v''\in V$.
\end{enumerate}
\end{thm}
\begin{proof}
See \cite[Chap. V, Thm. 3.5]{Sat}.
\end{proof}

For a tube domain, the conditions \eqref{T:Siegsym2} and \eqref{T:Siegsym3} of Theorem \ref{T:Siegsym} are trivially satisfies. Therefore, we get the following

\begin{cor}\label{C:tubesym}
The tube domain $\SS(U,\Omega)$ is biholomorphic to a bounded symmetric domain if and only if $\Omega\subset U$ is a symmetric cone.
\end{cor}

Now we want to classify the symmetric Siegel domains, or equivalently that satisfy the three conditions of Theorem \ref{T:Siegsym}. Actually, it is possible to classify the following bigger class of Siegel domains.

\begin{defi}\label{D:quasi-sym}
A Siegel domain $\SS(U,V,\Omega,H)$ is said to be \emph{quasi-symmetric} if and only if it satisfies the first two conditions of Theorem \ref{T:Siegsym}.
\end{defi}

Indeed, quasi-symmetric Siegel domains correspond to unital Jordan algebra representations of $(U,\circ_{\Omega})$ into $\Herm(V,h)$.

\begin{lemma}\label{L:repJor}
Fix a symmetric cone $\Omega\subset U$ and keep the notations as above.
\begin{enumerate}[(i)]
\item If $H:V\times V\to U_{\bbC}$ is $\Omega$-positive Hermitian map (as in Definition \ref{D:Sieg}) satisfying Theorem \ref{T:Siegsym}\eqref{T:Siegsym2}
then
\begin{equation}\label{E:2R}
\begin{aligned}
\rho:=2R: (U,\circ_{\Omega}) & \longrightarrow \Herm(V,h) \\
u & \mapsto 2R_u
\end{aligned}
\end{equation}
is a unital Jordan algebra homomorphism in the sense of Remark \ref{R:funJorcon} (we call it a \emph{complex representation}Ê of $(U,\circ_{\Omega})$).
\item Conversely, if we start from a unital Jordan algebra homomorphism \eqref{E:2R} and we define a Hermitian map $H:V\times V\to U_{\bbC}$ by mean of \eqref{E:def-R}, then $H$ is $\Omega$-positive and it satisfies Theorem \ref{T:Siegsym}\eqref{T:Siegsym2}.
\end{enumerate}
\end{lemma}
\begin{proof}
See \cite[Chap. IV, Prop. 4.1]{Sat}.
\end{proof}

Each quasi-symmetric Siegel domain is a product of irreducible quasi-symmetric Siegel domains, which we are now going to define.

\begin{defi}\label{D:irrSieg}
\noindent
\begin{enumerate}[(i)]
\item Let $\SS_1=\SS(U_1,V_1,\Omega_1,H_1)$ and $\SS_2=\SS(U_2,V_2,\Omega_2,H_2)$ two Siegel domains. The product $\SS_1\times \SS_2$ is equal to the
Siegel domain $\SS(U_1\oplus U_2, V_1\oplus V_2,\Omega_1+ \Omega_2,H=H_1\oplus H_2)$, where $H=H_1\oplus H_2:(V_1\oplus V_2)\times (V_1\oplus V_2)\to U_1\oplus U_2$ is defined by
$H_{|V_1\times V_2}\equiv 0$,  $H_{|V_1\times V_1}\equiv H_1$ and $H_{|V_2\times V_2}\equiv H_2$.
\item A Siegel domain is \emph{irreducible} if and only if it cannot be written as the product of two non-trivial Siegel domains.
\end{enumerate}
\end{defi}

\begin{thm}\label{T:dec-Sieg}
\noindent
\begin{enumerate}[(i)]
\item \label{T:dec-Sieg1} A quasi-symmetric Siegel  domain $\SS(U,V,\Omega,H)$ is irreducible if and only if $\Omega\subset U$ is irreducible.
\item \label{T:dec-Sieg2} Any quasi-symmetric (resp. symmetric) Siegel domain decomposes uniquely as the product of irreducible quasi-symmetric (resp. symmetric) Siegel domains.
\end{enumerate}
\end{thm}

According to Theorem \ref{T:bijJorcon}, Lemma \ref{L:repJor}Ê and Theorem \ref{T:dec-Sieg}, an irreducible quasi-symmetric Siegel domain is built up from  a simple Euclidean Jordan algebra
$(U,\circ)$ together with a  complex representation $\rho:(U, \circ)\to \Herm(V,h)$. Such pairs can be classified as it follows.

\begin{thm}\label{T:clas-rep}
The complex representations of the simple Euclidean Jordan algebras are given as it follows:
\begin{enumerate}[(i)]
\item Type $IV_{n; r,s} \: (\text{even } n \geq 4; r\geq s\geq 0)$: the representation $\rho_{r,s}={\rm sp}_1^{\oplus r}\oplus {\rm sp}_2^{\oplus s}$ of
$(\bbR\oplus \bbR^{n-1},\circ_Q)$, where ${\rm sp}_1$ and ${\rm sp}_2$ are the two spin representations (see \cite[Appendix, \S4-6]{Sat});
\item Type $IV_{n; r} \: (\text{odd } n \geq 3 \text{ or } n=2;    r\geq 0)$:  the representations $\rho_r={\rm sp}^{\oplus r}$ of  $(\bbR\oplus \bbR^{n-1},\circ_Q)$, where ${\rm sp}$ is the spin representation (see \cite[Appendix, \S4-6]{Sat});
\item Type $III_{n; r} \: (n\geq 3; r\geq 0)$: the representations $\rho_r=\id^{\oplus r}$ of $\Herm_n(\bbR)$, where $\id:\Herm_n(\bbR)\to \Herm_n(\bbC)$ is the natural injection.
\item Type $I_{n;r,s}   \: (n\geq 3; r\geq s\geq 0)$: the representations  $\rho_{r,s}=\id^{\oplus r}\oplus \ov{\id}^{\oplus s}$ of $\Herm_n(\bbC)$, where $\id:\Herm_n(\bbC)\to \Herm_n(\bbC)$ is the identity homomorphism and $\ov{\id}: \Herm_n(\bbC)\to \Herm_n(\bbC)$ is given by sending $A$ into its complex conjugate $\ov{A}$.
\item Type $II_{n;r}\: (n\geq 3; r\geq 0)$: the representations $\rho_r=\id^{\oplus r}$ of $\Herm_n(\bbH)$, where
$$\begin{aligned}
 \id:\Herm_n(\bbH)& \longrightarrow \Herm_{2n}(\bbC)\\
A+ j B & \mapsto \begin{pmatrix} A & B \\ -\ov{B} & \ov{A} \end{pmatrix} \\
\end{aligned}$$
\item Type $IV_0$: the trivial representation $\rho_0$ of $\Herm_3(\bbO)$.
\end{enumerate}
\end{thm}
\begin{proof}
See \cite[Chap. V, \S 5]{Sat} and the references therein.
\end{proof}

In Table \ref{F:qs-Siegel}, we have listed all the irreducible quasi-symmetric Siegel domains by specifying the irreducible symmetric cone
and the complex representation of their associated simple Euclidean Jordan algebra. Moreover, in the last column, we have specified the quasi-symmetric Siegel domains that are also symmetric (see \cite[Chap. V, \S 5]{Sat} and the references therein) together with the corresponding bounded symmetric domains (using the notations of Table \ref{F:BSD}) to which they are biholomorphic.

\begin{table}[!ht]
\begin{tabular}{|c|c|c|c|c|c|}
\hline
Type & Symmetric Cone & Complex representation & Symmetric cases \\
\hline \hline
$ \begin{aligned}
& IV_{n; r,s} \: (r\geq s\geq 0) \\
& \text{ even } n \geq 4  \\
\end{aligned} $
 & $\P(1,n-1) $ & $\rho_{r,s}={\rm sp}_1^{\oplus r}\oplus {\rm sp}_2^{\oplus s}$ &
 $\begin{aligned}
 & IV_{n;0,0}=IV_{n} \\
 & IV_{4;r,0}=I_{r+2,2}\\
 & IV_{6;1,0}=II_5 \\
 & IV_{8;1,0}=V
 \end{aligned} $\\
\hline
$ \begin{aligned}
& IV_{n; r} \: (r\geq 0)  \\
& \text{ odd } n \geq 3 \text{ or } n=2
\end{aligned} $
  & $\P(1,n-1)$ & $\rho_r={\rm sp}^{\oplus r}$ &
  $\begin{aligned}
  & IV_{n;0}=IV_n\: (n\geq 3) \\
  & IV_{2;r}=I_{r+1,1}
  \end{aligned}$ \\
\hline
$III_{n; r} \: (n\geq 3, r\geq 0)  $
& $\P_n(\bbR)= \Herm_n^{>0}(\bbR) $ & $\rho_r=\id^{\oplus r}$ & $III_{n;0}=III_n$ \\
\hline
$I_{n;r,s}   \: (n\geq 3, r\geq s\geq 0)$ &  $\P_n(\bbC)= \Herm_n^{>0}(\bbC) $  & $\rho_{r,s}=\id^{\oplus r}\oplus \ov{\id}^{\oplus s}$ & $I_{n;r,0}=I_{n+r,n}$ \\
\hline
$II_{n;r}  \: (n\geq 3, r\geq 0) $ & $\P_n(\bbH)= \Herm_n^{>0}(\bbH) $  & $\rho_r=\id^{\oplus r}$ & $II_{n;r}=II_{2n+r}\: (r=0,1)$ \\
\hline
$VI_0$ &  $ \P_3(\bbO)=\Herm_3^{>0}(\bbO) $  & $\rho_0=0$ & $VI_0=VI$ \\
\hline
\end{tabular}
\caption{Irreducible quasi-symmetric Siegel domains}\label{F:qs-Siegel}
\end{table}

By looking at the last column of Table \ref{F:qs-Siegel} and using the isomorphisms between bounded symmetric domains of small dimension belonging to different types (see Table \ref{F:small}),
it is easy to see that every bounded symmetric domain is biholomorphic to a unique Siegel domain. As a consequence, there exists a bijection between
symmetric Siegel domains and Hermitian positive JTSs, which we are now going to make explicit.

Start with a symmetric Siegel domain $\SS=\SS(U,V,\Omega,H)$. By Theorem \ref{T:Siegsym}\eqref{T:Siegsym1}, the cone $\Omega\subset U$ is symmetric with respect to a scalar product $\langle,\rangle$. Keeping the notations introduced before Theorem \ref{T:Siegsym}, we get a Jordan product $\circ_{\Omega}$ on $U$ which we extend linearly to $U_{\bbC}$. Using the fact that $(U,\circ_{\Omega})$ is an Euclidean Jordan algebra, it can be checked (see \cite[Chap. I, \S6]{Sat}) that
$U_{\bbC}$ becomes a Hermitian positive JTSs with respect to the triple product
\begin{equation}\label{E:U-JTS}
\{u_1,u_2,u_3\}_{\Omega}:=(u_1\circ_{\Omega} \ov{u}_2)\circ_{\Omega} u_3+ (u_3\circ_{\Omega} \ov{u}_2)\circ_{\Omega} u_1-(u_1\circ_{\Omega} u_3)\circ_{\Omega} \ov{u}_2.
\end{equation}
Define a triple product on $U_{\bbC}\oplus V$ as it follows
\begin{equation}\label{E:HerJTS}
\left\{\begin{pmatrix}u_1 \\ v_1\end{pmatrix},\begin{pmatrix} u_2 \\ v_2\end{pmatrix},\begin{pmatrix}Êu_3 \\ v_3\end{pmatrix}\right\}_{\SS}:=
\begin{pmatrix}
\{u_1,u_2,u_3\}_{\Omega} +2H(R_{\ov{u}_3}v_2,v_1)+2H(R_{\ov{u}_1}v_2,v_3) \\
2R_{u_3}R_{\ov{u}_2}v_1+2R_{u_1}R_{\ov{u}_2}v_3+2 R_{H(v_2,v_1)}v_3+2R_{H(v_2,v_3)}v_1
\end{pmatrix}.
\end{equation}
Using that $\SS(U,V,\Omega,H)$ is symmetric, it can be shown that $(U_{\bbC}\oplus V, \{.\,,.\,,.\}_{\SS})$ is a Hermitian positive JTS (see \cite[Chap. V, Thm. 6.9]{Sat} and the discussion following it).
Observe that the element $\wt{e}:=\begin{pmatrix} e \\ 0\end{pmatrix}\in U_{\bbC}\oplus V$ is an tripotent of the Hermitian positive JTS $(U_{\bbC}\oplus V, \{.\,,.\,,.\}_{\SS})$, i.e.
$\left\{\wt{e},\wt{e},\wt{e} \right\}_{\SS}=\wt{e}$.
Moreover, using the fact that $\ov{e}=e$ is the identity element of $(U_{\bbC},\{.\,,.\,,.\}_{\Omega})$ and that $2R_e=\id_V$ by Lemma \ref{L:repJor}, we can easily compute
\begin{equation}\label{e-e}
\left[\wt{e} \, \square\, \wt{e}Ê\right]\begin{pmatrix}u \\v\end{pmatrix}:=
\left\{\begin{pmatrix}e \\ 0\end{pmatrix}, \begin{pmatrix}e \\ 0\end{pmatrix}, \begin{pmatrix}u \\ v\end{pmatrix}\right\}_{\SS}=\begin{pmatrix}u \\ \frac{v}{2}\end{pmatrix}.
\end{equation}
In other words, $1$ and $\frac{1}{2}$ are the only eigenvalues of $\wt{e} \, \square\, \wt{e}$ with associated eigenspaces
\begin{equation}\label{}
V(\wt{e} \, \square\, \wt{e}; 1)=U_{\bbC}\, \text{ and }Ê\, V(\wt{e} \, \square\, \wt{e}; 1/2)=V.
\end{equation}

Conversely, start with a Hermitian positive JTS $(W,\{.\,,.\,,.\})$ and choose a tripotent $e\in W$, i.e. an element of $W$ such that $\{e,e,e\}=e$. The endomorphism $e\square e\in \End(W)$ is semisimple and it
satisfies the equation $(e\square e-1)(2e\square e-1)(e\square e)=0$ (see \cite[p. 242]{Sat}).  Therefore, the possible eigenvalues of $e\square e$ are  $1$, $1/2$ and $0$.
We can furthermore choose $e$ in such a way that $0$ is not an eigenvalue, in which case $e$ is called \emph{principal} (see \cite[Chap. V, \S6, Ex. 5]{Sat}). With this assumption,
we get a decomposition
\begin{equation}\label{E:decW}
W=W_{1}\oplus W_{1/2}=W(e\square e; 1)\oplus W(e\square e; 1/2)
\end{equation}
into eigenspaces for $e\square e$ relative to the eigenvalues $1$ and $1/2$, respectively.  The complex vector space $W_1$ becomes a Jordan algebra with unit element $e\in W_1$ with respect to the
Jordan product (see \cite[Chap. V, Prop. 6.1]{Sat})
\begin{equation}\label{E:JorW1}
a\circ b=\{a,e,b\} \: \: \text{Êfor any }Êa,b\in W_1.Ê
\end{equation}
Moreover, the map $a\mapsto a^*:=\{e,a,e\}$ is a $\bbC$-antilinear involution on $W_1$ (see \cite[Chap. V, Prop. 6.1]{Sat}).
Therefore
\begin{equation}\label{E:W1+}
\left(W_1^+:=\{a\in W_1\: : a^*=a\}, \circ\right)
\end{equation}
is a real Jordan algebra which turns out to be Euclidean (see \cite[p. 254]{Sat}). We will denote by $\Omega_W\subset W_1^+$ its associated symmetric cone (see Theorem \ref{T:bijJorcon}).
The Jordan algebra $(W_1,\circ)$ comes with a unital Jordan algebra homomorphism (see \cite[Chap. V, Prop. 6.2]{Sat})
\begin{equation}\label{E:def2R}
\begin{aligned}
2R: W_1& \longrightarrow \End(W_{1/2}),\\
a & \mapsto 2 R_a \: \: \text{Ês. t. }ÊR_a(x):=\{a,e,x\}.
\end{aligned}
\end{equation}
It can be checked (see \cite[p. 247, Eq. (6.21)]{Sat}) that $R_{a^*}=R_a^*$, where $R_a^*$ is the adjoint of $R_a$ with respect to the positive definite hermitian  form $h=\tau/2$ on $W_{1/2}$
with $\tau$ equal to the trace form of $(W,\{.\,,.\,,.\})$. Therefore, by restriction, we get a unital Jordan algebra homomorphism
\begin{equation}\label{E:def2Rbis}
2R: W_1^+ \longrightarrow \Herm(W_{1/2}, h).
\end{equation}
Consider now the Hermitian map
\begin{equation}\label{E:defiH}
\begin{aligned}
H_W: W_{1/2}\times W_{1/2} & \longrightarrow W_1,\\
(x,y) & \mapsto H(x,y):=\{e,x,y\}.
\end{aligned}
\end{equation}
It can be checked (see \cite[p. 247, Eq. (6.25)]{Sat})  that the map $H_W$ and the unital Jordan algebra homomorphism $2R$ satisfy formula \eqref{E:def-R}, i.e.
$$\langle a, H_W(x,y)\rangle= 2h(x, R_a y) \text{ for any } x,y \in W_{1/2} \text{Êand any } a\in W_{1},$$
where $\langle,\rangle$ is the trace form of the Jordan algebra $(W_1,\circ)$.  Therefore, from Lemma \ref{L:repJor} it follows that $H_W$ is $\Omega$-positive and it satisfies
Theorem \ref{T:Siegsym}\eqref{T:Siegsym2}. Moreover, it can be checked (see \cite[p. 245, Eq. (6:15'')]{Sat}) that $H_W$ satisfies Theorem \ref{T:Siegsym}\eqref{T:Siegsym3}.
Therefore, using Theorem \ref{T:Siegsym}, we infer that $\SS(W_1^+, W_{1/2}, \Omega_W, H_W)$ is a symmetric Siegel domain.

\begin{thm}\label{T:SieJTS}
Notations as above. There is a bijection
\begin{equation}
\begin{aligned}
\left\{\text{Symmetric Siegel spaces}\right\} & \stackrel{\cong}{\longrightarrow} \left\{\text{Hermitian positive JTSs}\right\}Ê\\
\SS=\SS(U,V,\Omega,H) & \longrightarrow (U_{\bbC}\oplus V, \{.\,,.\,,.\}_{\SS})\\
\SS(W_1^+, W_{1/2}, \Omega_W, H_W) & \longleftarrow (W,\{.\,,.\,,.\})
\end{aligned}
\end{equation}
sending irreducible symmetric Siegel domains into simple Hermitian positive JTSs.
\end{thm}
\begin{proof}
See \cite[Chap. V, \S6]{Sat}.
\end{proof}

For symmetric Siegel spaces of the first kind (or symmetric tube domains), the bijection of Theorem   \ref{T:SieJTS} assumes a particular simple form.
Indeed, let $\Omega\subset U$ be a symmetric cone and choose a base point $e\in \Omega$ as in \eqref{E:point-e}. Consider the associated Jordan product $\circ_{\Omega}$ on $U$
(see Theorem \ref{T:bijJorcon}). Then the Hermitian positive JTS $\{U_{\bbC},\{.\,,.\,,.\}_{\SS}\}$ associated to the Siegel domain of the first kind $\SS=\SS(\Omega,U)\subset U_{\bbC}$ (as in  \eqref{E:Sieg2})
is the one associated to the Euclidean Jordan algebra $(U, \circ_{\Omega})$ as in Example  \ref{E:Jor}\eqref{E:JA-JTS}.

Consider now the bounded symmetric domain $D_{\Omega}\subset U_{\bbC}$
(in its Harish-Chandra embedding) corresponding to the Hermitian positive JTS $\{U_{\bbC},\{.\,,.\,,.\}_{\SS}\}$ (see Theorem \ref{T:HSM-BSD} and Theorem \ref{T:LieJor}).
It is possible to describe explicitly the biholomorphism between $\SS(\Omega,U)$ and $D_{\Omega}$, generalizing the Cayley transform in dimension one
(see Example \ref{E:unitdisk}).

\begin{thm}\label{T:Cayley}
Notations as above.
Then we have the following biholomorphism (called the \emph{generalized Cayley transform})
\begin{equation}\label{E:Cayley}
\begin{aligned}
c: D_{\Omega} & \stackrel{\cong}{\longrightarrow} \SS(\Omega,U) \\
w & \mapsto i(e+w)\circ_{\Omega}(e-w)^{-1},
\end{aligned}
\end{equation}
The inverse is given by the map sending $z\in \SS(\Omega,U)$ into $\displaystyle (z-ie)\circ_{\Omega} (z+ie)^{-1}\in D_{\Omega}$.
\end{thm}

Indeed, generalized Cayley transforms have been defined for all symmetric Siegel spaces by Kor\'anyi-Wolf in \cite[Chap. VI]{KW}.

Looking at the last column of Table \ref{F:qs-Siegel}, it is easy to see which irreducible bounded symmetric domains are Êbiholomorphic to symmetric tube domains (we call them bounded symmetric domains \emph{of tube type}).

\begin{cor}\label{C:tube}
The irreducible bounded symmetric domains of tube type are the following:
\begin{enumerate}
\item $I_{n,n}$ for any $n\geq 1$;
\item $II_{2n}$ for any $n\geq 1$;
\item $III_n$ for any $n\geq 1$;
\item $IV_n$ for any $2\neq n\geq 1$;
\item $VI$.
\end{enumerate}
\end{cor}

We have collected the irreducible bounded symmetric domains of tube type (avoiding repetitions in small dimension, see Table \ref{F:small}) in the following Table \ref{F:BSDtube}, together with their associated symmetric cones.Ê

\begin{table}[!ht]
\begin{tabular}{|c|c|c|c|c|c|}
\hline
Symmetric Cone & Bounded symmetric domain of tube type \\
\hline \hline
$\P(1,n-1) \:\:\:  (n\geq 2)$ &
  $\begin{aligned}
  & IV_n\: & \text{ if }Ên\geq 3 \\
  & IV_1\: & \text{ if }Ên=2
  \end{aligned}$ \\
\hline
$\P_n(\bbR)= \Herm_n^{>0}(\bbR)  \:\:\:  (n\geq 3)$ & $III_n$ \\
\hline
  $\P_n(\bbC)= \Herm_n^{>0}(\bbC)  \:\:\:  (n\geq 3)$   & $I_{n,n}$ \\
\hline
 $\P_n(\bbH)= \Herm_n^{>0}(\bbH)  \:\:\:  (n\geq 3)$  & $II_{2n}$ \\
\hline
  $ \P_3(\bbO)=\Herm_3^{>0}(\bbO) $   & $VI$ \\
\hline
\end{tabular}
\caption{Irreducible bounded symmetric domains of tube type}\label{F:BSDtube}
\end{table}

\subsection{Boundary components of irreducible bounded symmetric domains}\label{SS:boun-irr}

In this subsection, we describe explicitly the boundary components of each of the irreducible bounded symmetry domains (see \S\ref{S:irrHSM}).
We begin with the following result.

\begin{thm}\label{T:1boun}
Let $D$ be an irreducible bounded symmetric domain and let $G=\Aut(D)^o$. Then all the boundary components of rank $k$ are conjugated by the group $G$.
\end{thm}
\begin{proof}
See \cite[p. 292]{Wol3}.
\end{proof}

In virtue of the above Theorem, it will be enough to describe for each irreducible symmetric domain $D$ of rank $r$ and each $0\leq k< r$ a boundary component $F\leq D$ of rank $k$.

\subsubsection{Type $I_{p,q}$ ($p\geq q\geq 1$)}\label{SS:bounI}

Every boundary component of $\calD_{I_{p,q}}$ of rank $k$ (with $0\leq k<q$) is conjugate to the following boundary component
\begin{equation}\label{E:st-bnd1}
\calD_{I_{p,q}}^k:=\left\{\begin{pmatrix}Z' & 0 \\ 0 & I_{q-k} \end{pmatrix}: Z'\in \calD_{I_{p-q+k,k}}\right\}
\cong \calD_{I_{p-q+k,k}},
\end{equation}
which we call the \emph{standard boundary component of rank $k$}.
The pair $(f_{\calD_{I_{p,q}}^k}, \phi_{\calD_{I_{p,q}}^k})$ of Lemma \ref{L:pair} associated to $\calD_{I_{p,q}}^k$ is equal to
\begin{equation}\label{E:1pair-f}
\begin{aligned}
f_{\calD_{I_{p,q}}^k}: \Delta & \longrightarrow \calD_{I_{p,q}}^k \\
z & \mapsto \begin{pmatrix}0 & 0 \\ 0 & z I_{q-k} \end{pmatrix},\\
\end{aligned}
\end{equation}
\begin{equation}\label{E:1pair-phi}
\begin{aligned}
\phi_{\calD_{I_{p,q}}^k}: \bbS^1\times \SL_2(\bbR) & \longrightarrow \SU(p,q)=\Hol(\calD_{I_{p,q}})^o \\
\left(e^{i\theta}, \begin{pmatrix}a & b \\ c & d \end{pmatrix} \right) & \mapsto
\begin{pmatrix} e^{i\theta}I_{p-q+k} & 0 & 0 & 0 \\ 0 & \frac{a+d+i(b-c)}{2}I_{q-k} & 0 & \frac{b+c+i(a-d)}{2} I_{q-k} \\
0& 0 & e^{-i\theta}I_k & 0 \\ 0 & \frac{b+c-i(a-d)}{2}I_{q-k} & 0 & \frac{a+d-i(b-c)}{2} I_{q-k} \\
\end{pmatrix}.
\end{aligned}
\end{equation}
In particular, the one-parameter subgroup of $\SU(p,q)$ associated to $\calD_{I_{p,q}}^k$ as in \eqref{E:wF} is given by
\begin{equation}\label{E:1wF}
\begin{aligned}
w_{\calD_{I_{p,q}}^k}: \G_m & \longrightarrow \SU(p,q)\\
t & \mapsto
\begin{pmatrix} I_{p-q+k} & 0 & 0 & 0 \\ 0 & \frac{t+t^{-1}}{2}I_{q-k} & 0 & \frac{i(t-t^{-1})}{2} I_{q-k} \\
0& 0 & I_k & 0 \\ 0 & \frac{-i(t-t^{-1})}{2}I_{q-k} & 0 & \frac{t+t^{-1}}{2} I_{q-k} \\
\end{pmatrix}.
\end{aligned}
\end{equation}
Using the above explicit expression of $w_{\calD_{I_{p,q}}^k}$ and Theorem \ref{T:5deco}, we can compute the Levi subgroup of the normalizer $N(\calD_{I_{p,q}}^k)$ subgroup of $\calD_{I_{p,q}}^k$ together with its decomposition as in Theorem \ref{T:5deco}\eqref{T:5deco4}
\begin{equation}\label{E:1Levi}
\begin{aligned}
 Z(w_{\calD_{I_{p,q}}^k})^o &=\left\{
\begin{pmatrix} A & 0 & B & 0 \\ 0 & \frac{E+(E^*)^{-1}}{2} & 0 & i \frac{E-(E^*)^{-1}}{2}  \\
C& 0 & D & 0 \\ 0 & -i \frac{E-(E^*)^{-1}}{2} & 0 & \frac{E+(E^*)^{-1}}{2} \\
\end{pmatrix}\in \SL(p+q,\bbC): 
\begin{aligned}
& \begin{pmatrix}A & B \\ C & D \end{pmatrix}\in \U(p-q+k,k) \\
& E\in \GL_{q-k}(\bbC)
\end{aligned}
\right\}\\
& = \SU(p-q+k,k)\cdot \GL_{q-k}^o(\bbC)\times \bbS^1=G_h(\calD_{I_{p,q}}^k) \cdot G_l(\calD_{I_{p,q}}^k)\cdot M(\calD_{I_{p,q}}^k),
\end{aligned}
\end{equation}
where $\GL_{q-k}^o(\bbC):=\{E\in \GL_{q-k}(\bbC): \det E\in \bbR^*\}\subset \GL_{q-k}(\bbC)$.

Similarly, using again the above explicit expression of $w_{\calD_{I_{p,q}}^k}$ and Theorem \ref{T:5deco}, we can compute  the unipotent radical of $N(\calD_{I_{p,q}}^k)$
\begin{equation}\label{E:1radi}
 W(\calD_{I_{p,q}}^k)=\left\{
\begin{pmatrix} I_{p-q+k} & F_1 & 0 & -iF_1 \\ -\ov{F_1}^t & I_{q-k}+i M & - i \ov{F_2}^t & M  \\
0& i F_2 & I_k & F_2 \\ i\ov{F_1}^t & M  & -\ov{F_2}^t & I_{q-k}-i M  \\
\end{pmatrix}:
\begin{aligned}
& F_1\in M_{p-q+k, q-k}(\bbC), F_2\in M_{k, q-k}(\bbC)\\
& M\in M_{q-k,q-k}(\bbC) \\
& \ov{F_1}^t F_1-\ov{F_2}^t F_2=i(\ov{M}^t-M) \\
\end{aligned}\right\}.
\end{equation}
Moreover, the center $U(\calD_{I_{p,q}}^k)$ of $W(\calD_{I_{p,q}}^k)$ is equal to the set of all matrices of $W(\calD_{I_{p,q}}^k)$ such that $F_1=F_2=0$, and is therefore isomorphic to the abelian unipotent Lie group underlying the vector space  $\Herm_{q-k}(\bbC)$.

The orbit of the point $o_{\calD_{I_{p,q}}^k}=f_{\calD_{I_{p,q}}^k}(1)=\begin{pmatrix}0 & 0 \\ 0 & I_{q-k}\end{pmatrix}$ under the natural action of the group $G_h(\calD_{I_{p,q}}^k)= \SU(p-q+k,k)$
is equal to $\calD_{I_{p,q}}^k$ and its stabilizer
subgroup is isomorphic to $K_h(\calD_{I_{p,q}}^k)= \SU(p-q+k,k) \cap \GS(\U_p\times \U_q)=\GS(U_{p-q+k}\times U_{k})$. Therefore, $\calD_{I_{p,q}}^k$ is a bounded symmetric domain of type 
$I_{p-q+k,k}$ and it is diffeomorphic to
\begin{equation}\label{E:1-presF}
\calD_{I_{p,q}}^k\cong \frac{G_h(\calD_{I_{p,q}}^k)}{K_h(\calD_{I_{p,q}}^k)}=\frac{\SU(p-q+k,k)}{\GS(U_{p-q+k}\times U_{k})}.
\end{equation}

The action of $G_l(\calD_{I_{p,q}}^k)=\GL_{q-k}^o(\bbC)$ on $U(\calD_{I_{p,q}}^k)\cong \Herm_{q-k}(\bbC)$ is given by
$(E,M)\mapsto E M  \ov{E}^t$. Under this action, the orbit of the point
$\Omega_{\calD_{I_{p,q}}^k}=\frac{1}{2} I_{q-k}\in \Herm_{q-k}(\bbC)$ is equal to the cone of positive definite Hermitian complex quadratic forms $\P_{q-k}(\bbC)$ of size $q-k$  and its stabilizer 
subgroup is equal to  $K_l(\calD_{I_{p,q}}^k)=\GL^o_{q-k}(\bbC)\cap \GS(\U_p\times \U_q)=\U^o(q-k)$, where $\U^o(q-k):=\{E\in \U(q-k): \det E=\pm 1\}.$ Ê
 Therefore, $C(\calD_{I_{p,q}}^k)$ is equal to the symmetric cone $\P_{n-k}(\bbC)$ (see Theorem \ref{T:clascone}) and it is diffeomorphic to
\begin{equation}\label{E:1-C(F)}
C(\calD_{I_{p,q}}^k)\cong \frac{G_l(\calD_{I_{p,q}}^k)}{K_l(\calD_{I_{p,q}}^k)}=\frac{\GL^o_{q-k}(\bbC)}{\U^o(q-k)}=\frac{\GL_{q-k}(\bbC)}{\U(q-k)}=\P_{q-k}(\bbC).
\end{equation}

\subsubsection{Type $II_{n}$}\label{SS:bounII}

Every boundary component of $\calD_{II_n}$ of rank $k$ (with $0\leq k< \lfloor \frac{n}{2} \rfloor$) is conjugate to the following boundary component (which we call the \emph{standard boundary component of rank $k$})
\begin{equation}\label{E:st-bnd2}
\calD_{II_n}^k:=\left\{\begin{pmatrix}Z' & 0 \\ 0 & E_{n-2k-\epsilon} \end{pmatrix}: Z'\in \calD_{II_{2k+\epsilon}}\right\}
\cong \calD_{II_{2k+\epsilon}},
\end{equation}
where $\epsilon=0$ (resp. $1$) if $n$ is even (resp. odd) and for any $m\in \bbN$ we denote by $E_{2m}$ the $2m\times 2m$-matrix formed by $m$ diagonal blocks of the form 
$\begin{pmatrix}Ê0 & 1 \\ -1 & 0 \end{pmatrix}$.

The decomposition into factors (as in Theorem \ref{T:5deco}\eqref{T:5deco4})  of the Levi subgroup $L(\calD_{II_n}^k)$ of the normalizer subgroup $N(\calD_{II_n}^k)$ of
$\calD_{II_n}^k$ is given by (see \cite[p. 116]{Sat})
\begin{equation}\label{E:2Levi}
L(\calD_{II_n}^k)=G_h(\calD_{II_n}^k) \cdot G_l(\calD_{II_n}^k)\cdot M(\calD_{II_n}^k)=
\begin{cases}
\{1\}\cdot \GL_{\frac{n}{2}}(\bbH)\cdot \{1\} & \text{ if }Êk=0 \: \text{Êand }Ên \: \text{Êis even}, \\
\{1\}\cdot \GL_{\frac{n-1}{2}}(\bbH)\cdot  \bbS^1 & \text{ if }Êk=0 \: \text{Êand }Ên \: \text{Êis odd}, \\ 
\SU(1,1) \cdot \GL_{\frac{n-2}{2}}(\bbH)\cdot \SL_1(\bbH) & \text{ if }Êk=1 \: \text{Êand }Ên \: \text{Êis even},\\
\SOnc(2n-4)\cdot \bbR^*\cdot \SL_1(\bbH)  & \text{ if }Êk=\frac{n-2-\epsilon}{2},\\
\SOnc(4k+2\epsilon) \cdot GL_{\frac{n-2k-\epsilon}{2}}(\bbH)\times \{1\} & \text{Êotherwise.}Ê 
\end{cases}
\end{equation}
Therefore, $\calD_{II_n}^k$ is a bounded symmetric domain of type $II_{2k+\epsilon}$ if $k\geq 1$ and it is a point if $k=0$.  Moreover, the symmetric cone associated to $\calD_{II_n}^k$ is equal to
 (see Theorem \ref{T:clascone})
\begin{equation}\label{E:2-C(F)}
C(\calD_{II_n}^k)=
\begin{cases}
\P(1,0) & \text{ if }Êk=\frac{n-2-\epsilon}{2}, \\
\P(1,5) & \text{ if }Êk=\frac{n-4-\epsilon}{2},\\
\P_{\frac{n-2k-\epsilon}{2}}(\bbH) & \text{Êotherwise.} Ê 
\end{cases}
\end{equation}



\subsubsection{Type $III_{n}$ ($n\geq 1$)}\label{SS:bounIII}

Every boundary component of $\calD_{III_n}$ of rank $k$ (with $0\leq k< n$) is conjugate to the following boundary component
\begin{equation}\label{E:st-bnd3}
\calD_{III_n}^k:=\left\{\begin{pmatrix}Z' & 0 \\ 0 & I_{n-k} \end{pmatrix}: Z'\in \calD_{III_k}\right\}
\cong \calD_{III_k},
\end{equation}
which we call the \emph{standard boundary component of rank $k$}.
The pair $(f_{\calD_{III_n}^k}, \phi_{\calD_{III_n}^k})$ of Lemma \ref{L:pair} associated to $\calD_{III_n}^k$ is equal to
\begin{equation}\label{E:3pair-f}
\begin{aligned}
f_{\calD_{III_n}^k}: \Delta & \longrightarrow \calD_{III_n}^k \\
z & \mapsto \begin{pmatrix}0 & 0 \\ 0 & z I_{n-k} \end{pmatrix},\\
\end{aligned}
\end{equation}
\begin{equation}\label{E:3pair-phi}
\begin{aligned}
\phi_{\calD_{III_n}^k}: \bbS^1\times \SL_2(\bbR) & \longrightarrow \Spnc(n)=\Hol(\calD_{III_n})^o \\
\left(e^{i\theta}, \begin{pmatrix}a & b \\ c & d \end{pmatrix} \right) & \mapsto
\begin{pmatrix} e^{i\theta}I_k & 0 & 0 & 0 \\ 0 & \frac{a+d+i(b-c)}{2}I_{n-k} & 0 & \frac{b+c+i(a-d)}{2} I_{n-k} \\
0& 0 & e^{-i\theta}I_k & 0 \\ 0 & \frac{b+c-i(a-d)}{2}I_{n-k} & 0 & \frac{a+d-i(b-c)}{2} I_{n-k} \\
\end{pmatrix}.
\end{aligned}
\end{equation}
In particular, the one-parameter subgroup of $\Spnc(n)$ associated to $\calD_{III_n}^k$ as in \eqref{E:wF} is given by
\begin{equation}\label{E:3wF}
\begin{aligned}
w_{\calD_{III_n}^k}: \G_m & \longrightarrow \Spnc(n)\\
t & \mapsto
\begin{pmatrix} I_k & 0 & 0 & 0 \\ 0 & \frac{t+t^{-1}}{2}I_{n-k} & 0 & \frac{i(t-t^{-1})}{2} I_{n-k} \\
0& 0 & I_k & 0 \\ 0 & \frac{-i(t-t^{-1})}{2}I_{n-k} & 0 & \frac{t+t^{-1}}{2} I_{n-k} \\
\end{pmatrix}.
\end{aligned}
\end{equation}
Using the above explicit expression of $w_{\calD_{III_n}^k}$ and Theorem \ref{T:5deco}, we can compute the Levi subgroup (together with its decomposition as in Theorem \ref{T:5deco}\eqref{T:5deco4}) and the unipotent radical of the normalizer $N(\calD_{III_n}^k)$ subgroup of $\calD_{III_n}^k$:
\begin{equation}\label{E:3Levi}
\begin{aligned}
 Z(w_{\calD_{III_n}^k})^o &=\left\{
\begin{pmatrix} A & 0 & B & 0 \\ 0 & \frac{E+(E^t)^{-1}}{2} & 0 & i \frac{E-(E^t)^{-1}}{2}  \\
C& 0 & D & 0 \\ 0 & -i \frac{E-(E^t)^{-1}}{2} & 0 & \frac{E+(E^t)^{-1}}{2} \\
\end{pmatrix}: \begin{pmatrix}A & B \\ C & D \end{pmatrix}\in \Spnc(k), E\in \GL_{n-k}(\bbR)\right\}\\
& = \Spnc(k)\times \GL_{n-k}(\bbR)\times \{1\}=G_h(\calD_{III_n}^k) \times G_l(\calD_{III_n}^k)\times M(\calD_{III_n}^k),
\end{aligned}
\end{equation}
\begin{equation}\label{E:3radi}
 W(\calD_{III_n}^k)=\left\{
\begin{pmatrix} I_r & F & 0 & -iF \\ -\ov{F}^t & I_{n-k}+i M & - i F^t & M  \\
0& i\ov{F} & I_k & \ov{F} \\ i\ov{F}^t & M  & -F^t & I_{n-k}-i M  \\
\end{pmatrix}:
\begin{aligned}
& F\in M_{k, n-k}(\bbC), M\in M_{n-k,n-k}(\bbR) \\
& \ov{F}^tF-F^t\ov{F}=i(M^t-M) \\
\end{aligned}\right\}.
\end{equation}
Moreover, the center $U(\calD_{III_n}^k)$ of $W(\calD_{III_n}^k)$ is equal to the set of all matrices of $W(\calD_{III_n}^k)$ such that $F=0$, and is therefore isomorphic to the abelian unipotent Lie group
underlying the vector space  $\Herm_{n-k}(\bbR)$.

The orbit of the point $o_{\calD_{III_n}^k}=f_{\calD_{III_n}^k}(1)=\begin{pmatrix}0 & 0 \\ 0 & I_{n-k}\end{pmatrix}$ under the natural action of the group $G_h(\calD_{III_n}^k)=\Spnc(k)$ is equal to $\calD_{III_n}^k$ and its stabilizer
subgroup is isomorphic to $K_h(\calD_{III_n}^k)=\Spnc(k)\cap U(n)=U(k)$. Therefore, $\calD_{III_n}^k$ is a bounded symmetric domain of type $III_k$ and it is diffeomorphic to
\begin{equation}\label{E:3-presF}
\calD_{III_n}^k\cong \frac{G_h(\calD_{III_n}^k)}{K_h(\calD_{III_n}^k)}=\frac{\Spnc(k)}{U(k)}.
\end{equation}

The action of $G_l(\calD_{III_n}^k)=\GL_{n-k}(\bbR)$ on $U(\calD_{III_n}^k)\cong \Herm_{n-k}(\bbR)$ is given by
$(E,M)\mapsto E M  E^t$. Under this action, the orbit of the point
$\Omega_{\calD_{III_n}^k}=\frac{1}{2} I_{n-k}\in \Herm_{n-k}(\bbR)$ is equal to the cone of positive definite real quadratic forms $\P_{n-k}(\bbR)$ of size $n-k$  and its stabilizer subgroup is isomorphic
$K_l(\calD_{III_n}^k)=\GL_{n-k}(\bbR)\cap U(n)=O(n-k)$. Therefore, $C(\calD_{III_n}^k)$ is equal to the symmetric cone
 $\P_{n-k}(\bbR)$ (see Theorem \ref{T:clascone}) and it is diffeomorphic to
\begin{equation}\label{E:3-C(F)}
C(\calD_{III_n}^k)\cong \frac{G_l(\calD_{III_n}^k)}{K_l(\calD_{III_n}^k)}=\frac{\GL_{n-k}(\bbR)}{O(n-k)}=\P_{n-k}(\bbR).
\end{equation}

\subsubsection{Type $IV_{n} (n\geq 3)$ }\label{SS:bounIV}

The \emph{standard boundary component} of $\calD_{IV}$ of rank $k$ (for $k=0,1$) are given by
\begin{equation}\label{E:st-boun4}
\begin{aligned}
& \calD_{IV_n}^0:=\{(-i,0,\ldots,0)^t \in \bbC^n\}, \\Ê
& \calD_{IV_n}^1:=\left\{\left(-i\frac{1+z}{1-z}, \frac{1+z}{1-z}, 0,\ldots, 0\right)^t \in \bbC^n: |z|<1\right\},
\end{aligned}
\end{equation}
see \cite[p. 355]{Wol3}.
The decomposition into factors (as in Theorem \ref{T:5deco}\eqref{T:5deco4})  of the Levi subgroup $L(\calD_{IV_n}^k)$ of the normalizer subgroup $N(\calD_{IV_n}^k)$ of
$\calD_{IV_n}^k$ is given by (see \cite[p. 117]{Sat})
\begin{equation}\label{E:4Levi}
L(\calD_{IV_n}^k)=G_h(\calD_{IV_n}^k) \cdot G_l(\calD_{IV_n}^k)\cdot M(\calD_{IV_n}^k)=
\begin{cases}
\SO(1,1)\cdot \GL_1(\bbR)\cdot \SO(n-2) & \text{ if }Êk=1, \\
\{1\}\cdot (\SO(n-1,1)\times \bbR^*)\cdot  \{1\} & \text{ if }Êk=0.
\end{cases}
\end{equation}
Therefore, $\calD_{IV_n}^k$ is a bounded symmetric domain of type $IV_1=I_{1,1}$ if $k=1$ and it is a point if $k=0$.  Moreover, the symmetric cone associated to $\calD_{IV_n}^k$ is equal to
 (see Theorem \ref{T:clascone})
\begin{equation}\label{E:4-C(F)}
C(\calD_{IV_n}^k)=
\begin{cases}
\P(1,0) & \text{ if }Êk=1, \\
\P(1,n-1) & \text{ if }Êk=0.
\end{cases}
\end{equation}

\subsubsection{Type $V$}\label{SS:bounV}

Using Remark \ref{R:idem}, Êdenote by $\calD_{V}^k$ (for $k=0,1$) the boundary component of $\calD_{V}$ of rank $k$ such that
$$o_{\calD_V^k}=
\begin{cases}
\begin{pmatrix} 0 \\ 1 \end{pmatrix}Ê& \text{Êif }Êk=1, \\
\begin{pmatrix}Ê1 \\ 1 \end{pmatrix}Ê& \text{Êif }Êk=0.
\end{cases}
$$
We call $\calD_V^k$ the \emph{standard boundary component} of $\calD_V$ of rank $k$.

The decomposition into factors (as in Theorem \ref{T:5deco}\eqref{T:5deco4})  of the Levi subgroup $L(\calD_{V}^k)$ of the normalizer subgroup $N(\calD_V^k)$ of
$\calD_V^k$ is given by (see \cite[p. 117]{Sat})
\begin{equation}\label{E:5Levi}
L(\calD_V^k)=G_h(\calD_{V}^k) \cdot G_l(\calD_{V}^k)\cdot M(\calD_{V}^k)=
\begin{cases}
\SU(5,1)\cdot \GL_1(\bbR)\cdot \{1\} & \text{ if }Êk=1, \\
\{1\}\cdot (\SO(7,1)\times \bbR^*)\cdot  \bbS^1 & \text{ if }Êk=0.
\end{cases}
\end{equation}
Therefore, $\calD_{V}^k$ is a bounded symmetric domain of type $I_{5,1}$ if $k=1$ and it is a point if $k=0$.  Moreover, the symmetric cone associated to $\calD_V^k$ is equal to
 (see Theorem \ref{T:clascone})
\begin{equation}\label{E:5-C(F)}
C(\calD_{V}^k)=
\begin{cases}
\P(1,0) & \text{ if }Êk=1, \\
\P(1,7) & \text{ if }Êk=0.
\end{cases}
\end{equation}

\subsubsection{Type $VI$}\label{SS:bounVI}

Using Remark \ref{R:idem}, Êdenote by $\calD_{VI}^k$ (for $k=0,1,2$) the boundary component of $\calD_{VI}$ of rank $k$ such that
$$o_{\calD_{VI}^k}=\begin{pmatrix}Ê0 & 0 \\ 0 & I_{3-k} \end{pmatrix}.$$
We call $\calD_{VI}^k$ the \emph{standard boundary component} of $\calD_{VI}$ of rank $k$.

The decomposition into factors (as in Theorem \ref{T:5deco}\eqref{T:5deco4})  of the Levi subgroup $L(\calD_{VI}^k)$ of the normalizer subgroup $N(\calD_{VI}^k)$ of
$\calD_{VI}^k$ is given by (see \cite[p. 117]{Sat})
\begin{equation}\label{E:6Levi}
L(\calD_{VI}^k)=G_h(\calD_{VI}^k) \cdot G_l(\calD_{VI}^k)\cdot M(\calD_{VI}^k)=
\begin{cases}
\SO(2,10) \cdot \GL_1(\bbR) \cdot \{1\}Ê& \text{Êif }Êk=2, \\
\SU(1,1)\cdot(\SO(9,1)\times \bbR^*) \cdot \{1\} & \text{ if }Êk=1, \\
\{1\}\cdot (E_6(-26) \times \bbR^*)\cdot  \{1\} & \text{ if }Êk=0.
\end{cases}
\end{equation}
Therefore, $\calD_{VI}^k$ is a bounded symmetric domain of type $IV_{10}$ if $k=2$, of type $I_{1,1}$ if $k=1$ and it is a point if $k=0$.
Moreover, the symmetric cone associated to $\calD_{VI}^k$ is equal to   (see Theorem \ref{T:clascone})
\begin{equation}\label{E:6-C(F)}
C(\calD_{V}^k)=
\begin{cases}
\P(1,0) & \text{ if }Êk=2, \\
\P(1,9) & \text{ if }Êk=1,\\
 \P_3(\bbO) & \text{Êif }Êk=0.
\end{cases}
\end{equation}

\section*{Acknowledgments}

These notes grew up from a PhD course held by the author at the University of Roma Tre  in spring  2013 and from a course held by the  author at the School ``Combinatorial Algebraic Geometry" held in Levico Terme (Trento, Italy) in June 2013. The author would like to thank the students and the colleagues that attended the above mentioned PhD course (Fabio Felici, Roberto Fringuelli, Alessandro Maria Masullo, Margarida Melo, Riane Melo, Paola Supino, Valerio Talamanca) for their patience, encouragement and interest in the material presented.  Moreover, the author would like to thank the organizers of the above mentioned School (Giorgio Ottaviani, Sandra Di Rocco, Bernd Sturmfels) for the invitation to give a course as well as all the participants to the School for their interest and feedbacks.
Many thanks are due to Jan Draisma for reading a preliminary version of this manuscript and to Radu Laza for suggesting some bibliographical references.

The author is a member of the research center CMUC (University of Coimbra) and he was supported by  the FCT project \textit{Espa\c cos de Moduli em Geometria Alg\'ebrica} (PTDC/MAT/111332/2009) and by the MIUR project  \textit{Spazi di moduli e applicazioni} (FIRB 2012).

\end{document}